\documentclass{amsart}
\usepackage{amssymb}
\usepackage{amsmath}
\usepackage{amsfonts}
\usepackage{diagrams}

\newtheorem{theorem}{Theorem}[section]

\newtheorem{conjecture}[theorem]{Conjecture}
\newtheorem{corollary}[theorem]{Corollary}

\newtheorem{definition}[theorem]{Definition}
\newtheorem{example}[theorem]{Example}

\newtheorem{notation}[theorem]{Notation}

\newtheorem{proposition}[theorem]{Proposition}
\newtheorem{remark}[theorem]{Remark}

\input{tcilatex}
\DeclareMathOperator{\coker}{coker}

\begin{document}
\title[Cosheaves]{Cosheaves}
\author{Andrei V. Prasolov}
\address{Institute of Mathematics and Statistics\\
The University of Troms\o\ - The Arctic University of Norway\\
N-9037 Troms\o , Norway}
\email{andrei.prasolov@uit.no}
\urladdr{http://serre.mat-stat.uit.no/ansatte/andrei/Welcome.html}
\date{}

\begin{abstract}
The categories $\mathbf{pCS}\left( X,\mathbf{Pro}\left( k\right) \right) $
of precosheaves and $\mathbf{CS}\left( X,\mathbf{Pro}\left( k\right) \right) 
$ of cosheaves on a small Grothendieck site $X$, with values in the category 
$\mathbf{Pro}\left( k\right) $ of pro-$k$-modules, are constructed. It is
proved that $\mathbf{pCS}\left( X,\mathbf{Pro}\left( k\right) \right) $
satisfies the AB4 and AB5* axioms, while $\mathbf{CS}\left( X,\mathbf{Pro}%
\left( k\right) \right) $ satisfies AB3 and AB5*. Homology theories for
cosheaves and precosheaves, based on quasi-projective resolutions, are
constructed and investigated.
\end{abstract}

\maketitle
\tableofcontents

\setcounter{section}{-1}

\section{Introduction}

A \emph{presheaf} (\emph{precosheaf}) on a topological space $X$ with values
in a category $\mathbf{K}$ is just a contravariant (covariant) functor from
the category of open subsets of $X$ to $\mathbf{K}$, while a \emph{sheaf} (%
\emph{cosheaf}) is such a functor satisfying some extra conditions. The
category of (pre)cosheaves with values in $\mathbf{K}$ is dual to the
category of (pre)sheaves with values in the dual category $\mathbf{K}^{op}$.

While the theory of sheaves is well developed, and is covered by a plenty of
publications, the theory of cosheaves is more poorly represented. The main
reason for this is that \emph{cofiltered limits} are \emph{not} exact in the
\textquotedblleft usual\textquotedblright\ categories like sets, abelian
groups, rings, or modules. On the contrary, \emph{filtered colimits} are
exact in the above categories, which allows to construct rather rich
theories of sheaves with values in \textquotedblleft
usual\textquotedblright\ categories. To sum up, the \textquotedblleft
usual\textquotedblright\ categories $\mathbf{K}$ are badly suited for
cosheaf theory. Dually, the categories $\mathbf{K}^{op}$ are badly suited
for sheaf theory.

The first step in building a suitable theory of cosheaves would be
constructing a \emph{cosheaf }$\mathcal{A}_{\#}$ \emph{associated with a
precosheaf} $\mathcal{A}$ (simply: \emph{cosheafification} of $\mathcal{A}$%
), as a right adjoint%
\begin{equation*}
\left( {}\right) _{\#}:Precosheaves\longrightarrow Cosheaves
\end{equation*}%
to the inclusion%
\begin{equation*}
\iota :Cosheaves\hookrightarrow Precosheaves.
\end{equation*}%
As is shown in \cite[Theorem 3.1]%
{Prasolov-Cosheafification-2016-zbMATH06684178}, it is possible in many
situations, namely for precosheaves with values in an arbitrary \emph{%
locally presentable }\cite[Chapter 1]%
{Adamek-Rosicky-1994-Locally-presentable-categories-MR1294136}\emph{\
category} (or a dual to such a category). See also Theorem \ref%
{Th-Cosheafification} in this paper.

However, our purpose is to prepare a foundation for \emph{homology} theory
of cosheaves (see Theorems \ref{Th-Precosheaf-homology}, \ref%
{Th-Cosheaf-homology}, and Conjecture \ref{Conj-Satellites-H}). In future
papers, we plan to develop also the \emph{nonabelian} homology theory (in
other words, the \emph{homotopy} theory) of (pre)cosheaves (see Conjectures %
\ref{Conj-Satellites-Pi}, \ref{Conj-Satellites-Locally-finite}, and \ref%
{Conj-Etale} below).

Therefore, we need a more or less \emph{explicit} construction. Moreover, we
need a construction satisfying \emph{good exactness} properties. As is shown
in \cite{Prasolov-Cosheafification-2016-zbMATH06684178}, the most suitable
categories for these purposes are the categories of (pre)cosheaves with
values in the pro-category $\mathbf{Pro}\left( \mathbf{K}\right) $
(Definition \ref{Def-Pro-category}), where $\mathbf{K}$ is a cocomplete
(Definition \ref{Def-(co)complete}) category. In \cite[Theorem 3.11]%
{Prasolov-Cosheafification-2016-zbMATH06684178}, connections with \textbf{%
shape theory} have been established: it was shown that the cosheafification $%
G_{\#}$ of the \emph{constant} precosheaf $G^{const}$, $G\in \mathbf{K}$, is
isomorphic to $G\otimes _{\mathbf{Set}}pro$-$\pi _{0}$, where $pro$-$\pi
_{0} $ is the pro-homotopy from Definition \ref{Def-Pro-homotopy-groups}
(for the pairing $\otimes _{\mathbf{Set}}$ see Definition \ref%
{Def-Pairings-functors}(\ref{Def-Pairings-functors-Pro(Set)})). If $\mathbf{%
K=Mod}\left( k\right) $ is the category of pro-modules over a commutative
ring $k$, the cosheafification $G_{\#}$ becomes the pro-homology (Definition %
\ref{Def-Pro-homology-groups}):%
\begin{equation*}
G_{\#}%
%TCIMACRO{\TeXButton{ISO}{\simeq}}%
%BeginExpansion
\simeq%
%EndExpansion
\left( U\longmapsto pro\text{-}H_{0}\left( U,G\right) \right) .
\end{equation*}

\begin{remark}
An interesting attempt is made in \cite{Schneiders-MR885939} where the
author sketches a cosheaf theory on topological spaces with values in a
category $\mathbf{L}$, dual to an \textquotedblleft
elementary\textquotedblright\ category $\mathbf{L}^{op}$. He proposes a
candidate for such a category. Let $\alpha <\beta $ be two inaccessible
cardinals. Then $\mathbf{L}$ is the category $\mathbf{Pro}_{\beta }\left( 
\mathbf{Ab}_{\alpha }\right) $ of abelian pro-groups $\left( G_{j}\right)
_{j\in \mathbf{J}}$ such that $card\left( G_{j}\right) <\alpha $ and $%
card\left( Mor\left( \mathbf{J}\right) \right) <\beta $. However, our\
pro-category $\mathbf{Pro}\left( \mathbf{K}\right) $ cannot be used in the
cosheaf theory from \cite{Schneiders-MR885939} because the category $\left( 
\mathbf{Pro}\left( \mathbf{K}\right) \right) ^{op}$ is \textbf{not}
elementary.
\end{remark}

The main results of this paper are establishing the most important
properties of precosheaves (Theorem \ref{Th-Properties-precosheaves}) and
cosheaves (Theorem \ref{Th-Properties-cosheaves}), as well as constructing
homology theory for precosheaves (Theorem \ref{Th-Precosheaf-homology}) and
cosheaves (Theorem \ref{Th-Cosheaf-homology}). We construct the \emph{abelian%
} homology theory of (pre)cosheaves with values in the category%
\begin{equation*}
\mathbf{Pro}\left( k\right) =\mathbf{Pro}\left( \mathbf{Mod}\left( k\right)
\right)
\end{equation*}%
(Notation \ref{Not-Pro(k)}), where $k$ is a \emph{quasi-noetherian}
(Definition \ref{Def-quasi-noetherian}) commutative ring. Due to Proposition %
\ref{Prop-Noetherian-Quasi-noetherian}, the class of such rings is
sufficiently large, and our construction includes, e.g., (pre)cosheaves with
values in%
\begin{equation*}
\mathbf{Pro}\left( \mathbf{Ab}\right) 
%TCIMACRO{\TeXButton{ISO}{\simeq}}%
%BeginExpansion
\simeq%
%EndExpansion
\mathbf{Pro}\left( \mathbf{Mod}\left( \mathbb{Z}\right) \right) =\mathbf{Pro}%
\left( \mathbb{Z}\right) .
\end{equation*}

\begin{remark}
A cosheaf theory with values in the category $\mathbf{Pro}\left( k\right) $
on topological spaces was sketched in \cite{Sugiki-2001-33}. Definition
2.2.7 of a cosheaf on a topological space $X$ in \cite{Sugiki-2001-33} is
dual to our definition of a cosheaf on the corresponding site $OPEN\left(
X\right) $, see Example \ref{Site-TOP} and Remark \ref%
{Denote-standard-site-simply}. Theorem 2.2.8 in \cite{Sugiki-2001-33} states
that the cosheafification exists. However, no proof of that theorem is
given, and no explicit construction of such cosheafification is provided.

Moreover, in \cite[Definition 4.1.3]{Sugiki-2001-33} the author introduces
the notion of \textbf{c-injective} cosheaves which seem to be dual to our 
\textbf{quasi-projective} cosheaves, and claims in \cite[Theorem 4.1.7]%
{Sugiki-2001-33} that c-injective cosheaves form a cogenerating subcategory
in the category of all cosheaves. That statement seems to be dual to our
Theorem \ref{Th-Cosheaf-homology}(\ref{Th-Cosheaf-homology-surjection}).
However, the proof is only sketched, and is based on several statements
given without proofs. Moreover, the cosheaf homology in \cite{Sugiki-2001-33}
is constructed (sketchy!) only for topological spaces (for the site $%
OPEN\left( X\right) $, see Example \ref{Site-TOP} and Remark \ref%
{Denote-standard-site-simply}). In this paper, on the contrary, we construct
the cosheaf homology theory for arbitrary small sites.
\end{remark}

\begin{conjecture}
\label{Conj-Satellites-H}~

\begin{enumerate}
\item \label{Conj-Satellites-H-paracompact}On the standard site $OPEN\left(
X\right) $ (Example \ref{Site-TOP}), the left satellites of $H_{0}$ are
naturally isomorphic to the \textbf{pro-homology} (Definition \ref%
{Def-Pro-homology-groups}):%
\begin{equation*}
H_{n}\left( X,pro\text{-}H_{0}\left( \bullet ,A\right) \right) =H_{n}\left(
X,A_{\#}\right) 
%TCIMACRO{\TeXButton{assigned}{{:=}}}%
%BeginExpansion
{:=}%
%EndExpansion
L_{n}H_{0}\left( X,A_{\#}\right) 
%TCIMACRO{\TeXButton{ISO}{\simeq}}%
%BeginExpansion
\simeq%
%EndExpansion
pro\text{-}H_{n}\left( X,A\right) ,
\end{equation*}%
provided $X$ is \textbf{Hausdorff paracompact}.

\item \label{Conj-Satellites-H-paracompact-normal}The above isomorphisms
exist also for the site $NORM\left( X\right) $ (Example \ref{Site-NORM}) and
are valid for \textbf{all} topological spaces.
\end{enumerate}
\end{conjecture}

Example \ref{Ex-Convergent-sequence} illustrates the conjecture.

\begin{conjecture}
\label{Conj-Satellites-Pi}~

\begin{enumerate}
\item \label{Conj-Satellites-Pi-paracompact}On the standard site $OPEN\left(
X\right) $, the \textbf{nonabelian} left satellites of $H_{0}$ are naturally
isomorphic to the pro-homotopy (Definition \ref{Def-Pro-homotopy-groups}):%
\begin{eqnarray*}
H_{n}\left( X,S_{\#}\right) &=&H_{n}\left( X,S\times pro\text{-}\pi
_{0}\right) 
%TCIMACRO{\TeXButton{assigned}{{:=}}}%
%BeginExpansion
{:=}%
%EndExpansion
L_{n}H_{0}\left( X,S_{\#}\right) 
%TCIMACRO{\TeXButton{ISO}{\simeq}}%
%BeginExpansion
\simeq%
%EndExpansion
S\times pro\text{-}\pi _{n}\left( X\right) , \\
H_{n}\left( X,\left( \mathbf{pt}\right) _{\#}\right) &=&H_{n}\left( X,pro%
\text{-}\pi _{0}\right) 
%TCIMACRO{\TeXButton{assigned}{{:=}}}%
%BeginExpansion
{:=}%
%EndExpansion
L_{n}H_{0}\left( X,\left( \mathbf{pt}\right) _{\#}\right) 
%TCIMACRO{\TeXButton{ISO}{\simeq}}%
%BeginExpansion
\simeq%
%EndExpansion
pro\text{-}\pi _{n}\left( X\right) ,
\end{eqnarray*}%
provided $X$ is Hausdorff paracompact.

\item \label{Conj-Satellites-Pi-normal}The above isomorphisms exist also for
the site $NORM\left( X\right) $ and are valid for \textbf{all} topological
spaces.
\end{enumerate}
\end{conjecture}

For general topological spaces, however, one could not expect that cosheaf
homology $H_{n}\left( X,G_{\#}\right) $ coincides with shape pro-homology $%
pro$-$H_{n}\left( X,G\right) $ (unless $n=0$, see Theorem \ref%
{Th-Cosheafification} and \cite[Theorem 3.11]%
{Prasolov-Cosheafification-2016-zbMATH06684178}). The thing is that general
spaces may lack \textquotedblleft good\textquotedblright\ polyhedral
expansions (Definition \ref{Def-HTOP-extension}). See Remark \ref%
{Rem-Shape-locally-finite} and Conjecture \ref%
{Conj-Satellites-Locally-finite}.

\begin{conjecture}
\label{Conj-Satellites-Locally-finite}Let $X$ be a (pointed) finite (or even
locally finite) topological space. Then:

\begin{enumerate}
\item The left satellites of $H_{0}$ are naturally isomorphic to the
singular homology:%
\begin{equation*}
H_{n}\left( X,G_{\#}\right) 
%TCIMACRO{\TeXButton{assigned}{{:=}}}%
%BeginExpansion
{:=}%
%EndExpansion
L_{n}H_{0}\left( X,G_{\#}\right) 
%TCIMACRO{\TeXButton{ISO}{\simeq}}%
%BeginExpansion
\simeq%
%EndExpansion
H_{n}^{sing}\left( X,G\right) .
\end{equation*}

\item The \textbf{nonabelian} left satellites of $H_{0}$ are naturally
isomorphic to the homotopy groups:%
\begin{eqnarray*}
H_{n}\left( X,S_{\#}\right) &=&H_{n}\left( X,S\times \pi _{0}\right) 
%TCIMACRO{\TeXButton{assigned}{{:=}}}%
%BeginExpansion
{:=}%
%EndExpansion
L_{n}H_{0}\left( X,S_{\#}\right) 
%TCIMACRO{\TeXButton{ISO}{\simeq}}%
%BeginExpansion
\simeq%
%EndExpansion
S\times \pi _{n}\left( X\right) , \\
H_{n}\left( X,\left( \mathbf{pt}\right) _{\#}\right) &=&H_{n}\left( X,\pi
_{0}\right) 
%TCIMACRO{\TeXButton{assigned}{{:=}}}%
%BeginExpansion
{:=}%
%EndExpansion
L_{n}H_{0}\left( X,\left( \mathbf{pt}\right) _{\#}\right) 
%TCIMACRO{\TeXButton{ISO}{\simeq}}%
%BeginExpansion
\simeq%
%EndExpansion
\pi _{n}\left( X\right) .
\end{eqnarray*}
\end{enumerate}
\end{conjecture}

Example \ref{Ex-Finite-circle} illustrates the conjecture.

On the contrary, the pro-homology and pro-homotopy of such spaces are rather
trivial:

\begin{remark}
\label{Rem-Shape-locally-finite}If $X$ is a locally finite (pointed)
topological space, then:%
\begin{eqnarray*}
&&pro\text{-}H_{n}\left( X,G\right) 
%TCIMACRO{\TeXButton{ISO}{\simeq}}%
%BeginExpansion
\simeq%
%EndExpansion
H_{n}\left( \left( \pi _{0}\left( X\right) \right) ^{\delta },G\right) , \\
&&pro\text{-}\pi _{n}\left( X\right) 
%TCIMACRO{\TeXButton{ISO}{\simeq}}%
%BeginExpansion
\simeq%
%EndExpansion
\pi _{n}\left( \left( \pi _{0}\left( X\right) \right) ^{\delta }\right) ,
\end{eqnarray*}%
where $\left( \pi _{0}\left( X\right) \right) ^{\delta }$ is the set of
connected components of $X$, supplied with the discrete topology. Indeed, it
is easy to check that the natural continuous projection%
\begin{equation*}
X\longrightarrow \left( \pi _{0}\left( X\right) \right) ^{\delta }
\end{equation*}%
is a polyhedral expansion (Definition \ref{Def-HTOP-extension}).
\end{remark}

Other possible applications could be in \'{e}tale homotopy theory \cite%
{Artin-Mazur-MR883959} as is summarized in the following

\begin{conjecture}
\label{Conj-Etale}Let $X^{et}$ be the site from Example \ref{Ex-Site-ETALE}.

\begin{enumerate}
\item The left satellites of $H_{0}$ are naturally isomorphic to the \'{e}%
tale pro-homology:%
\begin{equation*}
H_{n}\left( X^{et},A_{\#}\right) 
%TCIMACRO{\TeXButton{assigned}{{:=}}}%
%BeginExpansion
{:=}%
%EndExpansion
L_{n}H_{0}\left( X^{et},A_{\#}\right) 
%TCIMACRO{\TeXButton{ISO}{\simeq}}%
%BeginExpansion
\simeq%
%EndExpansion
H_{n}^{et}\left( X,A\right) .
\end{equation*}

\item The nonabelian left satellites of $H_{0}$ are naturally isomorphic to
the \'{e}tale pro-homotopy:%
\begin{equation*}
H_{n}\left( X^{et},\left( \mathbf{pt}\right) _{\#}\right) 
%TCIMACRO{\TeXButton{ISO}{\simeq}}%
%BeginExpansion
\simeq%
%EndExpansion
H_{n}\left( X^{et},\pi _{0}^{et}\right) 
%TCIMACRO{\TeXButton{assigned}{{:=}}}%
%BeginExpansion
{:=}%
%EndExpansion
L_{n}H_{0}\left( X^{et},\left( \mathbf{pt}\right) _{\#}\right) 
%TCIMACRO{\TeXButton{ISO}{\simeq}}%
%BeginExpansion
\simeq%
%EndExpansion
\pi _{n}^{et}\left( X\right) .
\end{equation*}
\end{enumerate}
\end{conjecture}

\section{Preliminaries}

Throughout this paper, we will constantly use the pairings%
\begin{eqnarray*}
\left\langle \bullet ,\bullet \right\rangle &:&\mathbf{Pro}\left( k\right)
^{op}\times \mathbf{Mod}\left( k\right) \longrightarrow \mathbf{Mod}\left(
k\right) , \\
\left\langle \bullet ,\bullet \right\rangle &:&\mathbf{pCS}\left( \mathbf{D},%
\mathbf{Pro}\left( k\right) \right) ^{op}\times \mathbf{Mod}\left( k\right)
\longrightarrow \mathbf{pS}\left( \mathbf{D},\mathbf{Mod}\left( k\right)
\right)
\end{eqnarray*}

from Definition \ref{Def-Pairings-functors}(\ref%
{Def-Pairings-pro-modules-Hom(Pro-k)}, \ref{Def-Pairings-functors-Hom(Pro-k)}%
).

\subsection{Quasi-noetherian rings}

Let $k$ be a commutative ring. From now on, $k$ is assumed to be \textbf{%
quasi-noetherian} (Definition \ref{Def-quasi-noetherian}), e.g. \textbf{%
noetherian} (see Proposition \ref{Prop-Noetherian-Quasi-noetherian}).

\begin{notation}
\label{Not-Pro(k)}Let $k$ be a commutative ring. Then $\mathbf{Pro}\left(
k\right) =\mathbf{Pro}\left( \mathbf{Mod}\left( k\right) \right) $ is the
category of pro-objects (Definition \ref{Def-Pro-C}) in the category $%
\mathbf{Mod}\left( k\right) $ of $k$-modules.
\end{notation}

\begin{remark}
Since any noetherian ring (e.g. $\mathbb{Z}$) is quasi-noetherian, our
considerations cover a large family of pro-categories like%
\begin{equation*}
\mathbf{Pro}\left( \mathbf{Ab}\right) 
%TCIMACRO{\TeXButton{ISO}{\simeq}}%
%BeginExpansion
\simeq%
%EndExpansion
\mathbf{Pro}\left( \mathbb{Z}\right) ,
\end{equation*}%
$\mathbf{Pro}\left( k\right) $ where $k$ is a field, $\mathbf{Pro}\left(
R\right) $ where $R$ is a finitely generated commutative algebra over a
noetherian ring, etc.
\end{remark}

\subsection{(Pre)cosheaves}

In this paper, we will consider (pre)cosheaves with values in $\mathbf{Pro}%
\left( \mathbf{K}\right) $ ($\mathbf{K}$ is a cocomplete category, see
Definition \ref{Def-(co)complete}) or $\mathbf{Pro}\left( k\right) $, and
(pre)sheaves with values in $\mathbf{L}$ ($\mathbf{L}$ is a complete
category, see Definition \ref{Def-(co)complete}) or $\mathbf{Mod}\left(
k\right) $. \textbf{Pre}(co)sheaves can be defined on small sites (in
particular) or on small categories (in general). Many of our constructions
and statements are also valid for those generalized pre(co)sheaves.

Let $X=\left( \mathbf{C}_{X}\mathbf{,}Cov\left( X\right) \right) $ be a
small site (Definition \ref{Def-Site}), and let $\mathbf{D}$ and $\mathbf{K}$
be categories. Assume that $\mathbf{D}$ is small, and $\mathbf{K}$ is
cocomplete (Definition \ref{Def-(co)complete}). Remind Definition \ref%
{Def-Comma-U} for $\mathbf{C}_{U}$ and Definition \ref{Def-Comma-R} for $%
\mathbf{C}_{R}$.

\begin{definition}
\label{Def-(Pre)cosheaves}~

\begin{enumerate}
\item A \textbf{precosheaf} $\mathcal{A}$ on $\mathbf{D}$ with values in $%
\mathbf{K}$ is a functor $\mathcal{A}:\mathbf{D}\rightarrow \mathbf{K}$.

\item A \textbf{precosheaf} $\mathcal{A}$ on $X$ with values in $\mathbf{K}$
is a functor $\mathcal{A}:\mathbf{C}_{X}\rightarrow \mathbf{K}$.

\item A precosheaf $\mathcal{A}$ on $X$ is \textbf{coseparated} provided%
\begin{equation*}
\mathcal{A}\otimes _{\mathbf{Set}^{\mathbf{C}_{X}}}R%
%TCIMACRO{\TeXButton{ISO}{\simeq}}%
%BeginExpansion
\simeq%
%EndExpansion
\underset{\left( V\rightarrow U\right) \in \mathbf{C}_{R}}{\underrightarrow{%
\lim }}\mathcal{A}\left( V\right) \longrightarrow \mathcal{A}\otimes _{%
\mathbf{Set}^{\mathbf{C}_{X}}}h_{U}%
%TCIMACRO{\TeXButton{ISO}{\simeq}}%
%BeginExpansion
\simeq%
%EndExpansion
\mathcal{A}\left( U\right)
\end{equation*}%
is an epimorphism for any $U\in \mathbf{C}_{X}$ and for any covering sieve
(Definition \ref{Def-Site}) $R$ over $U$. The pairing $\otimes _{\mathbf{Set}%
^{\mathbf{C}_{X}}}$ is introduced in Definition \ref{Def-Pairings-functors}(%
\ref{Def-Pairings-functors-Set-X}).

\item A precosheaf $\mathcal{A}$ on $X$ is a \textbf{cosheaf} provided%
\begin{equation*}
\mathcal{A}\otimes _{\mathbf{Set}^{\mathbf{C}_{X}}}R%
%TCIMACRO{\TeXButton{ISO}{\simeq}}%
%BeginExpansion
\simeq%
%EndExpansion
\underset{\left( V\rightarrow U\right) \in \mathbf{C}_{R}}{\underrightarrow{%
\lim }}\mathcal{A}\left( V\right) \longrightarrow \mathcal{A}\otimes _{%
\mathbf{Set}^{\mathbf{C}_{X}}}h_{U}%
%TCIMACRO{\TeXButton{ISO}{\simeq}}%
%BeginExpansion
\simeq%
%EndExpansion
\mathcal{A}\left( U\right)
\end{equation*}%
is an isomorphism for any $U\in \mathbf{C}_{X}$ and for any covering sieve $%
R $ over $U$.
\end{enumerate}
\end{definition}

\begin{remark}
The isomorphisms%
\begin{equation*}
\mathcal{A}\otimes _{\mathbf{Set}^{\mathbf{C}_{X}}}R%
%TCIMACRO{\TeXButton{ISO}{\simeq}}%
%BeginExpansion
\simeq%
%EndExpansion
\underset{\left( V\rightarrow U\right) \in \mathbf{C}_{R}}{\underrightarrow{%
\lim }}\mathcal{A}\left( V\right)
\end{equation*}%
and%
\begin{equation*}
\mathcal{A}\otimes _{\mathbf{Set}^{\mathbf{C}_{X}}}h_{U}%
%TCIMACRO{\TeXButton{ISO}{\simeq}}%
%BeginExpansion
\simeq%
%EndExpansion
\mathcal{A}\left( U\right)
\end{equation*}%
follow from Proposition \ref{Prop-A-ten-Set-CX-R}, because the
comma-category $\mathbf{C}_{U}%
%TCIMACRO{\TeXButton{ISO}{\simeq}}%
%BeginExpansion
\simeq%
%EndExpansion
\mathbf{C}_{h_{U}}$ (Definition \ref{Def-Comma-U} and Remark \ref%
{Rem-CU-equivalent-ChU}) has a terminal object $\left( U,\mathbf{1}%
_{U}\right) $.
\end{remark}

\begin{notation}
\label{Not-(Pre)cosheaves-categories}Denote by $\mathbf{CS}\left( X,\mathbf{K%
}\right) $ the category of cosheaves, and by $\mathbf{pCS}\left( X,\mathbf{K}%
\right) $ (respectively $\mathbf{pCS}\left( \mathbf{D},\mathbf{K}\right) $)
the category of precosheaves on $X$ (respectively on $\mathbf{D}$) with
values in $\mathbf{K}$.
\end{notation}

\begin{remark}
\label{Rem-Compare-cosheaves-with-sheaves}Compare to Definition \ref%
{Def-(Pre)sheaves} and Notation \ref{Not-(Pre)sheaves-categories} for
(pre)sheaves.
\end{remark}

\begin{definition}
\label{Def-Plus-construction}~

\begin{enumerate}
\item Assume that $\mathbf{K}$ is cocomplete. Given a precosheaf $\mathcal{A}%
\in \mathbf{pCS}\left( X,\mathbf{Pro}\left( \mathbf{K}\right) \right) $, let 
\begin{equation*}
\mathcal{A}_{+}\left( U\right) 
%TCIMACRO{\TeXButton{assigned}{{:=}}}%
%BeginExpansion
{:=}%
%EndExpansion
\left[ U\longmapsto \check{H}_{0}\left( U,\mathcal{A}\right) \right] ,
\end{equation*}%
alternatively 
\begin{equation*}
\mathcal{A}_{+}\left( U\right) 
%TCIMACRO{\TeXButton{assigned}{{:=}}}%
%BeginExpansion
{:=}%
%EndExpansion
\left[ U\longmapsto ~^{Roos}\check{H}_{0}\left( U,\mathcal{A}\right) \right]
,
\end{equation*}%
(see Definition \ref{Def-Cech-homology} (\ref{Def-Cech-homology-Cech-Hn(R)}, %
\ref{Def-Cech-homology-Cech-Hn(Ui)}) and Proposition \ref%
{Prop-Two-Cech-equivalent}). $\mathcal{A}_{+}$ is clearly a precosheaf, and
we have natural morphisms%
\begin{eqnarray*}
\lambda _{+}\left( \mathcal{A}\right) &:&\mathcal{A}_{+}\longrightarrow 
\mathcal{A}, \\
\lambda _{++}\left( \mathcal{A}\right) &=&\lambda _{+}\left( \mathcal{A}%
\right) \circ \lambda _{+}\left( \mathcal{A}_{+}\right) :\mathcal{A}%
_{++}\longrightarrow \mathcal{A}.
\end{eqnarray*}

\item Assume that $\mathbf{K}$ is complete. Given a presheaf $\mathcal{B}\in 
\mathbf{pS}\left( X,\mathbf{K}\right) $, let 
\begin{equation*}
\mathcal{B}^{+}\left( U\right) 
%TCIMACRO{\TeXButton{assigned}{{:=}}}%
%BeginExpansion
{:=}%
%EndExpansion
\left[ U\longmapsto \check{H}^{0}\left( U,\mathcal{A}\right) \right]
\end{equation*}%
(see Definition \ref{Def-Cech-homology} (\ref{Def-Cech-homology-Cech-Hn(R)}, %
\ref{Def-Cech-homology-Cech-Hn(Ui)})). $\mathcal{B}^{+}$ is clearly a
presheaf, and we have natural morphisms%
\begin{eqnarray*}
\lambda ^{+}\left( \mathcal{B}\right) &:&\mathcal{B}\longrightarrow \mathcal{%
B}^{+}, \\
\lambda ^{++}\left( \mathcal{B}\right) &=&\lambda ^{+}\left( \mathcal{B}%
^{+}\right) \circ \lambda ^{+}\left( \mathcal{B}\right) :\mathcal{B}%
\longrightarrow \mathcal{B}^{++}.
\end{eqnarray*}%
It is well-known that $\mathcal{B}^{++}$ is a sheaf. Apply, e.g., \cite[%
Theorem 3.1(3)]{Prasolov-Cosheafification-2016-zbMATH06684178} to $\mathbf{K}%
^{op}$.
\end{enumerate}
\end{definition}

The following theorem has been partially proved in \cite%
{Prasolov-Cosheafification-2016-zbMATH06684178}:

\begin{theorem}
\label{Th-Cosheafification}Assume that $\mathbf{K}$ is cocomplete. In (\ref%
{Th-Cosheafification-plus-right-exact}-\ref%
{Th-Cosheafification-plusplus-exact}) below assume in addition that $\mathbf{%
K}$ admits finite limits. Let%
\begin{eqnarray*}
\mathcal{A} &\in &\mathbf{pCS}\left( X,\mathbf{Pro}\left( \mathbf{K}\right)
\right) , \\
\mathcal{B} &\in &\mathbf{pCS}\left( X,\mathbf{K}\right) \subseteq \mathbf{%
pCS}\left( X,\mathbf{Pro}\left( \mathbf{K}\right) \right) , \\
\mathcal{C} &\in &\mathbf{pCS}\left( X,\mathbf{Pro}\left( k\right) \right) .
\end{eqnarray*}%
Then:

\begin{enumerate}
\item \label{Th-Cosheafification-Pro(K)-K-(co)complete-cosheaf}\label%
{Th-Main-Pro(K)-K-(co)complete-cosheaf}$\mathcal{B}$ is coseparated (a
cosheaf) iff it is coseparated (a cosheaf) when considered as a precosheaf
with values in $\mathbf{Pro}\left( \mathbf{K}\right) $.

\item \label{Th-Cosheafification-coreflective}The full subcategory of
cosheaves%
\begin{equation*}
\mathbf{CS}\left( X,\mathbf{Pro}\left( \mathbf{K}\right) \right) \subseteq 
\mathbf{pCS}\left( X,\mathbf{Pro}\left( \mathbf{K}\right) \right)
\end{equation*}%
is coreflective (Definition \ref{Def-(co)reflective}), and the coreflection 
\begin{equation*}
\mathbf{pCS}\left( X,\mathbf{Pro}\left( \mathbf{K}\right) \right)
\longrightarrow \mathbf{CS}\left( X,\mathbf{Pro}\left( \mathbf{K}\right)
\right)
\end{equation*}%
is given by%
\begin{equation*}
\mathcal{A}\longmapsto \mathcal{A}_{\#}%
%TCIMACRO{\TeXButton{assigned}{{:=}}}%
%BeginExpansion
{:=}%
%EndExpansion
\mathcal{A}_{++}.
\end{equation*}

\item \label{Th-Cosheafification-plus-right-exact}The functor%
\begin{equation*}
\left( {}\right) _{+}:\mathbf{pCS}\left( X,\mathbf{Pro}\left( \mathbf{K}%
\right) \right) \longrightarrow \mathbf{pCS}\left( X,\mathbf{Pro}\left( 
\mathbf{K}\right) \right)
\end{equation*}%
is \textbf{right exact} (Definition \ref{Def-exact-functors}).

\item \label{Th-Cosheafification-plusplus-exact}The functor%
\begin{equation*}
\left( {}\right) _{\#}=\left( {}\right) _{++}:\mathbf{pCS}\left( X,\mathbf{%
Pro}\left( \mathbf{K}\right) \right) \longrightarrow \mathbf{CS}\left( X,%
\mathbf{Pro}\left( \mathbf{K}\right) \right)
\end{equation*}%
is \textbf{exact} (Definition \ref{Def-exact-functors}).

\item \label{Th-Cosheafification-(A,T)-separated}$\mathcal{C}$ is \textbf{co}%
separated iff the presheaf $\left\langle \mathcal{C},T\right\rangle $ (see
Definition \ref{Def-Pairings-functors}(\ref{Def-Pairings-functors-Hom(Pro-k)}%
) is \textbf{separated} (Definition \ref{Def-(Pre)sheaves}) for any
injective $T\in \mathbf{Mod}\left( k\right) $.

\item \label{Th-Cosheafification-(A,T)-sheaf}$\mathcal{C}$ is a \textbf{co}%
sheaf iff the presheaf $\left\langle \mathcal{C},T\right\rangle $ is a 
\textbf{sheaf} (Definition \ref{Def-(Pre)sheaves}) for any injective $T\in 
\mathbf{Mod}\left( k\right) $.

\item \label{Th-Cosheafification-(A,T)-plus}%
\begin{eqnarray*}
&&\left\langle \mathcal{C}_{+},T\right\rangle 
%TCIMACRO{\TeXButton{ISO}{\simeq}}%
%BeginExpansion
\simeq%
%EndExpansion
\left\langle \mathcal{C},T\right\rangle ^{+}, \\
&&\left\langle \mathcal{C}_{\#},T\right\rangle 
%TCIMACRO{\TeXButton{ISO}{\simeq}}%
%BeginExpansion
\simeq%
%EndExpansion
\left\langle \mathcal{C},T\right\rangle ^{\#},
\end{eqnarray*}%
naturally in $\mathcal{C}$ and $T$, for \textbf{any} (not necessarily
injective) $T\in \mathbf{Mod}\left( k\right) $.
\end{enumerate}
\end{theorem}

\begin{proof}
(\textbf{\ref{Th-Cosheafification-Pro(K)-K-(co)complete-cosheaf}}, \textbf{%
\ref{Th-Cosheafification-coreflective}}) See \cite[Theorem 3.1(4)]%
{Prasolov-Cosheafification-2016-zbMATH06684178}.

(\textbf{\ref{Th-Cosheafification-plus-right-exact}}) Let $U\in \mathbf{C}%
_{X}$, and let $R\subseteq h_{U}$ be a sieve. Then the functor%
\begin{equation*}
\left[ \mathcal{A}\longmapsto H_{0}\left( R,\mathcal{A}\right) =\underset{%
\left( V\rightarrow U\right) \in \mathbf{C}_{R}}{\underrightarrow{\lim }}%
\mathcal{A}\left( V\right) \right] :\mathbf{pCS}\left( \mathbf{D},\mathbf{Pro%
}\left( \mathbf{K}\right) \right) \longrightarrow \mathbf{Pro}\left( \mathbf{%
K}\right)
\end{equation*}%
preserves arbitrary colimits (not necessarily finite!) because colimits
commute with colimits. Therefore, the above functor is right exact. Since
cofiltered limits are exact in the category $\mathbf{Pro}\left( \mathbf{K}%
\right) $ (Proposition \ref{Prop-Pro-objects-properties}(\ref%
{Prop-Pro-objects-properties-Cofiltered-limits-exact})), the functor 
\begin{equation*}
\left[ \mathcal{A}\longmapsto \mathcal{A}_{+}\left( U\right) =\underset{R\in
Cov\left( U\right) }{\underleftarrow{\lim }}~\underset{\left( V\rightarrow
U\right) \in \mathbf{C}_{R}}{\underrightarrow{\lim }}\mathcal{A}\left(
V\right) \right] :\mathbf{pCS}\left( \mathbf{D},\mathbf{Pro}\left( \mathbf{K}%
\right) \right) \longrightarrow \mathbf{Pro}\left( \mathbf{K}\right)
\end{equation*}%
is right exact as the composition of two right exact functors. Let $U\in 
\mathbf{C}_{X}$ vary. It follows that the corresponding functor%
\begin{equation*}
\left( {}\right) _{+}:\mathbf{pCS}\left( X,\mathbf{Pro}\left( \mathbf{K}%
\right) \right) \longrightarrow \mathbf{pCS}\left( X,\mathbf{Pro}\left( 
\mathbf{K}\right) \right)
\end{equation*}%
is exact.

(\textbf{\ref{Th-Cosheafification-plusplus-exact}}) Consider the composition%
\begin{equation*}
\left( {}\right) _{++}=\iota \circ \left( {}\right) _{\#}:\mathbf{pCS}\left(
X,\mathbf{Pro}\left( \mathbf{K}\right) \right) \longrightarrow \mathbf{CS}%
\left( X,\mathbf{Pro}\left( \mathbf{K}\right) \right) \longrightarrow 
\mathbf{pCS}\left( X,\mathbf{Pro}\left( \mathbf{K}\right) \right) ,
\end{equation*}%
which is right exact, due to (\ref{Th-Cosheafification-plus-right-exact}).
Since $\iota $ is fully faithful, the functor%
\begin{equation*}
\left( {}\right) _{\#}:\mathbf{pCS}\left( X,\mathbf{Pro}\left( \mathbf{K}%
\right) \right) \longrightarrow \mathbf{CS}\left( X,\mathbf{Pro}\left( 
\mathbf{K}\right) \right)
\end{equation*}%
is \textbf{right} exact as well. However, $\left( {}\right) _{\#}$, being a
right adjoint, preserves \textbf{arbitrary} (e.g., finite) limits, therefore
it is \textbf{left} exact.

(\textbf{\ref{Th-Cosheafification-(A,T)-separated}}) If $\mathcal{C}$ is
coseparated, then it follows from \cite[Proposition 2.10(1)]%
{Prasolov-Cosheafification-2016-zbMATH06684178} that $\left\langle \mathcal{C%
},T\right\rangle $ is separated for \textbf{any} (not necessarily injective) 
$T\in \mathbf{Mod}\left( k\right) $.

Assume now that $\left\langle \mathcal{C},T\right\rangle $ is separated for
any injective $T\in \mathbf{Mod}\left( k\right) $. Let $R\in Cov\left(
U\right) $ be a sieve. It follows that%
\begin{eqnarray*}
&&\left[ \left\langle \mathcal{C}\otimes _{\mathbf{Set}^{\mathbf{C}%
_{X}}}R,T\right\rangle \longleftarrow \left\langle \mathcal{C}\otimes _{%
\mathbf{Set}^{\mathbf{C}_{X}}}h_{U},T\right\rangle 
%TCIMACRO{\TeXButton{ISO}{\simeq}}%
%BeginExpansion
\simeq%
%EndExpansion
\left\langle \mathcal{C},T\right\rangle \left( U\right) =:\left\langle 
\mathcal{C}\left( U\right) ,T\right\rangle \right] 
%TCIMACRO{\TeXButton{ISO}{\simeq} }%
%BeginExpansion
\simeq
%EndExpansion
\\
&&%
%TCIMACRO{\TeXButton{ISO}{\simeq}}%
%BeginExpansion
\simeq%
%EndExpansion
\left[ Hom_{\mathbf{Set}^{\mathbf{C}_{X}}}\left( R,\left\langle \mathcal{C}%
,T\right\rangle \right) \longleftarrow Hom_{\mathbf{Set}^{\mathbf{C}%
_{X}}}\left( h_{U},\left\langle \mathcal{C},T\right\rangle \right) \right]
\end{eqnarray*}%
is a monomorphism, and, due to Proposition \ref{Prop-Pro-modules-properties}(%
\ref{Prop-Pro-modules-properties-Exact}) 
\begin{equation*}
\mathcal{C}\otimes _{\mathbf{Set}^{\mathbf{C}_{X}}}R\longrightarrow \mathcal{%
C}\otimes _{\mathbf{Set}^{\mathbf{C}_{X}}}h_{U}%
%TCIMACRO{\TeXButton{ISO}{\simeq}}%
%BeginExpansion
\simeq%
%EndExpansion
\mathcal{A}\left( U\right)
\end{equation*}%
is an epimorphism.

(\textbf{\ref{Th-Cosheafification-(A,T)-sheaf}}) Proved analogously, using 
\cite[Proposition 2.10(2)]{Prasolov-Cosheafification-2016-zbMATH06684178}
and Proposition \ref{Prop-Pro-modules-properties}(\ref%
{Prop-Pro-modules-properties-Exact}).

(\textbf{\ref{Th-Cosheafification-(A,T)-plus}}) See \cite[Proposition 2.11]%
{Prasolov-Cosheafification-2016-zbMATH06684178}.
\end{proof}

\subsection{Quasi-projective (pre)cosheaves}

\begin{definition}
\label{Def-Quasi-projective-(pre)cosheaf}Let $X$ be a small site, and $%
\mathbf{D}$ a small category.

\begin{enumerate}
\item \label{Def-Quasi-projective-precosheaf}Assume that $\mathcal{A}$ is a
precosheaf: either%
\begin{equation*}
\mathcal{A}\in \mathbf{pCS}\left( X,\mathbf{Pro}\left( k\right) \right) ,
\end{equation*}%
or%
\begin{equation*}
\mathcal{A}\in \mathbf{pCS}\left( \mathbf{D},\mathbf{Pro}\left( k\right)
\right) .
\end{equation*}
$\mathcal{A}$ is called \textbf{quasi-projective} iff for any injective $%
T\in \mathbf{Mod}\left( k\right) $, the presheaf 
\begin{equation*}
\left\langle \mathcal{A},T\right\rangle \in \mathbf{pS}\left( X,\mathbf{Mod}%
\left( k\right) \right) ,
\end{equation*}%
or%
\begin{equation*}
\left\langle \mathcal{A},T\right\rangle \in \mathbf{pS}\left( \mathbf{D},%
\mathbf{Mod}\left( k\right) \right) ,
\end{equation*}%
is \textbf{injective}.

\item \label{Def-Quasi-projective-cosheaf}A cosheaf 
\begin{equation*}
\mathcal{B}\in \mathbf{CS}\left( X,\mathbf{Pro}\left( k\right) \right)
\end{equation*}%
is called \textbf{quasi-projective} iff for any injective $T\in \mathbf{Mod}%
\left( k\right) $, the sheaf 
\begin{equation*}
\left\langle \mathcal{B},T\right\rangle \in \mathbf{S}\left( X,\mathbf{Mod}%
\left( k\right) \right)
\end{equation*}%
is \textbf{injective}.
\end{enumerate}
\end{definition}

\begin{notation}
\label{Not-Quasi-projective}Denote by%
\begin{equation*}
\mathbf{Q}\left( \mathbf{pCS}\left( X,\mathbf{Pro}\left( k\right) \right)
\right) \subseteq \mathbf{pCS}\left( X,\mathbf{Pro}\left( k\right) \right) ,
\end{equation*}%
or%
\begin{equation*}
\mathbf{Q}\left( \mathbf{pCS}\left( \mathbf{D},\mathbf{Pro}\left( k\right)
\right) \right) \subseteq \mathbf{pCS}\left( \mathbf{D},\mathbf{Pro}\left(
k\right) \right) ,
\end{equation*}%
the full subcategory of quasi-projective precosheaves, and by%
\begin{equation*}
\mathbf{Q}\left( \mathbf{CS}\left( X,\mathbf{Pro}\left( k\right) \right)
\right) \subseteq \mathbf{CS}\left( X,\mathbf{Pro}\left( k\right) \right)
\end{equation*}%
the full subcategory of quasi-projective cosheaves.
\end{notation}

\begin{definition}
\label{Def-discrete-site}~

\begin{enumerate}
\item A small category $\mathbf{C}$ is called \textbf{discrete} iff its only
morphisms are identities $\left( \mathbf{1}_{U}\right) _{U\in \mathbf{C}}$.

\item A site $X=\left( \mathbf{C}_{X},Cov\left( X\right) \right) $ is called 
\textbf{discrete} iff $\mathbf{C}_{X}$ is a discrete category and all sieves
are covering sieves.
\end{enumerate}
\end{definition}

\begin{example}
\label{Ex-Quasi-projective-discrete-category}Let $\mathbf{D}$ be a discrete
category, and assume that $\mathcal{A}\left( U\right) $ is a
quasi-projective pro-module (Definition \ref{Def-quasi-projective-pro-module}%
) for any $U\in \mathbf{D}$. Then the precosheaf $\mathcal{A}$ is
quasi-projective. Indeed, for any injective $T\in \mathbf{Mod}\left(
k\right) $, the $k$-modules $\left\langle \mathcal{A}\left( U\right)
,T\right\rangle $ are injective (remember that $k$ is quasi-noetherian!).
Since the functor%
\begin{equation*}
Hom_{\mathbf{pS}\left( \mathbf{D},\mathbf{Mod}\left( k\right) \right)
}\left( \bullet ,\left\langle \mathcal{A},T\right\rangle \right) 
%TCIMACRO{\TeXButton{ISO}{\simeq}}%
%BeginExpansion
\simeq%
%EndExpansion
\dprod\limits_{U\in \mathbf{D}}Hom_{\mathbf{Mod}\left( k\right) }\left(
\bullet \left( U\right) ,\left\langle \mathcal{A}\left( U\right)
,T\right\rangle \right)
\end{equation*}%
is exact, the presheaf $\left\langle \mathcal{A},T\right\rangle $ is
injective, and the precosheaf $\mathcal{A}$ is quasi-projective.
\end{example}

\begin{proposition}
\label{Prop-Left-Kan-quasi-projective}Let $\mathbf{D}$ and $\mathbf{E}$ be
small categories, and let%
\begin{equation*}
f:\mathbf{E}\longrightarrow \mathbf{D}
\end{equation*}%
be a functor. Then%
\begin{equation*}
f^{\dag }:\mathbf{pCS}\left( \mathbf{E},\mathbf{Pro}\left( k\right) \right)
\longrightarrow \mathbf{pCS}\left( \mathbf{D},\mathbf{Pro}\left( k\right)
\right) ,
\end{equation*}%
where $f^{\dag }$ is the left Kan extension of $f$ (Definition \ref%
{Def-Kan-extensions}) converts quasi-projectives into quasi-projectives.
\end{proposition}

\begin{proof}
Let $\mathcal{A}\in \mathbf{pCS}\left( \mathbf{E},\mathbf{Pro}\left(
k\right) \right) $ be quasi-projective, and $T\in \mathbf{Mod}\left(
k\right) $ be injective. It follows from Proposition \ref%
{Prop-Pro-modules-properties}(\ref{Prop-Pro-modules-properties-left-Kan})
that%
\begin{equation*}
\left\langle f^{\dag },T\right\rangle 
%TCIMACRO{\TeXButton{ISO}{\simeq}}%
%BeginExpansion
\simeq%
%EndExpansion
\left\langle f,T\right\rangle ^{\ddag }.
\end{equation*}%
Since $\left\langle f,T\right\rangle ^{\ddag }$ converts injectives into
injectives (Proposition \ref{Prop-Kan-extensions}(\ref%
{Prop-Kan-extensions-injective})), the presheaf $\left\langle f^{\dag }%
\mathcal{A},T\right\rangle $ is injective for any injective $T$, and the
precosheaf $f^{\dag }\mathcal{A}$ is quasi-projective.
\end{proof}

\begin{definition}
\label{Def-Flabby-cosheaf}~

\begin{enumerate}
\item A cosheaf $\mathcal{A}\in \mathbf{CS}\left( X,\mathbf{Pro}\left(
k\right) \right) $ on a \textbf{topological space} $X$ is called \textbf{%
flabby} iff $\mathcal{A}\left( V\rightarrow U\right) $ is a monomorphism for
any $\left( V\rightarrow U\right) \in \mathbf{C}_{X}$.

\item A cosheaf $\mathcal{A}\in \mathbf{CS}\left( X,\mathbf{Pro}\left(
k\right) \right) $ on a \textbf{small site} $X$ is called \textbf{flask} iff%
\begin{equation*}
H_{s}\left( R,\mathcal{A}\right) =0
\end{equation*}%
(see Definition \ref{Def-Cech-homology} (\ref{Def-Cech-homology-Cech-Hn(R)}, %
\ref{Def-Cech-homology-Cech-Hn(Ui)})) for any $s>0,$ and any covering sieve $%
R\subseteq h_{U}$.
\end{enumerate}
\end{definition}

\begin{definition}
\label{Def-Generators}Let $V\in \mathbf{E}$.

\begin{enumerate}
\item \label{Def-Generators-Precosheaves}Let $A\in \mathbf{Pro}\left(
k\right) $, considered as a precosheaf on the one-object category $\left\{
V\right\} $. Denote by $A^{V}$ and $A_{V}$\ the following precosheaves on $%
\mathbf{E}$:%
\begin{eqnarray*}
&&A^{V}%
%TCIMACRO{\TeXButton{assigned}{{:=}}}%
%BeginExpansion
{:=}%
%EndExpansion
\left( \left\{ V\right\} \longrightarrow \mathbf{E}\right) ^{\ddag }\left(
A\right) , \\
&&A_{V}%
%TCIMACRO{\TeXButton{assigned}{{:=}}}%
%BeginExpansion
{:=}%
%EndExpansion
\left( \left\{ V\right\} \longrightarrow \mathbf{E}\right) ^{\dag }\left(
A\right) ,
\end{eqnarray*}%
If $A$ is a quasi-projective pro-module, then, due to Example \ref%
{Ex-Quasi-projective-discrete-category} and Proposition \ref%
{Prop-Left-Kan-quasi-projective}, $A_{V}$ is a quasi-projective cosheaf on $%
\mathbf{E}$.

\item \label{Def-Generators-Presheaves}Let $A\in \mathbf{Mod}\left( k\right) 
$, considered as a presheaf on the one-object category $\left\{ V\right\} $.
Denote by $A^{V}$ and $A_{V}$\ the following presheaves on $\mathbf{E}$:%
\begin{eqnarray*}
&&A^{V}%
%TCIMACRO{\TeXButton{assigned}{{:=}}}%
%BeginExpansion
{:=}%
%EndExpansion
\left( \left\{ V\right\} \longrightarrow \mathbf{E}\right) ^{\ddag }\left(
A\right) , \\
&&A_{V}%
%TCIMACRO{\TeXButton{assigned}{{:=}}}%
%BeginExpansion
{:=}%
%EndExpansion
\left( \left\{ V\right\} \longrightarrow \mathbf{E}\right) ^{\dag }\left(
A\right) ,
\end{eqnarray*}%
If $A$ is an injective module, then, $A^{V}$ is an injective presheaf on $%
\mathbf{E}$ (compare to Example \ref{Ex-Quasi-projective-discrete-category}
and Proposition \ref{Prop-Left-Kan-quasi-projective}).
\end{enumerate}
\end{definition}

\begin{remark}
\label{Rem-Generators}~

\begin{enumerate}
\item \label{Rem-Generators-presheaves}The presheaves $\left\{ k_{V}~|~V\in 
\mathbf{E}\right\} $ form a set of generators for the category of presheaves 
$\mathbf{pS}\left( \mathbf{E},\mathbf{Mod}\left( k\right) \right) $. Indeed,%
\begin{equation*}
Hom_{\mathbf{pS}\left( \mathbf{E},\mathbf{Mod}\left( k\right) \right)
}\left( k_{V},\mathcal{A}\right) 
%TCIMACRO{\TeXButton{ISO}{\simeq}}%
%BeginExpansion
\simeq%
%EndExpansion
\left( \left\{ V\right\} \longrightarrow \mathbf{E}\right) _{\ast }\mathcal{A%
}%
%TCIMACRO{\TeXButton{ISO}{\simeq}}%
%BeginExpansion
\simeq%
%EndExpansion
Hom_{\mathbf{Pro}\left( k\right) }\left( k,\mathcal{A}\left( V\right)
\right) 
%TCIMACRO{\TeXButton{ISO}{\simeq}}%
%BeginExpansion
\simeq%
%EndExpansion
\mathcal{A}\left( V\right)
\end{equation*}%
for any $\mathcal{A}\in \mathbf{pS}\left( \mathbf{E},\mathbf{Mod}\left(
k\right) \right) $. Therefore, for any \textbf{proper} subpresheaf $\mathcal{%
B}\subseteq \mathcal{A}$, there exist a $V\in \mathbf{E}$, and an $a\in 
\mathcal{A}\left( V\right) $, $a\not\in \mathcal{B}\left( V\right) $. The
morphism $k_{V}\rightarrow \mathcal{A}$, corresponding to $a$, \textbf{does
not factor} through $\mathcal{B}$.

\item \label{Rem-Generators-sheaves}The sheaves $\left\{ \left( k_{V}\right)
^{\#}~|~V\in \mathbf{E}\right\} $ form a set of generators for the category
of sheaves $\mathbf{S}\left( X,\mathbf{Mod}\left( k\right) \right) $. Indeed,%
\begin{equation*}
Hom_{\mathbf{S}\left( X,\mathbf{Mod}\left( k\right) \right) }\left( \left(
k_{V}\right) ^{\#},\mathcal{A}\right) 
%TCIMACRO{\TeXButton{ISO}{\simeq}}%
%BeginExpansion
\simeq%
%EndExpansion
Hom_{\mathbf{pS}\left( X,\mathbf{Mod}\left( k\right) \right) }\left( k_{V},%
\mathcal{A}\right) 
%TCIMACRO{\TeXButton{ISO}{\simeq}}%
%BeginExpansion
\simeq%
%EndExpansion
\mathcal{A}\left( V\right)
\end{equation*}%
for any $\mathcal{A}\in \mathbf{S}\left( X,\mathbf{Mod}\left( k\right)
\right) $. Therefore, for any \textbf{proper} subsheaf $\mathcal{B}\subseteq 
\mathcal{A}$, there exist a $V\in \mathbf{E}$, and an $a\in \mathcal{A}%
\left( V\right) $, $a\not\in \mathcal{B}\left( V\right) $. The morphism $%
\left( k_{V}\right) ^{\#}\rightarrow \mathcal{A}$, corresponding to $a$, 
\textbf{does not factor} through $\mathcal{B}$.

\item \label{Rem-Generators-precosheaves}The precosheaves $\left\{
A^{V}~|~V\in \mathbf{E},~A\in \mathfrak{G}\subseteq \mathbf{Pro}\left(
k\right) \right\} $, where $\mathfrak{G}$ is the class from Proposition \ref%
{Prop-Pro-modules-properties}(\ref{Prop-Pro-modules-properties-cogenerators}%
), form a \textbf{class} of \textbf{co}generators for the category of
precosheaves $\mathbf{pCS}\left( \mathbf{E},\mathbf{Pro}\left( k\right)
\right) $. See Theorem \ref{Th-Properties-precosheaves} (\ref%
{Th-Properties-precosheaves-cogenerators}).

We \textbf{cannot}, however, choose a \textbf{set} of cogenerators for $%
\mathbf{pCS}\left( \mathbf{E},\mathbf{Pro}\left( k\right) \right) $, because
we cannot choose a set of cogenerators for $\mathbf{Pro}\left( k\right) $.

\item \label{Rem-Generators-cosheaves}The cosheaves $\left\{ \left(
A^{V}\right) _{\#}~|~V\in \mathbf{C}_{X},~A\in \mathfrak{G}\subseteq \mathbf{%
Pro}\left( k\right) \right\} $ form a \textbf{class} of \textbf{co}%
generators for the category of cosheaves $\mathbf{CS}\left( X,\mathbf{Pro}%
\left( k\right) \right) $. See Theorem \ref{Th-Properties-cosheaves} (\ref%
{Th-Properties-cosheaves-cogenerators}).

We cannot, however, choose a \textbf{set} of cogenerators for $\mathbf{CS}%
\left( X,\mathbf{Pro}\left( k\right) \right) $.
\end{enumerate}
\end{remark}

\begin{remark}
\label{Rem-Flabby-Topological-Spaces}~

\begin{enumerate}
\item A cosheaf $\mathcal{A}$ on a topological space is flabby iff $%
\left\langle \mathcal{A},T\right\rangle $ is a flabby sheaf \cite[Definition
II.5.1]{Bredon-Book-MR1481706} for all injective $T\in \mathbf{Mod}\left(
k\right) $. Indeed, $\left\langle \mathcal{A},T\right\rangle $ is flabby iff%
\begin{equation*}
\left\langle \mathcal{A},T\right\rangle \left( V\rightarrow U\right) 
%TCIMACRO{\TeXButton{ISO}{\simeq}}%
%BeginExpansion
\simeq%
%EndExpansion
\left\langle \mathcal{A}\left( V\rightarrow U\right) ,T\right\rangle
\end{equation*}%
is an epimorphism for any $\left( V\rightarrow U\right) \in \mathbf{C}_{X}$.
The latter is equivalent, since $T$ varies through all injective modules, to
the fact that $\mathcal{A}\left( V\rightarrow U\right) $ is a monomorphism
for any $\left( V\rightarrow U\right) \in \mathbf{C}_{X}$.

\item A cosheaf $\mathcal{A}$ on a general site is flask iff $\left\langle 
\mathcal{A},T\right\rangle $ is a flask sheaf (\cite[Definition 3.5.1]%
{Tamme-MR1317816} or \cite[Definition 2.4.1]{Artin-GT}) for all injective $%
T\in \mathbf{Mod}\left( k\right) $. Indeed, $\left\langle \mathcal{A}%
,T\right\rangle $ is flask iff%
\begin{equation*}
0=H^{s}\left( R,\left\langle \mathcal{A},T\right\rangle \right) 
%TCIMACRO{\TeXButton{ISO}{\simeq}}%
%BeginExpansion
\simeq%
%EndExpansion
\left\langle H_{s}\left( R,\mathcal{A}\right) ,T\right\rangle
\end{equation*}%
for all $s>0$ and all covering sieves $R$. The latter is equivalent, since $%
T $ varies through all injective modules, to the fact that $H_{s}\left( R,%
\mathcal{A}\right) $ is zero for all $s>0$ and all covering sieves $R$.

\item On a topological space, any flabby cosheaf is flask, because it
follows from \cite[Theorem II.5.5]{Bredon-Book-MR1481706}, that $%
\left\langle \mathcal{A},T\right\rangle $ is a flask sheaf whenever it is
flabby.
\end{enumerate}
\end{remark}

\section{Main results}

\subsection{Precosheaf homology}

\begin{theorem}
\label{Th-Properties-precosheaves}Let $\mathbf{E}$ be a small category.

\begin{enumerate}
\item \label{Th-Properties-precosheaves-abelian-(co)complete}The category $%
\mathbf{pCS}\left( \mathbf{E},\mathbf{Pro}\left( k\right) \right) $ of
precosheaves is abelian, complete and cocomplete, and satisfies both the $%
AB3 $ and $AB3^{\ast }$ axioms (\cite[1.5]%
{Grothendieck-Tohoku-1957-MR0102537}, \cite[Ch. 5.8]%
{Bucur-Deleanu-1968-Introduction-categories-functors-MR0236236}).

\item \label{Th-Properties-precosheaves-colimit}For any diagram%
\begin{equation*}
\mathcal{X}:\mathbf{I}\longrightarrow \mathbf{pCS}\left( \mathbf{E},\mathbf{%
Pro}\left( k\right) \right)
\end{equation*}%
and any $T\in \mathbf{Mod}\left( k\right) $ (not necessarily injective!)%
\begin{equation*}
\left\langle \underrightarrow{\lim }_{i\in \mathbf{I}}\mathcal{X}%
_{i},T\right\rangle 
%TCIMACRO{\TeXButton{ISO}{\simeq}}%
%BeginExpansion
\simeq%
%EndExpansion
\underleftarrow{\lim }_{i\in \mathbf{I}}\left\langle \mathcal{X}%
_{i},T\right\rangle
\end{equation*}%
in $\mathbf{pS}\left( \mathbf{E},\mathbf{Mod}\left( k\right) \right) $.

\item \label{Th-Properties-precosheaves-cofiltered-limit}For any \textbf{%
cofiltered} diagram%
\begin{equation*}
\mathcal{X}:\mathbf{I}\longrightarrow \mathbf{pCS}\left( \mathbf{E},\mathbf{%
Pro}\left( k\right) \right)
\end{equation*}%
and any $T\in \mathbf{Mod}\left( k\right) $ (not necessarily injective!)%
\begin{equation*}
\left\langle \underleftarrow{\lim }_{i\in \mathbf{I}}\mathcal{X}%
_{i},T\right\rangle 
%TCIMACRO{\TeXButton{ISO}{\simeq}}%
%BeginExpansion
\simeq%
%EndExpansion
\underrightarrow{\lim }_{i\in \mathbf{I}}\left\langle \mathcal{X}%
_{i},T\right\rangle
\end{equation*}%
in $\mathbf{pS}\left( \mathbf{E},\mathbf{Mod}\left( k\right) \right) $.

\item \label{Th-Properties-precosheaves-product}For any family $\left( 
\mathcal{X}_{i}\right) _{i\in I}$ in $\mathbf{pCS}\left( \mathbf{E},\mathbf{%
Pro}\left( k\right) \right) $ and any $T\in \mathbf{Mod}\left( k\right) $
(not necessarily injective!)%
\begin{equation*}
\left\langle \dprod\limits_{i\in I}\mathcal{X}_{i},T\right\rangle 
%TCIMACRO{\TeXButton{ISO}{\simeq}}%
%BeginExpansion
\simeq%
%EndExpansion
\dbigoplus\limits_{i\in I}\left\langle \mathcal{X}_{i},T\right\rangle
\end{equation*}%
in $\mathbf{pS}\left( \mathbf{E},\mathbf{Mod}\left( k\right) \right) $.

\item \label{Th-Properties-precosheaves-limit}For \textbf{any} (not
necessarily cofiltered\textbf{)} diagram%
\begin{equation*}
\mathcal{X}:\mathbf{I}\longrightarrow \mathbf{pCS}\left( \mathbf{E},\mathbf{%
Pro}\left( k\right) \right)
\end{equation*}%
and any injective $T\in \mathbf{Mod}\left( k\right) $%
\begin{equation*}
\left\langle \underleftarrow{\lim }_{i\in \mathbf{I}}\mathcal{X}%
_{i},T\right\rangle 
%TCIMACRO{\TeXButton{ISO}{\simeq}}%
%BeginExpansion
\simeq%
%EndExpansion
\underrightarrow{\lim }_{i\in \mathbf{I}}\left\langle \mathcal{X}%
_{i},T\right\rangle
\end{equation*}%
in $\mathbf{pS}\left( \mathbf{E},\mathbf{Mod}\left( k\right) \right) $.

\item \label{Th-Properties-precosheaves-left-Kan}\label%
{Prop-Pro-modules-properties-left-Kan}Let $\mathbf{D}$ and $\mathbf{E}$ be
small categories, let%
\begin{equation*}
f:\mathbf{E}\longrightarrow \mathbf{D}
\end{equation*}%
be a functor, and let $T\in \mathbf{Mod}\left( k\right) $. Then%
\begin{equation*}
\left\langle f^{\dag }\left( \bullet \right) ,T\right\rangle =\left(
f^{op}\right) ^{\ddag }\left\langle \bullet ,T\right\rangle :\mathbf{pS}%
\left( \mathbf{E},\mathbf{Mod}\left( k\right) \right) \longrightarrow 
\mathbf{pS}\left( \mathbf{D},\mathbf{Mod}\left( k\right) \right) ,
\end{equation*}%
where $f^{\dag }$ and $g^{\ddag }$ are the left and the right Kan extensions
(Definition \ref{Def-Kan-extensions}).

\item \label{Th-Properties-precosheaves-Zero}Let $\mathcal{M}\in \mathbf{pCS}%
\left( \mathbf{E},\mathbf{Pro}\left( k\right) \right) $. Then $\mathcal{M}%
%TCIMACRO{\TeXButton{ISO}{\simeq}}%
%BeginExpansion
\simeq%
%EndExpansion
0$ iff $\left\langle \mathcal{M},T\right\rangle 
%TCIMACRO{\TeXButton{ISO}{\simeq}}%
%BeginExpansion
\simeq%
%EndExpansion
0$ for any injective $T\in \mathbf{Mod}\left( k\right) $.

\item \label{Th-Properties-precosheaves-Homology}Let%
\begin{equation*}
\mathcal{E=}\left( \mathcal{M}\overset{\alpha }{\longleftarrow }\mathcal{N}%
\overset{\beta }{\longleftarrow }\mathcal{K}\right)
\end{equation*}%
be a sequence of morphisms in $\mathbf{pCS}\left( \mathbf{E},\mathbf{Pro}%
\left( k\right) \right) $ with $\beta \circ \alpha =0$, and let $T\in 
\mathbf{Mod}\left( k\right) $ be injective. Then%
\begin{equation*}
H\left( \mathcal{E}\right) 
%TCIMACRO{\TeXButton{assigned}{{:=}}}%
%BeginExpansion
{:=}%
%EndExpansion
\frac{\ker \left( \alpha \right) }{im\left( \beta \right) }
\end{equation*}%
satisfies%
\begin{equation*}
\left\langle H\left( \mathcal{E}\right) ,T\right\rangle 
%TCIMACRO{\TeXButton{ISO}{\simeq}}%
%BeginExpansion
\simeq%
%EndExpansion
H\left( \left\langle \mathcal{E},T\right\rangle \right) 
%TCIMACRO{\TeXButton{assigned}{{:=}}}%
%BeginExpansion
{:=}%
%EndExpansion
\frac{\ker \left( \left\langle \beta ,T\right\rangle \right) }{im\left(
\left\langle \alpha ,T\right\rangle \right) }.
\end{equation*}

\item \label{Th-Properties-precosheaves-Exact}Let%
\begin{equation*}
\mathcal{E=}\left( \mathcal{M}\overset{\alpha }{\longleftarrow }\mathcal{N}%
\overset{\beta }{\longleftarrow }\mathcal{K}\right)
\end{equation*}%
be a sequence of morphisms in $\mathbf{pCS}\left( \mathbf{E},\mathbf{Pro}%
\left( k\right) \right) $ with $\beta \circ \alpha =0$. Then $\mathcal{E}$
is exact iff the sequence%
\begin{equation*}
\left\langle \mathcal{M},T\right\rangle \overset{\left\langle \alpha
,T\right\rangle }{\longrightarrow }\left\langle \mathcal{N},T\right\rangle 
\overset{\left\langle \beta ,T\right\rangle }{\longrightarrow }\left\langle 
\mathcal{K},T\right\rangle
\end{equation*}%
is exact in $\mathbf{pS}\left( \mathbf{E},\mathbf{Mod}\left( k\right)
\right) $ for all injective $T\in \mathbf{Mod}\left( k\right) $.

\item \label{Th-Properties-precosheaves-AB4}The category $\mathbf{pCS}\left( 
\mathbf{E},\mathbf{Pro}\left( k\right) \right) $ satisfies the $AB4$ axiom (%
\cite[1.5]{Grothendieck-Tohoku-1957-MR0102537}, \cite[Ch. 5.8]%
{Bucur-Deleanu-1968-Introduction-categories-functors-MR0236236}).

\item \label{Th-Properties-precosheaves-AB4*}The category $\mathbf{pCS}%
\left( \mathbf{E},\mathbf{Pro}\left( k\right) \right) $ satisfies the $%
AB4^{\ast }$ axiom (\cite[1.5]{Grothendieck-Tohoku-1957-MR0102537}, \cite[%
Ch. 5.8]{Bucur-Deleanu-1968-Introduction-categories-functors-MR0236236}).

\item \label{Th-Properties-precosheaves-cofiltered-limits-exact}\label%
{Th-Properties-precosheaves-AB5*}The category $\mathbf{pCS}\left( \mathbf{E},%
\mathbf{Pro}\left( k\right) \right) $ satisfies the $AB5^{\ast }$ axiom (%
\cite[1.5]{Grothendieck-Tohoku-1957-MR0102537}, \cite[Ch. 5.8]%
{Bucur-Deleanu-1968-Introduction-categories-functors-MR0236236}): cofiltered
limits are exact in $\mathbf{pCS}\left( \mathbf{E},\mathbf{Pro}\left(
k\right) \right) $.

\item \label{Th-Properties-precosheaves-cogenerators}The \textbf{class} (%
\textbf{not} a set) 
\begin{equation*}
\left\{ A^{V}~|~V\in \mathbf{E},~A\in \mathfrak{G}\subseteq \mathbf{Pro}%
\left( k\right) \right\} ,
\end{equation*}%
where $\mathfrak{G}$ is the class from Proposition \ref%
{Prop-Pro-modules-properties}(\ref{Prop-Pro-modules-properties-cogenerators}%
) forms a class of cogenerators (\cite[1.9]%
{Grothendieck-Tohoku-1957-MR0102537}, \cite[Ch. 5.9]%
{Bucur-Deleanu-1968-Introduction-categories-functors-MR0236236}) of the
category $\mathbf{pCS}\left( \mathbf{E},\mathbf{Pro}\left( k\right) \right) $%
.
\end{enumerate}
\end{theorem}

For the proof, see Section \ref{Sec-Proof-Properties-precosheaves}.

\begin{theorem}
\label{Th-Precosheaf-homology}Let $X=\left( \mathbf{C}_{X},Cov\left(
X\right) \right) $ be a small site and $\mathbf{D}$ a small category. Let
also%
\begin{eqnarray*}
\mathcal{A} &\in &\mathbf{pCS}\left( X,\mathbf{Pro}\left( k\right) \right) ,
\\
\mathcal{B} &\in &\mathbf{pCS}\left( \mathbf{D},\mathbf{Pro}\left( k\right)
\right) .
\end{eqnarray*}%
The statements on the precosheaf $\mathcal{B}$ below are \textbf{more general%
} than the statements on the precosheaf $\mathcal{A}$, therefore they are
valid also for $\mathcal{A}$. Remind that $\check{H}$ and~$^{Roos}\check{H}$
are (isomorphic when the topology is generated by a pre-topology!) 
%TCIMACRO{\TeXButton{Cech }{\u{C}ech} }%
%BeginExpansion
\u{C}ech
%EndExpansion
homologies from Definition \ref{Def-Cech-homology}.

\begin{enumerate}
\item \label{Th-Properties-precosheaves-quasi-projective}\label%
{Th-Precosheaf-homology-surjection}There exists a functorial epimorphism%
\begin{equation*}
\pi :\mathcal{P}\left( \mathcal{B}\right) \twoheadrightarrow \mathcal{B},
\end{equation*}%
where $\mathcal{P}\left( \mathcal{B}\right) $ is quasi-projective
(Definition \ref{Def-Quasi-projective-(pre)cosheaf}(\ref%
{Def-Quasi-projective-precosheaf})).

\item \label{Th-Precosheaf-homology-F-projective}~

\begin{enumerate}
\item The full subcategory%
\begin{equation*}
\mathbf{Q}\left( \mathbf{pCS}\left( \mathbf{D},\mathbf{Pro}\left( k\right)
\right) \right) \subseteq \mathbf{pCS}\left( \mathbf{D},\mathbf{Pro}\left(
k\right) \right)
\end{equation*}%
(Notation \ref{Not-Quasi-projective}) is $F$-projective (Definition \ref%
{Def-F-projective}) with respect to the functors%
\begin{equation*}
F\left( \bullet \right) =H_{0}\left( R,\bullet \right) ,
\end{equation*}%
where $R\subseteq h_{U}$ runs through the sieves (Definition \ref{Def-Sieve}%
) on $\mathbf{D}$;

\item The full subcategory%
\begin{equation*}
\mathbf{Q}\left( \mathbf{pCS}\left( X,\mathbf{Pro}\left( k\right) \right)
\right) \subseteq \mathbf{pCS}\left( X,\mathbf{Pro}\left( k\right) \right)
\end{equation*}%
is $F$-projective with respect to the functors%
\begin{equation*}
F\left( \bullet \right) =~^{Roos}\check{H}_{0}\left( U,\bullet \right) 
%TCIMACRO{\TeXButton{ISO}{\simeq}}%
%BeginExpansion
\simeq%
%EndExpansion
\check{H}_{0}\left( U,\bullet \right) ,~U\in \mathbf{C}_{X}.
\end{equation*}
\end{enumerate}

\item \label{Th-Precosheaf-Cech-homology}\label%
{Th-Precosheaf-homology-Cech-homology}~

\begin{enumerate}
\item If the sieve $R$ is generated by a base-changeable (Definition \ref%
{Def-Quarrable}) family $\left\{ U_{i}\rightarrow U\right\} $, then the left
satellites (Definition \ref{Def-Left-derived-functors}) $L_{n}H_{0}\left( R,%
\mathcal{B}\right) $ satisfy%
\begin{equation*}
L_{n}H_{0}\left( R,\mathcal{B}\right) 
%TCIMACRO{\TeXButton{ISO}{\simeq}}%
%BeginExpansion
\simeq%
%EndExpansion
H_{n}\left( \left\{ U_{i}\rightarrow U\right\} ,\mathcal{B}\right) ,
\end{equation*}%
naturally in $\mathcal{B}$ and $R$.

\item The left satellites $L_{n}\check{H}_{0}\left( U,\mathcal{A}\right) $
are naturally, in $U$ and $\mathcal{A}$, isomorphic to 
\begin{equation*}
\check{H}_{n}\left( U,\mathcal{A}\right) 
%TCIMACRO{\TeXButton{ISO}{\simeq}}%
%BeginExpansion
\simeq%
%EndExpansion
~^{Roos}\check{H}_{n}\left( U,\mathcal{A}\right) .
\end{equation*}
\end{enumerate}

\item \label{Th-Precosheaf-homology-pairing}There are isomorphisms, natural
in $\mathcal{A}$, $\mathcal{B}$, $R$, and $T$,

\begin{enumerate}
\item 
\begin{equation*}
\left\langle H_{n}\left( R,\mathcal{B}\right) ,T\right\rangle 
%TCIMACRO{\TeXButton{ISO}{\simeq}}%
%BeginExpansion
\simeq%
%EndExpansion
H^{n}\left( R,\left\langle \mathcal{B},T\right\rangle \right)
\end{equation*}%
for any injective $T\in \mathbf{Mod}\left( k\right) $.

\item 
\begin{equation*}
\left\langle ~^{Roos}\check{H}_{n}\left( U,\mathcal{A}\right)
,T\right\rangle 
%TCIMACRO{\TeXButton{ISO}{\simeq}}%
%BeginExpansion
\simeq%
%EndExpansion
\left\langle \check{H}_{n}\left( U,\mathcal{A}\right) ,T\right\rangle 
%TCIMACRO{\TeXButton{ISO}{\simeq}}%
%BeginExpansion
\simeq%
%EndExpansion
\check{H}^{n}\left( U,\left\langle \mathcal{A},T\right\rangle \right) 
%TCIMACRO{\TeXButton{ISO}{\simeq}}%
%BeginExpansion
\simeq%
%EndExpansion
~^{Roos}\check{H}^{n}\left( U,\left\langle \mathcal{A},T\right\rangle \right)
\end{equation*}%
for any injective $T\in \mathbf{Mod}\left( k\right) $ (see Notation \ref%
{Not-Hn(sieve)}).
\end{enumerate}
\end{enumerate}
\end{theorem}

See Section \ref{Sec-Proof-precosheaf-homology} for the proof.

\begin{notation}
\label{Not-Hn(sieve)}\label{Rem-Hn(sieve)}For a sieve $R\subseteq h_{U}$,
the left satellites $L_{n}H_{0}\left( R,\bullet \right) $ are denoted by $%
H_{n}\left( R,\bullet \right) $.
\end{notation}

\subsection{Cosheaf homology}

\begin{theorem}
\label{Th-Properties-cosheaves}Let $X=\left( \mathbf{C}_{X},Cov\left(
X\right) \right) $ be a site.

\begin{enumerate}
\item \label{Th-Properties-cosheaves-abelian-(co)cocomplete}The category $%
\mathbf{CS}\left( X,\mathbf{Pro}\left( k\right) \right) $ is abelian,
complete and cocomplete, and satisfies both the $AB3$ and $AB3^{\ast }$
axioms (\cite[1.5]{Grothendieck-Tohoku-1957-MR0102537}, \cite[Ch. 5.8]%
{Bucur-Deleanu-1968-Introduction-categories-functors-MR0236236}).

\item \label{Th-Properties-cosheaves-colimit}For any diagram%
\begin{equation*}
\mathcal{X}:\mathbf{I}\longrightarrow \mathbf{CS}\left( X,\mathbf{Pro}\left(
k\right) \right)
\end{equation*}%
and any $T\in \mathbf{Mod}\left( k\right) $ (not necessarily injective!)%
\begin{equation*}
\left\langle \underrightarrow{\lim }_{i\in \mathbf{I}}\mathcal{X}%
_{i},T\right\rangle 
%TCIMACRO{\TeXButton{ISO}{\simeq}}%
%BeginExpansion
\simeq%
%EndExpansion
\underleftarrow{\lim }_{i\in \mathbf{I}}\left\langle \mathcal{X}%
_{i},T\right\rangle
\end{equation*}%
in $\mathbf{S}\left( X,\mathbf{Mod}\left( k\right) \right) $.

\item \label{Th-Properties-cosheaves-cofiltered-limit}For any \textbf{%
cofiltered} diagram%
\begin{equation*}
\mathcal{X}:\mathbf{I}\longrightarrow \mathbf{CS}\left( X,\mathbf{Pro}\left(
k\right) \right)
\end{equation*}%
and any $T\in \mathbf{Mod}\left( k\right) $ (not necessarily injective!)%
\begin{equation*}
\left\langle \underleftarrow{\lim }_{i\in \mathbf{I}}\mathcal{X}%
_{i},T\right\rangle 
%TCIMACRO{\TeXButton{ISO}{\simeq}}%
%BeginExpansion
\simeq%
%EndExpansion
\underrightarrow{\lim }_{i\in \mathbf{I}}\left\langle \mathcal{X}%
_{i},T\right\rangle
\end{equation*}%
in $\mathbf{S}\left( X,\mathbf{Mod}\left( k\right) \right) $.

\item \label{Th-Properties-cosheaves-product}For any family $\left( \mathcal{%
X}_{i}\right) _{i\in I}$ in $\mathbf{CS}\left( X,\mathbf{Pro}\left( k\right)
\right) $ and any $T\in \mathbf{Mod}\left( k\right) $ (not necessarily
injective!)%
\begin{equation*}
\left\langle \dprod\limits_{i\in I}\mathcal{X}_{i},T\right\rangle 
%TCIMACRO{\TeXButton{ISO}{\simeq}}%
%BeginExpansion
\simeq%
%EndExpansion
\dbigoplus\limits_{i\in I}\left\langle \mathcal{X}_{i},T\right\rangle
\end{equation*}%
in $\mathbf{S}\left( X,\mathbf{Mod}\left( k\right) \right) $.

\item \label{Th-Properties-cosheaves-limit}For \textbf{any} (not necessarily
cofiltered\textbf{)} diagram%
\begin{equation*}
\mathcal{X}:\mathbf{I}\longrightarrow \mathbf{Pro}\left( k\right)
\end{equation*}%
and any injective $T\in \mathbf{Mod}\left( k\right) $%
\begin{equation*}
\left\langle \underleftarrow{\lim }_{i\in \mathbf{I}}\mathcal{X}%
_{i},T\right\rangle 
%TCIMACRO{\TeXButton{ISO}{\simeq}}%
%BeginExpansion
\simeq%
%EndExpansion
\underrightarrow{\lim }_{i\in \mathbf{I}}\left\langle \mathcal{X}%
_{i},T\right\rangle
\end{equation*}%
in $\mathbf{S}\left( X,\mathbf{Mod}\left( k\right) \right) $.

\item \label{Th-Properties-cosheaves-Zero}Let $\mathcal{M}\in \mathbf{CS}%
\left( X,\mathbf{Pro}\left( k\right) \right) $. Then $\mathcal{M}%
%TCIMACRO{\TeXButton{ISO}{\simeq}}%
%BeginExpansion
\simeq%
%EndExpansion
0$ iff $\left\langle \mathcal{M},T\right\rangle =0$ for any injective $T\in 
\mathbf{Mod}\left( k\right) $.

\item \label{Th-Properties-cosheaves-Homology}Let%
\begin{equation*}
\mathcal{E=}\left( \mathcal{M}\overset{\alpha }{\longleftarrow }\mathcal{N}%
\overset{\beta }{\longleftarrow }\mathcal{K}\right)
\end{equation*}%
be a sequence of morphisms in $\mathbf{CS}\left( X,\mathbf{Pro}\left(
k\right) \right) $ with $\beta \circ \alpha =0$, and let $T\in \mathbf{Mod}%
\left( k\right) $ be injective. Then%
\begin{equation*}
H\left( \mathcal{E}\right) 
%TCIMACRO{\TeXButton{assigned}{{:=}}}%
%BeginExpansion
{:=}%
%EndExpansion
\frac{\ker \left( \alpha \right) }{im\left( \beta \right) }
\end{equation*}%
satisfies%
\begin{equation*}
\left\langle H\left( \mathcal{E}\right) ,T\right\rangle 
%TCIMACRO{\TeXButton{ISO}{\simeq}}%
%BeginExpansion
\simeq%
%EndExpansion
H\left( \left\langle \mathcal{E},T\right\rangle \right) 
%TCIMACRO{\TeXButton{assigned}{{:=}}}%
%BeginExpansion
{:=}%
%EndExpansion
\frac{\ker \left( \left\langle \beta ,T\right\rangle \right) }{im\left(
\left\langle \alpha ,T\right\rangle \right) }.
\end{equation*}

\item \label{Th-Properties-cosheaves-Exact}Let%
\begin{equation*}
\mathcal{E=}\left( \mathcal{M}\overset{\alpha }{\longleftarrow }\mathcal{N}%
\overset{\beta }{\longleftarrow }\mathcal{K}\right)
\end{equation*}%
be a sequence of morphisms in $\mathbf{CS}\left( X,\mathbf{Pro}\left(
k\right) \right) $ with $\beta \circ \alpha =0$. Then $\mathcal{E}$ is exact
iff the sequence%
\begin{equation*}
\left\langle \mathcal{M},T\right\rangle \overset{\left\langle \alpha
,T\right\rangle }{\longrightarrow }\left\langle \mathcal{N},T\right\rangle 
\overset{\left\langle \beta ,T\right\rangle }{\longrightarrow }\left\langle 
\mathcal{K},T\right\rangle
\end{equation*}%
is exact in $\mathbf{S}\left( X,\mathbf{Mod}\left( k\right) \right) $ for
all injective $T\in \mathbf{Mod}\left( k\right) $.

\item \label{Th-Properties-cosheaves-AB4*}The category $\mathbf{CS}\left( X,%
\mathbf{Pro}\left( k\right) \right) $ satisfies the $AB4^{\ast }$ axiom (%
\cite[1.5]{Grothendieck-Tohoku-1957-MR0102537}, \cite[Ch. 5.8]%
{Bucur-Deleanu-1968-Introduction-categories-functors-MR0236236}).

\item \label{Th-Properties-cosheaves-cofiltered-limits-exact}\label%
{Th-Properties-cosheaves-AB5*}The category $\mathbf{CS}\left( X,\mathbf{Pro}%
\left( k\right) \right) $ satisfies the $AB5^{\ast }$ axiom (\cite[1.5]%
{Grothendieck-Tohoku-1957-MR0102537}, \cite[Ch. 5.8]%
{Bucur-Deleanu-1968-Introduction-categories-functors-MR0236236}): cofiltered
limits are exact in $\mathbf{CS}\left( X,\mathbf{Pro}\left( k\right) \right) 
$.

\item \label{Th-Properties-cosheaves-cogenerators}The \textbf{class} (%
\textbf{not} a set) 
\begin{equation*}
\left\{ \left( A^{V}\right) _{\#}~|~V\in \mathbf{E},~A\in \mathfrak{G}%
\subseteq \mathbf{Pro}\left( k\right) \right\} ,
\end{equation*}%
where $\mathfrak{G}$ is the class from Proposition \ref%
{Prop-Pro-modules-properties}(\ref{Prop-Pro-modules-properties-cogenerators}%
) forms a class of cogenerators (\cite[1.9]%
{Grothendieck-Tohoku-1957-MR0102537}, \cite[Ch. 5.9]%
{Bucur-Deleanu-1968-Introduction-categories-functors-MR0236236}) of the
category $\mathbf{CS}\left( X,\mathbf{Pro}\left( k\right) \right) $.
\end{enumerate}
\end{theorem}

For the proof, see Section \ref{Sec-Proof-Properties-cosheaves}.

\begin{theorem}
\label{Th-Cosheaf-homology}Let $X$ be a small site. Let also $\check{H}$ and~%
$^{Roos}\check{H}$ be (isomorphic when the topology is generated by a
pre-topology!) 
%TCIMACRO{\TeXButton{Cech }{\u{C}ech} }%
%BeginExpansion
\u{C}ech
%EndExpansion
homologies from Definition \ref{Def-Cech-homology}.

\begin{enumerate}
\item \label{Th-Cosheaf-homology-surjection}For an arbitrary cosheaf $%
\mathcal{A}\in \mathbf{CS}\left( X,\mathbf{Pro}\left( k\right) \right) $,
there exists a functorial epimorphism%
\begin{equation*}
\sigma \left( \mathcal{A}\right) :\mathcal{R}\left( \mathcal{A}\right)
\twoheadrightarrow \mathcal{A},
\end{equation*}%
where $R\left( \mathcal{A}\right) $ is quasi-projective.

\item \label{Th-Cosheaf-homology-F-projective}\label%
{Th-Cosheaf-homology-surjection-F-projective}The full subcategory%
\begin{equation*}
\mathbf{Q}\left( \mathbf{CS}\left( X,\mathbf{Pro}\left( k\right) \right)
\right) \subseteq \mathbf{CS}\left( X,\mathbf{Pro}\left( k\right) \right)
\end{equation*}%
is $F$-projective (Definition \ref{Def-F-projective}) with respect to the
functors:

\begin{enumerate}
\item 
\begin{equation*}
F\left( \bullet \right) =\Gamma \left( U,\bullet \right) 
%TCIMACRO{\TeXButton{assigned}{{:=}}}%
%BeginExpansion
{:=}%
%EndExpansion
\bullet \left( U\right) ;
\end{equation*}

\item 
\begin{equation*}
F=\iota :\mathbf{CS}\left( X,\mathbf{Pro}\left( k\right) \right)
\hookrightarrow \mathbf{pCS}\left( X,\mathbf{Pro}\left( k\right) \right) .
\end{equation*}
\end{enumerate}

\item \label{Th-Cosheaf-homology-pairing}The left satellites $L_{n}\Gamma
\left( U,\bullet \right) $ satisfy%
\begin{equation*}
\left\langle L_{n}\Gamma \left( U,\bullet \right) ,T\right\rangle 
%TCIMACRO{\TeXButton{ISO}{\simeq}}%
%BeginExpansion
\simeq%
%EndExpansion
H^{n}\left( U,\left\langle \bullet ,T\right\rangle \right)
\end{equation*}%
for any injective $T\in \mathbf{Mod}\left( k\right) $.

\item \label{Th-Cosheaf-homology-forgetting}The left satellites $L_{n}\iota $
satisfy

\begin{enumerate}
\item 
\begin{equation*}
\left\langle \left( L_{n}\iota \right) \bullet ,T\right\rangle 
%TCIMACRO{\TeXButton{ISO}{\simeq}}%
%BeginExpansion
\simeq%
%EndExpansion
\mathcal{H}^{n}\left( \left\langle \bullet ,T\right\rangle \right) ,
\end{equation*}%
for any injective $T\in \mathbf{Mod}\left( k\right) $ (see Notation \ref%
{Not-LRn-Iota} for $\mathcal{H}^{n}$),

\item 
\begin{equation*}
\left[ \left( L_{n}\iota \right) \mathcal{A}\right] \left( U\right) 
%TCIMACRO{\TeXButton{ISO}{\simeq}}%
%BeginExpansion
\simeq%
%EndExpansion
H_{n}\left( U,\mathcal{A}\right) .
\end{equation*}
\end{enumerate}

\item \label{Th-Cosheaf-homology-H-t-plus-trivial}%
\begin{equation*}
\left( \mathcal{H}_{t}\mathcal{A}\right) _{+}=0
\end{equation*}%
for all $t>0$.

\item \label{Th-Cosheaf-homology-Spectral-sequence-R}~

\begin{enumerate}
\item For any $U\in \mathbf{C}_{X}$ and any covering sieve $R$ on $U$ there
exists a natural spectral sequence%
\begin{equation*}
E_{s,t}^{2}=H_{s}\left( R,\mathcal{H}_{t}\left( \mathcal{A}\right) \right)
\implies H_{s+t}\left( U,\mathcal{A}\right) ,
\end{equation*}%
converging to the homology of $\mathcal{A}$ (see Notation \ref{Not-LRn-Iota}
for $\mathcal{H}_{t}$).

\item For any $U\in \mathbf{C}_{X}$ there exists a natural spectral sequence%
\begin{equation*}
E_{s,t}^{2}=~^{Roos}\check{H}_{s}\left( U,\mathcal{H}_{t}\left( \mathcal{A}%
\right) \right) \implies H_{s+t}\left( U,\mathcal{A}\right) ,
\end{equation*}%
converging to the homology of $\mathcal{A}$.

\item There are natural (in $U$ and $\mathcal{A}$) isomorphisms%
\begin{eqnarray*}
&&H_{0}\left( U,\mathcal{A}\right) 
%TCIMACRO{\TeXButton{ISO}{\simeq}}%
%BeginExpansion
\simeq%
%EndExpansion
\check{H}_{0}\left( U,\mathcal{A}\right) , \\
&&H_{1}\left( U,\mathcal{A}\right) 
%TCIMACRO{\TeXButton{ISO}{\simeq}}%
%BeginExpansion
\simeq%
%EndExpansion
\check{H}_{1}\left( U,\mathcal{A}\right) ,
\end{eqnarray*}%
and a natural (in $U$ and $\mathcal{A}$) epimorphism%
\begin{equation*}
H_{2}\left( U,\mathcal{A}\right) \twoheadrightarrow \check{H}_{2}\left( U,%
\mathcal{A}\right) .
\end{equation*}
\end{enumerate}

\item \label{Th-Cosheaf-homology-Spectral-sequence-Cech}Assume that the
topology on $X$ is generated by a pretopology (Definition \ref%
{Def-Pretopology}). Then:

\begin{enumerate}
\item The spectral sequence from (\ref%
{Th-Cosheaf-homology-Spectral-sequence-R}a) becomes%
\begin{equation*}
E_{s,t}^{2}=H_{s}\left( \left\{ U_{i}\rightarrow U\right\} ,\mathcal{H}%
_{t}\left( \mathcal{A}\right) \right) \implies H_{s+t}\left( U,\mathcal{A}%
\right) .
\end{equation*}

\item The spectral sequence from (\ref%
{Th-Cosheaf-homology-Spectral-sequence-R}b) becomes%
\begin{equation*}
E_{s,t}^{2}=\check{H}_{s}\left( U,\mathcal{H}_{t}\left( \mathcal{A}\right)
\right) \implies H_{s+t}\left( U,\mathcal{A}\right) .
\end{equation*}
\end{enumerate}
\end{enumerate}
\end{theorem}

See Section \ref{Sec-Proof-cosheaf-homology} for the proof.

\begin{notation}
\label{Not-LRn-Iota}~

\begin{enumerate}
\item Denote by $\mathcal{H}_{n}$ the left satellites of the embedding%
\begin{equation*}
\iota :\mathbf{CS}\left( X,\mathbf{Pro}\left( k\right) \right)
\longrightarrow \mathbf{pCS}\left( X,\mathbf{Pro}\left( k\right) \right) ;
\end{equation*}

\item Denote by $\mathcal{H}^{n}$ the right satellites of the embedding%
\begin{equation*}
\iota :\mathbf{S}\left( X,\mathbf{Pro}\left( k\right) \right)
\longrightarrow \mathbf{pS}\left( X,\mathbf{Pro}\left( k\right) \right) ;
\end{equation*}
\end{enumerate}
\end{notation}

\section{Examples}

\begin{example}
\label{Ex-Convergent-sequence}Let $X$ be the convergent sequence from \cite[%
Example 4.8]{Prasolov-Cosheafification-2016-zbMATH06684178}:%
\begin{equation*}
X=\left\{ x_{0}\right\} \cup \left\{ x_{1},x_{2},x_{3},...\right\} =\left\{
0\right\} \cup \left\{ 1,\frac{1}{2},\frac{1}{3},...\right\} \subseteq 
\mathbb{R},
\end{equation*}%
let $G\in \mathbf{Ab}$, $G\neq \left\{ 0\right\} $, and let $\mathcal{A}%
=G_{\#}$ be the constant cosheaf. Then%
\begin{equation*}
H_{n}\left( X,\mathcal{A}\right) =\check{H}_{n}\left( X,\mathcal{A}\right)
=pro\text{-}H_{n}\left( X,G\right) =\left\{ 
\begin{array}{ccc}
\mathbf{B} & \text{if} & n=0, \\ 
0 & \text{if} & n\neq 0,%
\end{array}%
\right.
\end{equation*}%
where $\mathbf{B}$ is an abelian pro-group which is \textbf{not} rudimentary
(Remark \ref{Rem-Trivial-pro-object}), i.e. 
\begin{equation*}
\mathbf{B}\not\in \mathbf{Ab\subseteq Pro}\left( \mathbf{Ab}\right) .
\end{equation*}
\end{example}

\begin{proof}
Let $T\in \mathbf{Ab}$ be injective. It is easy to check that the cosheaf $%
\mathcal{A}$ is flabby, therefore the sheaf $\left\langle \mathcal{A}%
,T\right\rangle $ is flabby, thus acyclic. Due to Theorem \ref%
{Th-Cosheaf-homology}(\ref{Th-Cosheaf-homology-pairing}), the cosheaf $%
\mathcal{A}$ is acyclic, too:%
\begin{equation*}
H_{n}\left( X,\mathcal{A}\right) =\left\{ 
\begin{array}{ccc}
\mathbf{0} & \text{if} & n>0; \\ 
\check{H}_{0}\left( X,\mathcal{A}\right) =\mathcal{A}\left( X\right) =pro%
\text{-}H_{0}\left( X,G\right) & \text{if} & n=0.%
\end{array}%
\right.
\end{equation*}%
It remains to calculate $pro$-$H_{0}\left( X,G\right) $. Since%
\begin{equation*}
\mathbf{Pro}\left( \mathbf{Ab}\right) \longrightarrow \left( \mathbf{Set}^{%
\mathbf{Ab}}\right) ^{op}
\end{equation*}%
is a full embedding by Definition \ref{Def-Pro-category}, it is enough to
describe the functor%
\begin{equation*}
Hom_{\mathbf{Pro}\left( \mathbf{Ab}\right) }\left( pro\text{-}H_{0}\left(
X,G\right) ,\bullet \right) :\mathbf{Ab}\longrightarrow \mathbf{Set.}
\end{equation*}%
During the proof of \cite[Proof of Theorem 3.11(3)]%
{Prasolov-Cosheafification-2016-zbMATH06684178}, it is established a natural
in $X$ isomorphism%
\begin{equation*}
pro\text{-}H_{0}\left( X,G\right) 
%TCIMACRO{\TeXButton{ISO}{\simeq}}%
%BeginExpansion
\simeq%
%EndExpansion
G\otimes _{\mathbf{Set}}pro\text{-}\pi _{0}\left( X\right) .
\end{equation*}%
Moreover, in \cite[Proposition 3.13]%
{Prasolov-Cosheafification-2016-zbMATH06684178} another natural (in $X$ and $%
T\in \mathbf{Ab}$) isomorphism is proved:%
\begin{equation*}
Hom_{\mathbf{Pro}\left( \mathbf{Ab}\right) }\left( G\otimes _{\mathbf{Set}%
}pro\text{-}\pi _{0}\left( X\right) ,T\right) 
%TCIMACRO{\TeXButton{ISO}{\simeq}}%
%BeginExpansion
\simeq%
%EndExpansion
\left[ Hom_{\mathbf{Ab}}\left( G,T\right) \right] ^{X},
\end{equation*}%
where $\left[ Hom_{\mathbf{Ab}}\left( G,T\right) \right] ^{X}$ is the set of
continuous mappings from $X$ to $Hom_{\mathbf{Ab}}\left( G,T\right) $, where
the latter space is supplied with the discrete topology. In fact,%
\begin{equation*}
\left[ Hom_{\mathbf{Ab}}\left( G,T\right) \right] ^{X}%
%TCIMACRO{\TeXButton{ISO}{\simeq}}%
%BeginExpansion
\simeq%
%EndExpansion
\check{H}^{0}\left( X,Hom_{\mathbf{Ab}}\left( G,T\right) \right) ,
\end{equation*}%
where $\check{H}^{0}$ is the classical 
%TCIMACRO{\TeXButton{Cech }{\u{C}ech} }%
%BeginExpansion
\u{C}ech
%EndExpansion
\textbf{co}homology for topological spaces, but we \textbf{do not need} this
fact. Continuous mappings to a discrete space are locally constant, and vice
versa. Consider such a mapping%
\begin{equation*}
f:X\longrightarrow Hom_{\mathbf{Ab}}\left( G,T\right) .
\end{equation*}%
Since it is locally constant at $x=x_{0}$, there exists an $n\in \mathbb{Z}$
such that for all $i>n$%
\begin{equation*}
f\left( x_{i}\right) =f\left( x_{0}\right) .
\end{equation*}%
Therefore,%
\begin{equation*}
\left[ Hom_{\mathbf{Ab}}\left( G,T\right) \right] ^{X}%
%TCIMACRO{\TeXButton{ISO}{\simeq}}%
%BeginExpansion
\simeq%
%EndExpansion
\underrightarrow{\lim }\left( 
%TCIMACRO{%
%\TeXButton{Sequence}{\begin{diagram}
%C_{1} & \rTo^{q_{1\rightarrow 2}} & C_{2} & \rTo^{q_{2\rightarrow 3}} & ... & \rTo & C_{n} & \rTo^{q_{n\rightarrow n+1}} & ...
%\end{diagram}}}%
%BeginExpansion
\begin{diagram}
C_{1} & \rTo^{q_{1\rightarrow 2}} & C_{2} & \rTo^{q_{2\rightarrow 3}} & ... & \rTo & C_{n} & \rTo^{q_{n\rightarrow n+1}} & ...
\end{diagram}%
%EndExpansion
\right) ,
\end{equation*}%
where%
\begin{equation*}
C_{n}=\left[ Hom_{\mathbf{Ab}}\left( G,T\right) \right] ^{n+1}=\left[ Hom_{%
\mathbf{Ab}}\left( G,T\right) \right] ^{\left\{ 0,1,2,...,n\right\} },
\end{equation*}%
and%
\begin{eqnarray*}
q_{n\rightarrow n+1}\left( \varphi _{0},\varphi _{1},...\varphi
_{n-1},\varphi _{n}\right) &=&\left( \varphi _{0},\varphi _{1},...\varphi
_{n-1},\varphi _{n},\varphi _{0}\right) , \\
\varphi _{i} &\in &Hom_{\mathbf{Ab}}\left( G,T\right) .
\end{eqnarray*}%
One gets a sequence of natural isomorphisms:

\begin{eqnarray*}
&&\left[ Hom_{\mathbf{Ab}}\left( G,T\right) \right] ^{X}%
%TCIMACRO{\TeXButton{ISO}{\simeq}}%
%BeginExpansion
\simeq%
%EndExpansion
\underrightarrow{\lim }\left( 
%TCIMACRO{%
%\TeXButton{Sequence}{\begin{diagram}
%C_{1} & \rTo^{q_{1\rightarrow 2}} & C_{2} & \rTo^{q_{2\rightarrow 3}} & ... & \rTo & C_{n} & \rTo^{q_{n\rightarrow n+1}} & ...
%\end{diagram}}}%
%BeginExpansion
\begin{diagram}
C_{1} & \rTo^{q_{1\rightarrow 2}} & C_{2} & \rTo^{q_{2\rightarrow 3}} & ... & \rTo & C_{n} & \rTo^{q_{n\rightarrow n+1}} & ...
\end{diagram}%
%EndExpansion
\right) 
%TCIMACRO{\TeXButton{ISO}{\simeq} }%
%BeginExpansion
\simeq
%EndExpansion
\\
&&\underrightarrow{\lim }\left( {\scriptsize 
%TCIMACRO{%
%\TeXButton{Sequence}{\begin{diagram}[size=2.0em,textflow]
%Hom_{\mathbf{Ab}}\left( B_{1},T\right) & \rTo^{Hom_{\mathbf{Ab}}\left( r_{1\leftarrow 2},T\right)} & Hom_{\mathbf{Ab}}\left( B_{1},T\right) & \rTo^{Hom_{\mathbf{Ab}}\left( r_{2\leftarrow 3},T\right)} & ... & \rTo & Hom_{\mathbf{Ab}}\left( B_{n},T\right) & \rTo & ...\\
%\end{diagram}}}%
%BeginExpansion
\begin{diagram}[size=2.0em,textflow]
Hom_{\mathbf{Ab}}\left( B_{1},T\right) & \rTo^{Hom_{\mathbf{Ab}}\left( r_{1\leftarrow 2},T\right)} & Hom_{\mathbf{Ab}}\left( B_{1},T\right) & \rTo^{Hom_{\mathbf{Ab}}\left( r_{2\leftarrow 3},T\right)} & ... & \rTo & Hom_{\mathbf{Ab}}\left( B_{n},T\right) & \rTo & ...\\
\end{diagram}%
%EndExpansion
}\right) \\
&&%
%TCIMACRO{\TeXButton{ISO}{\simeq}}%
%BeginExpansion
\simeq%
%EndExpansion
Hom_{\mathbf{Pro}\left( \mathbf{Ab}\right) }\left( \mathbf{B},T\right) ,
\end{eqnarray*}%
where%
\begin{eqnarray*}
\mathbf{I} &=&\left( 1\longleftarrow 2\longleftarrow 3\longleftarrow
...\longleftarrow n\longleftarrow ...\right) , \\
\mathbf{B} &\mathbf{=}&\left( B_{i}\right) _{i\in \mathbf{I}}=\left( 
%TCIMACRO{%
%\TeXButton{Sequence}{\begin{diagram}
%B_{1} & \lTo^{r_{1\leftarrow 2}} & B_{2} & \lTo^{r_{2\leftarrow 3}} & B_{3} & \lTo & ... & \lTo & B_{n} & \lTo^{r_{n\leftarrow n+1}} & ...
%\end{diagram}}}%
%BeginExpansion
\begin{diagram}
B_{1} & \lTo^{r_{1\leftarrow 2}} & B_{2} & \lTo^{r_{2\leftarrow 3}} & B_{3} & \lTo & ... & \lTo & B_{n} & \lTo^{r_{n\leftarrow n+1}} & ...
\end{diagram}%
%EndExpansion
\right) , \\
B_{n} &=&G^{n+1}=G^{\left\{ 0,1,2,...,n\right\} },
\end{eqnarray*}%
and%
\begin{eqnarray*}
r_{n\leftarrow n+1}\left( g_{0},g_{1},...g_{n},g_{n+1}\right) &=&\left(
g_{0}+g_{n+1},g_{1},...g_{n-1},g_{n}\right) , \\
g_{i} &\in &G.
\end{eqnarray*}%
We have proved that%
\begin{equation*}
Hom_{\mathbf{Pro}\left( \mathbf{Ab}\right) }\left( \mathbf{B},\bullet
\right) 
%TCIMACRO{\TeXButton{ISO}{\simeq}}%
%BeginExpansion
\simeq%
%EndExpansion
Hom_{\mathbf{Pro}\left( \mathbf{Ab}\right) }\left( pro\text{-}H_{0}\left(
X,G\right) ,\bullet \right)
\end{equation*}%
in $\mathbf{Set}^{\mathbf{Ab}}$, therefore $\mathbf{B}%
%TCIMACRO{\TeXButton{ISO}{\simeq}}%
%BeginExpansion
\simeq%
%EndExpansion
pro$-$H_{0}\left( X,G\right) $ in $\mathbf{Pro}\left( \mathbf{Ab}\right) $.

It remains to show that $\mathbf{B}$ is \textbf{not a rudimentary}
pro-object. Assume on the contrary that $\mathbf{B}%
%TCIMACRO{\TeXButton{ISO}{\simeq}}%
%BeginExpansion
\simeq%
%EndExpansion
Z$ where $Z\in \mathbf{Ab}$. I follows from Proposition \ref%
{Prop-Trivial-pro-object} that there exists a homomorphism%
\begin{equation*}
\tau _{0}:B_{i_{0}}\longrightarrow Z,
\end{equation*}%
satisfying the property: for any morphism $s:i\rightarrow i_{0}$, there
exist a morphism $\sigma :Z\rightarrow B_{i}$ and a morphism $t:j\rightarrow
i$ satisfying%
\begin{eqnarray*}
\tau _{0}\circ B\left( s\right) \circ \sigma &=&\mathbf{1}_{Z}, \\
\sigma \circ \tau _{0}\circ B\left( s\right) \circ B\left( t\right)
&=&B\left( t\right) .
\end{eqnarray*}%
Take $s=\left( i_{0}\leftarrow i_{0}+1\right) $. Choose a \textbf{nonzero}
element $a\in \ker B\left( s\right) $, say%
\begin{equation*}
a=\left( g,0,...,0,-g\right) ,~g\neq 0.
\end{equation*}%
Since $B\left( t\right) $ is surjective, choose $b\in B_{j}$ with $\left[
B\left( t\right) \right] \left( b\right) =a$. Apply the second equation from
above:%
\begin{eqnarray*}
\left[ \sigma \circ \tau _{0}\circ B\left( s\right) \circ B\left( t\right) %
\right] \left( b\right) &=&\left[ B\left( t\right) \right] \left( b\right) ,
\\
\left[ \sigma \circ \tau _{0}\circ B\left( s\right) \right] \left( a\right)
&=&a, \\
0 &=&a\neq 0.
\end{eqnarray*}%
Contradiction.
\end{proof}

\begin{example}
\label{Ex-Finite-circle}Let $X$ be the \textbf{pseudocircle}, i.e. the $4$%
-point topological space%
\begin{equation*}
X=\left\{ a,b,c,d\right\}
\end{equation*}%
with the topology%
\begin{equation*}
\tau =\left\{ \varnothing ,\left\{ a,b,c,d\right\} ,\left\{ a\right\}
,\left\{ c\right\} ,\left\{ a,c\right\} ,\left\{ a,b,c\right\} ,\left\{
a,c,d\right\} \right\} .
\end{equation*}%
This space can be also described as the \textbf{non-Hausdorff suspension} 
\cite[Section 8, p. 472]{McCord-MR0196744} $\mathfrak{S}\left( Y\right) $,
where%
\begin{equation*}
X\supseteq Y=\left\{ a,c\right\} 
%TCIMACRO{\TeXButton{ISO}{\simeq}}%
%BeginExpansion
\simeq%
%EndExpansion
S^{0}.
\end{equation*}%
Let again $\mathcal{A}=G_{\#}$ be the constant cosheaf ($\left\{ 0\right\}
\neq G\in \mathbf{Ab}$). Then%
\begin{equation*}
\check{H}_{n}\left( X,\mathcal{A}\right) 
%TCIMACRO{\TeXButton{ISO}{\simeq}}%
%BeginExpansion
\simeq%
%EndExpansion
H_{n}\left( X,\mathcal{A}\right) 
%TCIMACRO{\TeXButton{ISO}{\simeq}}%
%BeginExpansion
\simeq%
%EndExpansion
H_{n}^{sing}\left( X,G\right) =\left\{ 
\begin{array}{ccc}
G & \text{if} & n=0,1, \\ 
0 & \text{if} & n\neq 0,1.%
\end{array}%
\right.
\end{equation*}%
where $H_{\bullet }^{sing}$ is the ordinary singular homology. Notice that
the pro-homology $pro$-$H_{1}\left( X,G\right) $ is zero (Remark \ref%
{Rem-Shape-locally-finite}) and%
\begin{equation*}
\mathbf{0}=pro\text{-}H_{1}\left( X,G\right) \ncong H_{1}\left( X,\mathcal{A}%
\right) .
\end{equation*}%
The reason is that we \textbf{could not} apply Conjecture \ref%
{Conj-Satellites-H}(\ref{Conj-Satellites-H-paracompact}) because $X$ is 
\textbf{not} Hausdorff.
\end{example}

\begin{proof}
Let%
\begin{equation*}
\mathcal{U}=\left\{ U_{0,1,2,3}\longrightarrow X\right\} =\left\{ \left\{
a\right\} ,\left\{ c\right\} ,\left\{ a,b,c\right\} ,\left\{ a,c,d\right\}
\right\} .
\end{equation*}%
Define the bicomplex%
\begin{equation*}
C_{s,t}=\check{C}_{s}\left( \mathcal{U},\mathcal{P}_{t}\right)
\end{equation*}%
where $\mathcal{P}_{\bullet }\rightarrow \mathcal{A}$ is a \textbf{qis}
(Notation \ref{Not-Derived-category}) from a quasi-projective complex to $%
\mathcal{A}$, considered as a complex concentrated in degree $0$ (i.e. $%
\mathcal{P}_{\bullet }\rightarrow \mathcal{A}$ is a quasi-projective
resolution of $\mathcal{A}$). Since $\mathbf{Pro}\left( \mathbf{Ab}\right) $
is an abelian (Proposition \ref{Prop-Pro(K)-abelian}) category, we can apply
Theorem \ref{Th-Spectral-sequence} in order to obtain two spectral sequences
converging to the total complex $Tot_{\bullet }\left( C\right) $. Notice
that $\mathcal{A}|_{U_{i_{1}}\times ...\times U_{i_{s}}}$ is flabby
(Definition \ref{Def-Flabby-cosheaf}) for each $s$, therefore%
\begin{equation*}
H_{t}\left( U_{i_{1}}\times ...\times U_{i_{s}},\mathcal{A}\right) =0
\end{equation*}%
if $t>0$. Calculate the entries in the first spectral sequence:%
\begin{eqnarray*}
^{ver}E_{s,t}^{1} &=&\left\{ 
\begin{array}{ccc}
H_{t}\left( U_{i_{1}}\times ...\times U_{i_{s}},\mathcal{A}\right) =0 & 
\text{if} & t>0 \\ 
\mathcal{A}\left( U_{i_{1}}\times ...\times U_{i_{s}}\right) & \text{if} & 
t=0%
\end{array}%
\right. \\
^{ver}E_{s,t}^{2} &=&\left\{ 
\begin{array}{ccc}
0 & \text{if} & t>0 \\ 
H_{s}\left( \mathcal{U},\mathcal{A}\right) & \text{if} & t=0%
\end{array}%
\right.
\end{eqnarray*}%
It follows that%
\begin{equation*}
H_{n}\left( Tot_{\bullet }\left( C\right) \right) 
%TCIMACRO{\TeXButton{ISO}{\simeq}}%
%BeginExpansion
\simeq%
%EndExpansion
~^{ver}E_{n,0}^{2}%
%TCIMACRO{\TeXButton{ISO}{\simeq}}%
%BeginExpansion
\simeq%
%EndExpansion
H_{n}\left( \mathcal{U},\mathcal{A}\right) .
\end{equation*}%
The second spectral sequence gives%
\begin{eqnarray*}
^{hor}E_{s,t}^{1} &=&\left\{ 
\begin{array}{ccccc}
0 & \text{if} & s>0 & \text{since} & \mathcal{P}_{t}\text{ is
quasi-projective as a (pre)cosheaf} \\ 
\mathcal{P}_{t}\left( X\right) & \text{if} & s=0 & \text{since} & \mathcal{P}%
_{t}\text{ is a cosheaf}%
\end{array}%
\right. \\
^{hor}E_{s,t}^{2} &=&\left\{ 
\begin{array}{ccc}
0 & \text{if} & s>0 \\ 
H_{t}\left( X,\mathcal{A}\right) & \text{if} & s=0%
\end{array}%
\right.
\end{eqnarray*}%
It follows that%
\begin{equation*}
H_{n}\left( Tot_{\bullet }\left( C\right) \right) 
%TCIMACRO{\TeXButton{ISO}{\simeq}}%
%BeginExpansion
\simeq%
%EndExpansion
~^{hor}E_{0,n}^{2}%
%TCIMACRO{\TeXButton{ISO}{\simeq}}%
%BeginExpansion
\simeq%
%EndExpansion
H_{n}\left( X,\mathcal{A}\right) .
\end{equation*}%
Finally%
\begin{equation*}
H_{n}\left( X,\mathcal{A}\right) 
%TCIMACRO{\TeXButton{ISO}{\simeq}}%
%BeginExpansion
\simeq%
%EndExpansion
H_{n}\left( \mathcal{U},\mathcal{A}\right) .
\end{equation*}%
The latter pro-groups (in fact, \emph{rudimentary} pro-groups, i.e. just
ordinary groups) can be easily calculated. It remains to apply \cite[Theorem
2 and the example in \S 5]{McCord-MR0196744}:%
\begin{equation*}
H_{n}\left( \mathcal{U},\mathcal{A}\right) =\left\{ 
\begin{array}{ccc}
G & \text{if} & n=0,1, \\ 
0 & \text{if} & n\neq 0,1;%
\end{array}%
\right. 
%TCIMACRO{\TeXButton{ISO}{\simeq}}%
%BeginExpansion
\simeq%
%EndExpansion
H_{n}^{sing}\left( S^{1},G\right) 
%TCIMACRO{\TeXButton{ISO}{\simeq}}%
%BeginExpansion
\simeq%
%EndExpansion
H_{n}^{sing}\left( X,G\right) .
\end{equation*}
\end{proof}

\section{\label{Sec-Proof-Properties-precosheaves}Proof of Theorem \protect
\ref{Th-Properties-precosheaves}}

\begin{proof}
The category $\mathbf{pCS}\left( \mathbf{E},\mathbf{Pro}\left( k\right)
\right) $ inherits most properties from the category $\mathbf{Pro}\left(
k\right) $, therefore we can apply Proposition \ref%
{Prop-Pro-modules-properties}.\ (\textbf{\ref%
{Th-Properties-precosheaves-abelian-(co)complete}-\ref%
{Th-Properties-precosheaves-limit}}) Follow from Proposition \ref%
{Prop-Pro-modules-properties}(\ref%
{Prop-Pro-modules-properties-abelian-(co)complete}-\ref%
{Prop-Pro-modules-properties-limit}).

(\textbf{\ref{Prop-Pro-modules-properties-left-Kan}}) Let $\mathcal{A}\in 
\mathbf{pCS}\left( \mathbf{E},\mathbf{Pro}\left( k\right) \right) $, and $%
U\in \mathbf{E}$. It follows from Proposition \ref{Prop-Kan-extensions} that%
\begin{eqnarray*}
&&\left[ f^{\dag }\mathcal{A}\right] \left( U\right) 
%TCIMACRO{\TeXButton{ISO}{\simeq}}%
%BeginExpansion
\simeq%
%EndExpansion
\underset{\left( f\left( V\right) \rightarrow U\right) \in f\downarrow U}{%
\underrightarrow{\lim }}\mathcal{A}\left( V\right) , \\
&&\left[ \left( f^{op}\right) ^{\ddag }\mathcal{B}\right] \left( U\right) 
%TCIMACRO{\TeXButton{ISO}{\simeq}}%
%BeginExpansion
\simeq%
%EndExpansion
\underset{\left( f\left( V\right) \rightarrow U\right) \in f\downarrow U}{%
\underleftarrow{\lim }}\mathcal{B}\left( V\right) .
\end{eqnarray*}%
Therefore%
\begin{equation*}
\left\langle f^{\dag }\mathcal{A},T\right\rangle \left( U\right)
=\left\langle \left[ f^{\dag }\mathcal{A}\right] \left( U\right)
,T\right\rangle =\left\langle \underset{\left( f\left( V\right) \rightarrow
U\right) \in f\downarrow U}{\underrightarrow{\lim }}\mathcal{A}\left(
V\right) ,T\right\rangle 
%TCIMACRO{\TeXButton{ISO}{\simeq}}%
%BeginExpansion
\simeq%
%EndExpansion
\end{equation*}%
\begin{equation*}
%TCIMACRO{\TeXButton{ISO}{\simeq}}%
%BeginExpansion
\simeq%
%EndExpansion
\underset{\left( f\left( V\right) \rightarrow U\right) \in f\downarrow U}{%
\underleftarrow{\lim }}\left\langle \mathcal{A}\left( V\right)
,T\right\rangle 
%TCIMACRO{\TeXButton{ISO}{\simeq}}%
%BeginExpansion
\simeq%
%EndExpansion
\left[ \left( f^{op}\right) ^{\ddag }\left\langle \mathcal{A},T\right\rangle %
\right] \left( U\right) .
\end{equation*}

(\textbf{\ref{Th-Properties-precosheaves-Zero}}-\textbf{\ref%
{Th-Properties-precosheaves-Exact}}) Follow from Proposition \ref%
{Prop-Pro-modules-properties}(\ref{Prop-Pro-modules-properties-Zero}-\ref%
{Prop-Pro-modules-properties-Exact}).

(\textbf{\ref{Th-Properties-precosheaves-AB4}}-\textbf{\ref%
{Th-Properties-precosheaves-AB5*}}) Follow from Proposition \ref%
{Prop-Pro-modules-properties}(\ref{Prop-Pro-modules-properties-AB4}-\ref%
{Prop-Pro-modules-properties-AB5*}).

(\textbf{\ref{Th-Properties-precosheaves-cogenerators}}) Let%
\begin{equation*}
\left( \varphi :\mathcal{C}\twoheadrightarrow \mathcal{D}\right) \in \mathbf{%
pCS}\left( \mathbf{E},\mathbf{Pro}\left( k\right) \right)
\end{equation*}%
be a non-trivial epimorphism. It follows that 
\begin{equation*}
\varphi \left( U\right) :\mathcal{C}\left( U\right) \longrightarrow \mathcal{%
D}\left( U\right)
\end{equation*}%
is an epimorphism in $\mathbf{Pro}\left( k\right) $ for any $U\in \mathbf{E}$%
, and that there exists a $V\in \mathbf{E}$, such that 
\begin{equation*}
\varphi \left( V\right) :\mathcal{C}\left( V\right) \longrightarrow \mathcal{%
D}\left( V\right)
\end{equation*}%
is \textbf{non-trivial} epimorphism. Due to Proposition \ref%
{Prop-Pro-modules-properties}(\ref{Prop-Pro-modules-properties-cogenerators}%
), there exist an $A\in \mathfrak{G}\subseteq \mathbf{Pro}\left( k\right) $,
and a morphism%
\begin{equation*}
\left( \psi :\mathcal{C}\left( V\right) \longrightarrow A\right) \in \mathbf{%
Pro}\left( k\right) ,
\end{equation*}%
which does \textbf{not} factor through $\mathcal{D}\left( V\right) $. The
morphism%
\begin{equation*}
\xi :\mathcal{C}\longrightarrow A^{V},
\end{equation*}%
which corresponds to $\psi $ under the adjunction%
\begin{equation*}
Hom_{\mathbf{pCS}\left( \mathbf{E},\mathbf{Pro}\left( k\right) \right)
}\left( \mathcal{C},A^{V}\right) 
%TCIMACRO{\TeXButton{ISO}{\simeq}}%
%BeginExpansion
\simeq%
%EndExpansion
Hom_{\mathbf{Pro}\left( k\right) }\left( \mathcal{C}\left( V\right) ,A\right)
\end{equation*}%
does \textbf{not} factor through $\mathcal{D}$.
\end{proof}

\section{\label{Sec-Proof-precosheaf-homology}Proof of Theorem \protect\ref%
{Th-Precosheaf-homology}}

\begin{proof}
(\textbf{\ref{Th-Precosheaf-homology-surjection}}) Let $\mathbf{D}^{\delta }$
be the discrete category with the same set of objects as $\mathbf{D}$:%
\begin{equation*}
Ob\left( \mathbf{D}^{\delta }\right) =Ob\left( \mathbf{D}\right) ,
\end{equation*}%
and let $f:\mathbf{D}^{\delta }\longrightarrow \mathbf{D}$ be the evident
functor, identical on objects. Define the precosheaf $\mathcal{P}\left( 
\mathcal{B}\right) $ by the following:%
\begin{equation*}
\mathcal{P}\left( \mathcal{B}\right) 
%TCIMACRO{\TeXButton{assigned}{{:=}}}%
%BeginExpansion
{:=}%
%EndExpansion
f^{\dag }\mathcal{G}\left( f_{\ast }\left( \mathcal{B}\right) \right) ,
\end{equation*}%
where $\mathcal{F}$ is the functor from Proposition \ref%
{Prop-Pro-modules-properties}(\ref%
{Prop-Pro-modules-properties-quasi-projective}), and%
\begin{equation*}
\mathcal{G}\left( U\right) 
%TCIMACRO{\TeXButton{assigned}{{:=}}}%
%BeginExpansion
{:=}%
%EndExpansion
\mathbf{F}\left( \left[ f_{\ast }\left( \mathcal{B}\right) \right] \left(
U\right) \right) =\mathbf{F}\left( \mathcal{B}\left( U\right) \right)
\end{equation*}%
for $U\in \mathbf{D}$. It follows from Proposition \ref{Prop-Kan-extensions}
that%
\begin{equation*}
\mathcal{P}\left( \mathcal{B}\right) \left( U\right)
=\dbigoplus\limits_{V\rightarrow U}\mathbf{F}\left( \mathcal{B}\left(
V\right) \right) .
\end{equation*}%
The morphism%
\begin{equation*}
\mathcal{G}\left( f_{\ast }\left( \mathcal{B}\right) \right) \longrightarrow
f_{\ast }\left( \mathcal{B}\right)
\end{equation*}%
induces, by adjunction, the desired homomorphism%
\begin{equation*}
\pi :\mathcal{P}\left( \mathcal{B}\right) =f^{\dag }\mathcal{G}\left(
f_{\ast }\left( \mathcal{B}\right) \right) \longrightarrow \mathcal{B}.
\end{equation*}%
Indeed, $\mathcal{G}\left( f_{\ast }\left( \mathcal{B}\right) \right) $ is
quasi-projective due to Example \ref{Ex-Quasi-projective-discrete-category},
and $\mathcal{P}\left( \mathcal{B}\right) $ is quasi-projective due to
Proposition \ref{Prop-Left-Kan-quasi-projective}. For any $U\in \mathbf{D}$,
the composition%
\begin{equation*}
\mathbf{F}\left( \mathcal{B}\left( U\right) \right) \hookrightarrow \mathcal{%
P}\left( \mathcal{B}\right) \left( U\right) =\dbigoplus\limits_{V\rightarrow
U}\mathbf{F}\left( \mathcal{B}\left( V\right) \right) \overset{\pi \left(
U\right) }{\longrightarrow }\mathcal{B}\left( U\right)
\end{equation*}%
is the epimorphism from Proposition \ref{Prop-Pro-modules-properties}(\ref%
{Prop-Pro-modules-properties-quasi-projective}), therefore $\pi $ is an
epimorphism as well.

(\textbf{\ref{Th-Precosheaf-homology-F-projective}}) We have just proved the
condition (\ref{Def-F-projective-generating}) of Definition \ref%
{Def-F-projective}. It remains to check the other two conditions.

Given a short exact sequence%
\begin{equation*}
0\longrightarrow \mathcal{B}^{\prime }\longrightarrow \mathcal{B}%
\longrightarrow \mathcal{B}^{\prime \prime }\longrightarrow 0
\end{equation*}%
of precosheaves, assume that%
\begin{equation*}
\mathcal{B},\mathcal{B}^{\prime \prime }\in Q\left( \mathbf{pCS}\left( 
\mathbf{D},\mathbf{Pro}\left( k\right) \right) \right) .
\end{equation*}%
Therefore, for any injective $T\in \mathbf{Mod}\left( k\right) $, the
sequence%
\begin{equation*}
0\longrightarrow \left\langle \mathcal{B}^{\prime \prime },T\right\rangle
\longrightarrow \left\langle \mathcal{B},T\right\rangle \longrightarrow
\left\langle \mathcal{B}^{\prime },T\right\rangle \longrightarrow 0
\end{equation*}%
is exact. Since $\left\langle \mathcal{B}^{\prime \prime },T\right\rangle $
and $\left\langle \mathcal{B},T\right\rangle $ are injective presheaves, it
follows that the sequence above is \textbf{split} exact, and%
\begin{equation*}
\left\langle \mathcal{B},T\right\rangle 
%TCIMACRO{\TeXButton{ISO}{\simeq}}%
%BeginExpansion
\simeq%
%EndExpansion
\left\langle \mathcal{B}^{\prime },T\right\rangle \times \left\langle 
\mathcal{B}^{\prime \prime },T\right\rangle .
\end{equation*}%
The presheaf $\left\langle \mathcal{B}^{\prime },T\right\rangle $, being a
direct summand of the injective presheaf $\left\langle \mathcal{B}%
,T\right\rangle $, is injective (for any injective $T$), therefore $\mathcal{%
B}^{\prime }$ is quasi-projective. The condition (\ref%
{Def-F-projective-closed-by-extensions}) of Definition \ref{Def-F-projective}
is proved!

Let now $R\subseteq h_{U}$ be a sieve. Since both $H^{0}\left( R,\bullet
\right) $ and $\check{H}^{0}$ are \textbf{additive} functors%
\begin{equation*}
\mathbf{pS}\left( X,\mathbf{Pro}\left( k\right) \right) \longrightarrow 
\mathbf{Mod}\left( k\right) ,
\end{equation*}%
the sequences of $k$-modules%
\begin{eqnarray*}
0 &\longrightarrow &H^{0}\left( R,\left\langle \mathcal{B}^{\prime \prime
},T\right\rangle \right) \longrightarrow H^{0}\left( R,\left\langle \mathcal{%
B}^{\prime \prime },T\right\rangle \right) \longrightarrow H^{0}\left(
R,\left\langle \mathcal{B}^{\prime \prime },T\right\rangle \right)
\longrightarrow 0, \\
0 &\longrightarrow &\check{H}^{0}\left( U,\left\langle \mathcal{A}^{\prime
\prime },T\right\rangle \right) \longrightarrow \check{H}^{0}\left(
U,\left\langle \mathcal{A}^{\prime \prime },T\right\rangle \right)
\longrightarrow \check{H}^{0}\left( U,\left\langle \mathcal{A}^{\prime
\prime },T\right\rangle \right) \longrightarrow 0,
\end{eqnarray*}%
are exact (in fact, split exact). It follows from Proposition \ref%
{Prop-Pro-modules-properties}(\ref{Prop-Pro-modules-properties-Exact}) that
the corresponding sequences of pro-modules%
\begin{eqnarray*}
0 &\longrightarrow &H_{0}\left( R,\mathcal{A}^{\prime }\right)
\longrightarrow H_{0}\left( R,\mathcal{A}\right) \longrightarrow H_{0}\left(
R,\mathcal{A}^{\prime \prime }\right) \longrightarrow 0, \\
0 &\longrightarrow &\check{H}_{0}\left( U,\mathcal{A}^{\prime }\right)
\longrightarrow \check{H}_{0}\left( U,\mathcal{A}\right) \longrightarrow 
\check{H}_{0}\left( U,\mathcal{A}^{\prime \prime }\right) \longrightarrow 0,
\end{eqnarray*}%
are exact, because%
\begin{eqnarray*}
&&\left\langle H_{0}\left( R,\mathcal{E}\right) ,T\right\rangle 
%TCIMACRO{\TeXButton{ISO}{\simeq}}%
%BeginExpansion
\simeq%
%EndExpansion
H^{0}\left( R,\left\langle \mathcal{E},T\right\rangle \right) \\
&&\left\langle \check{H}_{0}\left( R,\mathcal{E}\right) ,T\right\rangle 
%TCIMACRO{\TeXButton{ISO}{\simeq}}%
%BeginExpansion
\simeq%
%EndExpansion
\check{H}^{0}\left( R,\left\langle \mathcal{E},T\right\rangle \right)
\end{eqnarray*}%
for any precosheaf $\mathcal{E}$ (see the statement (\ref%
{Th-Precosheaf-homology-pairing}) of our theorem).

(\textbf{\ref{Th-Precosheaf-Cech-homology}}) Choose a quasi-projective
resolution%
\begin{equation*}
0\longleftarrow \mathcal{A}\longleftarrow \mathcal{P}_{0}\longleftarrow 
\mathcal{P}_{1}\longleftarrow \mathcal{P}_{2}\longleftarrow
...\longleftarrow \mathcal{P}_{n}\longleftarrow ...
\end{equation*}%
and construct a bicomplex%
\begin{equation*}
X_{s,t}=\check{C}_{s}\left( \left\{ U_{i}\rightarrow U\right\} ,\mathcal{P}%
_{t}\right) .
\end{equation*}%
Due to Theorem \ref{Th-Spectral-sequence}, one gets two spectral sequences%
\begin{equation*}
^{ver}E_{s,t}^{r},~^{hor}E_{s,t}^{r}\implies H_{s+t}\left( Tot_{\bullet
}\left( X\right) \right) .
\end{equation*}%
Apply $\left\langle \bullet ,T\right\rangle $ where $T\in \mathbf{Mod}\left(
k\right) $ is an arbitrary injective module. It follows that%
\begin{equation*}
\left\langle \mathcal{P}_{t},T\right\rangle \in \mathbf{pS}\left( X,\mathbf{%
Mod}\left( k\right) \right)
\end{equation*}%
are injective \textbf{presheaves} for all $t$. Due to \cite[Corollary 1.4.2]%
{Artin-GT} or \cite[Theorem 2.2.3]{Tamme-MR1317816}, and the fact that%
\begin{equation*}
\left\langle \check{C}_{s}\left( \left\{ U_{i}\rightarrow U\right\} ,%
\mathcal{P}_{t}\right) ,T\right\rangle 
%TCIMACRO{\TeXButton{ISO}{\simeq}}%
%BeginExpansion
\simeq%
%EndExpansion
\check{C}^{s}\left( \left\{ U_{i}\rightarrow U\right\} ,\left\langle 
\mathcal{P}_{t},T\right\rangle \right)
\end{equation*}%
the sequence%
\begin{equation*}
\left\langle ~^{hor}E_{0,t}^{0},T\right\rangle \longrightarrow \left\langle
~^{hor}E_{1,t}^{0},T\right\rangle \longrightarrow \left\langle
~^{hor}E_{2,t}^{0},T\right\rangle \longrightarrow ...\longrightarrow
\left\langle ~^{hor}E_{s,t}^{0},T\right\rangle \longrightarrow ...
\end{equation*}%
is exact for all $s>0$ (and all $T$!), therefore the sequence%
\begin{equation*}
~^{hor}E_{0,t}^{0}\longleftarrow ~^{hor}E_{1,t}^{0}\longleftarrow
~^{hor}E_{2,t}^{0}\longleftarrow ...\longleftarrow
~^{hor}E_{s,t}^{0}\longleftarrow ...
\end{equation*}%
is exact for all $s>0$, and%
\begin{equation*}
^{hor}E_{s,t}^{1}=0
\end{equation*}%
if $s>0$. The spectral sequence $^{hor}E^{r}$ degenerates from $E^{2}$ on,
implying%
\begin{equation*}
H_{n}\left( Tot_{\bullet }\left( X\right) \right) 
%TCIMACRO{\TeXButton{ISO}{\simeq}}%
%BeginExpansion
\simeq%
%EndExpansion
~^{hor}E_{0,n}^{2}%
%TCIMACRO{\TeXButton{ISO}{\simeq}}%
%BeginExpansion
\simeq%
%EndExpansion
L_{n}H_{0}\left( \left\{ U_{i}\rightarrow U\right\} ,\mathcal{A}\right) .
\end{equation*}

On the other hand, since products are exact in $\mathbf{Mod}\left( k\right) $%
, one gets exact (for $t>0$) sequences%
\begin{equation*}
\left\langle ~^{ver}E_{s,0}^{0},T\right\rangle \longrightarrow \left\langle
~^{ver}E_{s,1}^{0},T\right\rangle \longrightarrow \left\langle
~^{ver}E_{s,2}^{0},T\right\rangle \longrightarrow ...\longrightarrow
\left\langle ~^{ver}E_{s,t}^{0},T\right\rangle \longrightarrow ...
\end{equation*}%
in $\mathbf{Mod}\left( k\right) $, and exact (for $t>0$) sequences%
\begin{equation*}
~^{ver}E_{s,0}^{0}\longleftarrow ~^{ver}E_{s,1}^{0}\longleftarrow
~^{ver}E_{s,2}^{0}\longleftarrow ...\longleftarrow
~^{ver}E_{s,t}^{0}\longleftarrow ...
\end{equation*}%
It follows that $^{ver}E_{s,t}^{1}=0$ for $t>0$, and the sequence $%
^{ver}E^{r}$ degenerates from $E^{2}$ on, therefore%
\begin{equation*}
L_{n}H_{0}\left( \left\{ U_{i}\rightarrow U\right\} ,\mathcal{A}\right) 
%TCIMACRO{\TeXButton{ISO}{\simeq}}%
%BeginExpansion
\simeq%
%EndExpansion
H_{n}\left( Tot_{\bullet }\left( X\right) \right) 
%TCIMACRO{\TeXButton{ISO}{\simeq}}%
%BeginExpansion
\simeq%
%EndExpansion
~^{ver}E_{n,0}^{2}%
%TCIMACRO{\TeXButton{ISO}{\simeq}}%
%BeginExpansion
\simeq%
%EndExpansion
H_{n}\left( \left\{ U_{i}\rightarrow U\right\} ,\mathcal{A}\right) .
\end{equation*}

Apply $\underset{R\in Cov\left( U\right) }{\underleftarrow{\lim }}$ to the
bicomplexes $X_{\bullet ,\bullet }$ to get the bicomplex $\check{X}_{\bullet
,\bullet }$. The two spectral sequences for $\check{X}_{\bullet ,\bullet }$
degenerate form $E^{2}$ on, giving the desired isomorphisms.

(\textbf{\ref{Th-Precosheaf-homology-pairing}}) See the proof of (\ref%
{Th-Precosheaf-Cech-homology}). It remains only to remind (Proposition \ref%
{Prop-Pro-objects-properties} (\ref%
{Prop-Pro-objects-properties-Convert-limits-to-colimits})) that $%
\left\langle \bullet ,T\right\rangle $ converts cofiltered limits $%
\underleftarrow{\lim }$ into filtered colimits $\underrightarrow{\lim }$.
\end{proof}

\section{\label{Sec-Proof-Properties-cosheaves}Proof of Theorem \protect\ref%
{Th-Properties-cosheaves}}

\begin{proof}
(\textbf{\ref{Th-Properties-cosheaves-abelian-(co)cocomplete}})

\begin{itemize}
\item \emph{Kernels}. Given a morphism $f$ of cosheaves 
\begin{equation*}
f:\mathcal{A}\longrightarrow \mathcal{B},
\end{equation*}%
let 
\begin{equation*}
\mathcal{K=}\ker \left( \iota f:\iota \mathcal{A}\longrightarrow \iota 
\mathcal{B}\right)
\end{equation*}%
in $\mathbf{pCS}\left( X,\mathbf{Pro}\left( k\right) \right) $. Then, for
any $\mathcal{C}\in \mathbf{CS}\left( X,\mathbf{Pro}\left( k\right) \right) $%
,%
\begin{eqnarray*}
&&Hom_{\mathbf{CS}\left( X,\mathbf{Pro}\left( k\right) \right) }\left( 
\mathcal{C},\mathcal{K}_{\#}\right) 
%TCIMACRO{\TeXButton{ISO}{\simeq}}%
%BeginExpansion
\simeq%
%EndExpansion
Hom_{\mathbf{pCS}\left( X,\mathbf{Pro}\left( k\right) \right) }\left( 
\mathcal{C},\mathcal{K}\right) 
%TCIMACRO{\TeXButton{ISO}{\simeq} }%
%BeginExpansion
\simeq
%EndExpansion
\\
&&%
%TCIMACRO{\TeXButton{ISO}{\simeq}}%
%BeginExpansion
\simeq%
%EndExpansion
\ker \left( Hom_{\mathbf{CS}\left( X,\mathbf{Pro}\left( k\right) \right)
}\left( \mathcal{C},\mathcal{A}\right) \longrightarrow Hom_{\mathbf{CS}%
\left( X,\mathbf{Pro}\left( k\right) \right) }\left( \mathcal{C},\mathcal{B}%
\right) \right) ,
\end{eqnarray*}%
therefore $\mathcal{K}_{\#}$ is the kernel of $f$ in $\mathbf{CS}\left( X,%
\mathbf{Pro}\left( k\right) \right) $.

\item \emph{Cokernels}. The cokernel of $\iota f$ is clearly a cosheaf,
therefore%
\begin{equation*}
%TCIMACRO{\TeXButton{coker}{\coker}}%
%BeginExpansion
\coker%
%EndExpansion
f%
%TCIMACRO{\TeXButton{assigned}{{:=}}}%
%BeginExpansion
{:=}%
%EndExpansion
%TCIMACRO{\TeXButton{coker}{\coker}}%
%BeginExpansion
\coker%
%EndExpansion
\iota f
\end{equation*}%
is the desired cokernel.

\item \emph{Products}. Let 
\begin{equation*}
\left( \mathcal{A}_{i}\right) _{i\in I}
\end{equation*}%
be a family of cosheaves, and let%
\begin{equation*}
\mathcal{B}%
%TCIMACRO{\TeXButton{assigned}{{:=}}}%
%BeginExpansion
{:=}%
%EndExpansion
\dprod\limits_{i\in I}\iota \left( \mathcal{A}_{i}\right)
\end{equation*}%
in $\mathbf{pCS}\left( X,\mathbf{Pro}\left( k\right) \right) $. Then, for
any $\mathcal{C}\in \mathbf{CS}\left( X,\mathbf{Pro}\left( k\right) \right) $%
,%
\begin{eqnarray*}
&&Hom_{\mathbf{CS}\left( X,\mathbf{Pro}\left( k\right) \right) }\left( 
\mathcal{C},\mathcal{B}_{\#}\right) 
%TCIMACRO{\TeXButton{ISO}{\simeq}}%
%BeginExpansion
\simeq%
%EndExpansion
Hom_{\mathbf{pCS}\left( X,\mathbf{Pro}\left( k\right) \right) }\left( 
\mathcal{C},\mathcal{B}\right) 
%TCIMACRO{\TeXButton{ISO}{\simeq} }%
%BeginExpansion
\simeq
%EndExpansion
\\
&&%
%TCIMACRO{\TeXButton{ISO}{\simeq}}%
%BeginExpansion
\simeq%
%EndExpansion
\dprod\limits_{i\in I}Hom_{\mathbf{CS}\left( X,\mathbf{Pro}\left( k\right)
\right) }\left( \mathcal{C},\mathcal{A}_{i}\right) ,
\end{eqnarray*}%
therefore $\mathcal{B}_{\#}$ is the product of $\mathcal{A}_{i}$ in $\mathbf{%
CS}\left( X,\mathbf{Pro}\left( k\right) \right) $.

\item \emph{Coproducts}. The coproduct%
\begin{equation*}
\dbigoplus\limits_{i\in I}\iota \left( \mathcal{A}_{i}\right)
\end{equation*}%
is clearly a cosheaf, and can therefore serve as a coproduct in $\mathbf{CS}%
\left( X,\mathbf{Pro}\left( k\right) \right) $.

\item \emph{Limits} $\underleftarrow{\lim }$ are built as combinations of
products and kernels. The category $\mathbf{CS}\left( X,\mathbf{Pro}\left(
k\right) \right) $ is \emph{complete}.

\item \emph{Colimits} $\underrightarrow{\lim }$ are built as combinations of
coproducts and cokernels. The category $\mathbf{CS}\left( X,\mathbf{Pro}%
\left( k\right) \right) $ is \emph{cocomplete}.

\item \emph{Images and coimages}. Let 
\begin{equation*}
\left( f:\mathcal{A}\longrightarrow \mathcal{B}\right) \in \mathbf{CS}\left(
X,\mathbf{Pro}\left( k\right) \right) .
\end{equation*}%
Consider the diagram of (pre)cosheaves%
\begin{equation*}
{\scriptsize 
%TCIMACRO{%
%\TeXButton{(Co)image}{\begin{diagram}
%\ker \left( \iota f\right) & \rTo^{h} & \mathcal{A} & \rTo^{f} & \mathcal{B} & \rTo^{g} & \coker\left( \iota f\right)  \\
%\uTo & \ruTo^{h_{\#}} & \dTo & & \uTo & & \dTo^{=}\\
%\left[ \left( \ker \left( \iota f\right) \right) _{\#}{=}\ker f\right] & & \left[ {\coker}\left( h\right)=coim\left( \iota f\right) \right] & \rTo^{\simeq}_{\varphi} & \left[ \ker \left( \iota g\right) =im\left( \iota f\right) \right] &  & {\coker}f \\
% & & \uTo & & \uTo \\
% & & \left[ \left( {\coker}\left( h\right)\right) _{\#}{\simeq}{\coker}\left( h_{\#}\right) =coim\left( f\right) \right] & \rTo^{\simeq}_{\varphi _{\#}} & \left[ \left( \ker \left( \iota g\right) \right) _{\#}=\ker g=im\left( f\right) \right] \\
%\end{diagram}}}%
%BeginExpansion
\begin{diagram}
\ker \left( \iota f\right) & \rTo^{h} & \mathcal{A} & \rTo^{f} & \mathcal{B} & \rTo^{g} & \coker\left( \iota f\right)  \\
\uTo & \ruTo^{h_{\#}} & \dTo & & \uTo & & \dTo^{=}\\
\left[ \left( \ker \left( \iota f\right) \right) _{\#}{=}\ker f\right] & & \left[ {\coker}\left( h\right)=coim\left( \iota f\right) \right] & \rTo^{\simeq}_{\varphi} & \left[ \ker \left( \iota g\right) =im\left( \iota f\right) \right] &  & {\coker}f \\
 & & \uTo & & \uTo \\
 & & \left[ \left( {\coker}\left( h\right)\right) _{\#}{\simeq}{\coker}\left( h_{\#}\right) =coim\left( f\right) \right] & \rTo^{\simeq}_{\varphi _{\#}} & \left[ \left( \ker \left( \iota g\right) \right) _{\#}=\ker g=im\left( f\right) \right] \\
\end{diagram}%
%EndExpansion
}
\end{equation*}%
The cosheafification functor $\left( {}\right) _{\#}$ is exact, due to
Theorem \ref{Th-Cosheafification} (\ref{Th-Cosheafification-plusplus-exact}%
), therefore%
\begin{equation*}
\left( 
%TCIMACRO{\TeXButton{coker}{\coker}}%
%BeginExpansion
\coker%
%EndExpansion
\left( h\right) \right) _{\#}%
%TCIMACRO{\TeXButton{ISO}{\simeq}}%
%BeginExpansion
\simeq%
%EndExpansion
%TCIMACRO{\TeXButton{coker}{\coker}}%
%BeginExpansion
\coker%
%EndExpansion
\left( h_{\#}\right) .
\end{equation*}%
Since the category of precosheaves $\mathbf{pCS}\left( X,\mathbf{Pro}\left(
k\right) \right) $ is abelian,%
\begin{equation*}
\varphi :coim\left( \iota f\right) \longrightarrow im\left( \iota f\right)
\end{equation*}%
is an isomorphism. It follows that%
\begin{equation*}
\varphi _{\#}:coim\left( f\right) =\left( coim\left( \iota f\right) \right)
_{\#}\longrightarrow \left( im\left( \iota f\right) \right) _{\#}=im\left(
f\right)
\end{equation*}%
is an isomorphism as well, and the category of cosheaves $\mathbf{CS}\left(
X,\mathbf{Pro}\left( k\right) \right) $ is abelian.
\end{itemize}

(\textbf{\ref{Th-Properties-cosheaves-colimit}}) Follows from\textbf{\ }%
Theorem \ref{Th-Properties-precosheaves}(\ref%
{Th-Properties-precosheaves-colimit}), because the inclusion functor%
\begin{equation*}
\iota :\mathbf{CS}\left( X,\mathbf{Pro}\left( k\right) \right)
\longrightarrow \mathbf{pCS}\left( X,\mathbf{Pro}\left( k\right) \right) ,
\end{equation*}%
being a left adjoint to $\left( {}\right) _{\#}$, preserves \textbf{colimits}%
, while%
\begin{equation*}
\iota :\mathbf{S}\left( X,\mathbf{Mod}\left( k\right) \right)
\longrightarrow \mathbf{pS}\left( X,\mathbf{Mod}\left( k\right) \right) ,
\end{equation*}%
being a right adjoint to $\left( {}\right) ^{\#}$, preserves \textbf{limits}.

(\textbf{\ref{Th-Properties-cosheaves-cofiltered-limit}})%
\begin{eqnarray*}
&&\left\langle \underleftarrow{\lim }_{i\in \mathbf{I}}\mathcal{X}%
_{i},T\right\rangle 
%TCIMACRO{\TeXButton{ISO}{\simeq}}%
%BeginExpansion
\simeq%
%EndExpansion
\left( \iota \left\langle \underleftarrow{\lim }_{i\in \mathbf{I}}\mathcal{X}%
_{i},T\right\rangle \right) ^{\#}%
%TCIMACRO{\TeXButton{ISO}{\simeq}}%
%BeginExpansion
\simeq%
%EndExpansion
\left\langle \left( \iota \circ \underleftarrow{\lim }_{i\in \mathbf{I}}%
\mathcal{X}_{i}\right) _{\#},T\right\rangle 
%TCIMACRO{\TeXButton{ISO}{\simeq} }%
%BeginExpansion
\simeq
%EndExpansion
\\
&&%
%TCIMACRO{\TeXButton{ISO}{\simeq}}%
%BeginExpansion
\simeq%
%EndExpansion
\left\langle \left( \underleftarrow{\lim }_{i\in \mathbf{I}}\left( \iota
\circ \mathcal{X}_{i}\right) \right) _{\#},T\right\rangle 
%TCIMACRO{\TeXButton{ISO}{\simeq}}%
%BeginExpansion
\simeq%
%EndExpansion
\left\langle \underleftarrow{\lim }_{i\in \mathbf{I}}\left( \iota \circ 
\mathcal{X}_{i}\right) ,T\right\rangle 
%TCIMACRO{\TeXButton{ISO}{\simeq} }%
%BeginExpansion
\simeq
%EndExpansion
\\
&&%
%TCIMACRO{\TeXButton{ISO}{\simeq}}%
%BeginExpansion
\simeq%
%EndExpansion
\underrightarrow{\lim }_{i\in \mathbf{I}}\left\langle \iota \circ \mathcal{X}%
_{i},T\right\rangle 
%TCIMACRO{\TeXButton{ISO}{\simeq}}%
%BeginExpansion
\simeq%
%EndExpansion
\underrightarrow{\lim }_{i\in \mathbf{I}}\left\langle \left( \iota \circ 
\mathcal{X}_{i}\right) _{\#},T\right\rangle 
%TCIMACRO{\TeXButton{ISO}{\simeq}}%
%BeginExpansion
\simeq%
%EndExpansion
\underrightarrow{\lim }_{i\in \mathbf{I}}\left\langle \mathcal{X}%
_{i},T\right\rangle ,
\end{eqnarray*}%
since $\mathcal{X}_{i}$ are cosheaves.

(\textbf{\ref{Th-Properties-cosheaves-product}}) If $A\subseteq I$ is
finite, then the isomorphism%
\begin{equation*}
\left\langle \dprod\limits_{i\in A}\mathcal{X}_{i},T\right\rangle 
%TCIMACRO{\TeXButton{ISO}{\simeq}}%
%BeginExpansion
\simeq%
%EndExpansion
\dbigoplus\limits_{i\in A}\left\langle \mathcal{X}_{i},T\right\rangle
\end{equation*}%
follows from the additivity of $\left\langle \bullet ,T\right\rangle $. Let
now $\mathbf{J}$ be the poset of finite subsets of $I$. $\mathbf{J}$ is
clearly filtered, and $\mathbf{J}^{op}$ is cofiltered. Due to (\ref%
{Th-Properties-cosheaves-cofiltered-limit}),%
\begin{eqnarray*}
&&\left\langle \dprod\limits_{i\in I}\mathcal{X}_{i},T\right\rangle 
%TCIMACRO{\TeXButton{ISO}{\simeq}}%
%BeginExpansion
\simeq%
%EndExpansion
\left\langle \underset{A\in \mathbf{J}^{op}}{\underleftarrow{\lim }}\left(
\dprod\limits_{i\in A}\mathcal{X}_{i}\right) ,T\right\rangle 
%TCIMACRO{\TeXButton{ISO}{\simeq}}%
%BeginExpansion
\simeq%
%EndExpansion
\underset{A\in \mathbf{J}}{\underrightarrow{\lim }}\left\langle
\dprod\limits_{i\in A}\mathcal{X}_{i},T\right\rangle 
%TCIMACRO{\TeXButton{ISO}{\simeq} }%
%BeginExpansion
\simeq
%EndExpansion
\\
&&%
%TCIMACRO{\TeXButton{ISO}{\simeq}}%
%BeginExpansion
\simeq%
%EndExpansion
\underset{A\in \mathbf{J}}{\underrightarrow{\lim }}\left(
\dbigoplus\limits_{i\in A}\left\langle \mathcal{X}_{i},T\right\rangle
\right) 
%TCIMACRO{\TeXButton{ISO}{\simeq}}%
%BeginExpansion
\simeq%
%EndExpansion
\dbigoplus\limits_{i\in I}\left\langle \mathcal{X}_{i},T\right\rangle .
\end{eqnarray*}

(\textbf{\ref{Th-Properties-cosheaves-limit}}) It is enough to prove:

\begin{enumerate}
\item $\left\langle \bullet ,T\right\rangle $ converts products into
coproducts: done in (\ref{Th-Properties-cosheaves-product});

\item $\left\langle \bullet ,T\right\rangle $ converts kernels into
cokernels: done in (\ref{Th-Properties-cosheaves-Homology}).
\end{enumerate}

(\textbf{\ref{Th-Properties-cosheaves-Zero}}) Follows from\textbf{\ }Theorem %
\ref{Th-Properties-precosheaves}(\ref{Th-Properties-precosheaves-Zero}).

(\textbf{\ref{Th-Properties-cosheaves-Homology}})%
\begin{equation*}
H\left( \mathcal{E}\right) =\frac{\ker \left( \alpha \right) }{im\left(
\beta \right) }=%
%TCIMACRO{\TeXButton{coker}{\coker}}%
%BeginExpansion
\coker%
%EndExpansion
\left( \mathcal{K}\longrightarrow \ker \left( \alpha \right) =\left( \ker
\left( \iota \alpha \right) \right) _{\#}\right) .
\end{equation*}%
It follows from Theorem \ref{Th-Properties-precosheaves} (\ref%
{Th-Properties-precosheaves-Homology}) that%
\begin{eqnarray*}
&&\left\langle H\left( \mathcal{E}\right) ,T\right\rangle 
%TCIMACRO{\TeXButton{ISO}{\simeq}}%
%BeginExpansion
\simeq%
%EndExpansion
\left\langle 
%TCIMACRO{\TeXButton{coker}{\coker}}%
%BeginExpansion
\coker%
%EndExpansion
\left( \iota \mathcal{K}\longrightarrow \left( \ker \left( \iota \alpha
\right) \right) _{\#}\right) ,T\right\rangle 
%TCIMACRO{\TeXButton{ISO}{\simeq}}%
%BeginExpansion
\simeq%
%EndExpansion
\ker \left( \left\langle \left( \ker \left( \iota \alpha \right) \right)
_{\#}\longrightarrow \iota \mathcal{K},T\right\rangle \right) 
%TCIMACRO{\TeXButton{ISO}{\simeq} }%
%BeginExpansion
\simeq
%EndExpansion
\\
&&%
%TCIMACRO{\TeXButton{ISO}{\simeq}}%
%BeginExpansion
\simeq%
%EndExpansion
\ker \left( \left\langle \left( \ker \left( \iota \alpha \right) \right)
_{\#},T\right\rangle \longrightarrow \left\langle \iota \mathcal{K}%
,T\right\rangle \right) 
%TCIMACRO{\TeXButton{ISO}{\simeq}}%
%BeginExpansion
\simeq%
%EndExpansion
\ker \left( \left\langle \ker \left( \iota \alpha \right) ,T\right\rangle
^{\#}\longrightarrow \left\langle \iota \mathcal{K},T\right\rangle \right) 
%TCIMACRO{\TeXButton{ISO}{\simeq} }%
%BeginExpansion
\simeq
%EndExpansion
\\
&&%
%TCIMACRO{\TeXButton{ISO}{\simeq}}%
%BeginExpansion
\simeq%
%EndExpansion
\ker \left( \left( 
%TCIMACRO{\TeXButton{coker}{\coker}}%
%BeginExpansion
\coker%
%EndExpansion
\left\langle \iota \alpha ,T\right\rangle \right) ^{\#}\longrightarrow
\left\langle \iota \mathcal{K},T\right\rangle \right) 
%TCIMACRO{\TeXButton{ISO}{\simeq}}%
%BeginExpansion
\simeq%
%EndExpansion
\ker \left( 
%TCIMACRO{\TeXButton{coker}{\coker}}%
%BeginExpansion
\coker%
%EndExpansion
\left\langle \alpha ,T\right\rangle \longrightarrow \left\langle \mathcal{K}%
,T\right\rangle \right) 
%TCIMACRO{\TeXButton{ISO}{\simeq}}%
%BeginExpansion
\simeq%
%EndExpansion
\frac{\ker \left( \left\langle \beta ,T\right\rangle \right) }{im\left(
\left\langle \alpha ,T\right\rangle \right) }.
\end{eqnarray*}

(\textbf{\ref{Th-Properties-cosheaves-Exact}}) Follows from (\ref%
{Th-Properties-cosheaves-Homology}) and (\ref{Th-Properties-cosheaves-Zero}).

(\textbf{\ref{Th-Properties-cosheaves-AB4*}}) Follows, since $\mathbf{S}%
\left( X,\mathbf{Mod}\left( k\right) \right) $ satisfies AB4, from (\ref%
{Th-Properties-cosheaves-product}).

(\textbf{\ref{Th-Properties-cosheaves-AB5*}}) Follows, since $\mathbf{S}%
\left( X,\mathbf{Mod}\left( k\right) \right) $ satisfies AB5, from (\ref%
{Th-Properties-cosheaves-cofiltered-limit}).

(\textbf{\ref{Th-Properties-cosheaves-cogenerators}}) Let%
\begin{equation*}
\left( \varphi :\mathcal{C\twoheadrightarrow D}\right) \in \mathbf{CS}\left(
X,\mathbf{Pro}\left( k\right) \right)
\end{equation*}%
be a non-trivial epimorphism. It means that%
\begin{equation*}
\ker \left( \varphi \right) \neq 0.
\end{equation*}%
Since%
\begin{equation*}
%TCIMACRO{\TeXButton{coker}{\coker}}%
%BeginExpansion
\coker%
%EndExpansion
\left( \varphi \right) 
%TCIMACRO{\TeXButton{ISO}{\simeq}}%
%BeginExpansion
\simeq%
%EndExpansion
%TCIMACRO{\TeXButton{coker}{\coker}}%
%BeginExpansion
\coker%
%EndExpansion
\left( \iota \varphi \right) ,
\end{equation*}%
$\iota \varphi $ is an epimorphism in $\mathbf{pCS}\left( X,\mathbf{Pro}%
\left( k\right) \right) $ as well. It is non-trivial ($\ker \left( \iota
\varphi \right) \neq 0$), because if it were trivial, then%
\begin{equation*}
0\neq \ker \left( \varphi \right) =\left( \ker \left( \iota \varphi \right)
\right) _{\#}=0_{\#}=0.
\end{equation*}%
It follows from Theorem \ref{Th-Properties-precosheaves} (\ref%
{Th-Properties-precosheaves-cogenerators}) that there exists an $A\in 
\mathfrak{G}$, $V\in \mathbf{C}_{X}$, and a morphism%
\begin{equation*}
\psi :\mathcal{C\longrightarrow }A^{V},
\end{equation*}%
that cannot be factored through $\mathcal{D}$. In other words,%
\begin{eqnarray*}
&&Hom_{\mathbf{pCS}\left( X,\mathbf{Pro}\left( k\right) \right) }\left( 
\mathcal{D},\left( A^{V}\right) _{\#}\right) 
%TCIMACRO{\TeXButton{ISO}{\simeq}}%
%BeginExpansion
\simeq%
%EndExpansion
Hom_{\mathbf{pCS}\left( X,\mathbf{Pro}\left( k\right) \right) }\left( 
\mathcal{D},A^{V}\right) \\
&\longrightarrow &Hom_{\mathbf{pCS}\left( X,\mathbf{Pro}\left( k\right)
\right) }\left( \mathcal{C},A^{V}\right) 
%TCIMACRO{\TeXButton{ISO}{\simeq}}%
%BeginExpansion
\simeq%
%EndExpansion
Hom_{\mathbf{pCS}\left( X,\mathbf{Pro}\left( k\right) \right) }\left( 
\mathcal{C},\left( A^{V}\right) _{\#}\right)
\end{eqnarray*}%
is \textbf{not} an epimorphism. It follows that the corresponding morphism%
\begin{equation*}
\psi _{\#}:\mathcal{C}\longrightarrow \left( A^{V}\right) _{\#}
\end{equation*}%
\textbf{cannot} be factored through $\mathcal{D}$.
\end{proof}

\section{\label{Sec-Proof-cosheaf-homology}Proof of Theorem \protect\ref%
{Th-Cosheaf-homology}}

\begin{proof}
(\textbf{\ref{Th-Cosheaf-homology-surjection}}) Define the following functor%
\begin{eqnarray*}
\mathcal{Q} &:&\mathbf{CS}\left( X,\mathbf{Pro}\left( k\right) \right)
\longrightarrow \mathbf{CS}\left( X,\mathbf{Pro}\left( k\right) \right) : \\
\mathcal{Q}\left( \mathcal{A}\right) &=&\left[ \mathcal{P}\left( \mathcal{A}%
\right) \right] _{\#},
\end{eqnarray*}%
where $\mathcal{P}$ is from Theorem \ref{Th-Precosheaf-homology}(\ref%
{Th-Precosheaf-homology-surjection}) One has the following natural
epimorphism%
\begin{equation*}
\rho \left( \mathcal{A}\right) :\mathcal{Q}\left( \mathcal{A}\right) =%
\mathcal{P}\left( \mathcal{A}\right) _{\#}\longrightarrow \mathcal{P}\left( 
\mathcal{A}\right) \longrightarrow \mathcal{A}.
\end{equation*}%
For ordinals $\alpha $, define $\mathcal{Q}_{\alpha }$ using transfinite
induction:%
\begin{equation*}
\mathcal{Q}_{\alpha }\left( \mathcal{A}\right) 
%TCIMACRO{\TeXButton{assigned}{{:=}}}%
%BeginExpansion
{:=}%
%EndExpansion
\mathcal{Q}\left( \mathcal{Q}_{\beta }\left( \mathcal{A}\right) \right)
\end{equation*}%
if $\alpha =\beta +1$, and%
\begin{equation*}
\mathcal{Q}_{\alpha }%
%TCIMACRO{\TeXButton{assigned}{{:=}}}%
%BeginExpansion
{:=}%
%EndExpansion
\underset{\beta <\beta }{\underleftarrow{\lim }}\mathcal{Q}_{\beta }\left( 
\mathcal{A}\right)
\end{equation*}%
if $\alpha $ is a limit ordinal. The sheaves $\left( \left( k_{V}\right)
^{\#}\right) _{V\in Ob\left( \mathbf{C}_{X}\right) }$ form a set of
generators of $\mathbf{S}\left( X,\mathbf{Mod}\left( k\right) \right) $
(Remark \ref{Rem-Generators}). Consider the coproduct%
\begin{equation*}
\mathcal{G}%
%TCIMACRO{\TeXButton{assigned}{{:=}}}%
%BeginExpansion
{:=}%
%EndExpansion
\dbigoplus\limits_{V\in Ob\left( \mathbf{C}_{X}\right) }\left( k_{V}\right)
^{\#}=\left( \dbigoplus\limits_{V\in Ob\left( \mathbf{C}_{X}\right)
}k_{V}\right) ^{\#}
\end{equation*}%
in the category of sheaves. Let $W$ be the set of representatives of all
subsheaves of $\mathcal{G}$. Let further, for $\mathcal{E}\in W$,%
\begin{equation*}
S\left( \mathcal{E}\right) =\left( \dcoprod\limits_{U\in Ob\left( \mathbf{C}%
_{X}\right) }\mathcal{E}\left( U\right) \right) \in \mathbf{Set,}
\end{equation*}%
be the coproduct in the category $\mathbf{Set}$, and let $\beta $ be any
cardinal of cofinality larger than%
\begin{equation*}
\sup \left( card\left( S\left( \mathcal{E}\right) \right) \right) _{\mathcal{%
E}\in W}.
\end{equation*}%
We claim that the epimorphism%
\begin{equation*}
\mathcal{R}\left( \mathcal{A}\right) 
%TCIMACRO{\TeXButton{assigned}{{:=}}}%
%BeginExpansion
{:=}%
%EndExpansion
\mathcal{Q}_{\beta }\left( \mathcal{A}\right) \longrightarrow \mathcal{A}
\end{equation*}%
is as desired. Indeed, it is enough to prove that $\mathcal{Q}_{\beta
}\left( \mathcal{A}\right) $ is quasi-projective.

Let $T$ be any injective $k$-module, and let%
\begin{equation*}
\mathcal{J}_{\alpha }%
%TCIMACRO{\TeXButton{assigned}{{:=}}}%
%BeginExpansion
{:=}%
%EndExpansion
\left\langle \mathcal{Q}_{\alpha }\left( \mathcal{A}\right) ,T\right\rangle
,\alpha \leq \beta .
\end{equation*}%
We have to prove that $\mathcal{J}_{\beta }$ is an injective sheaf. Since $%
\mathcal{G}$ is a generator for $\mathbf{S}\left( X,\mathbf{Mod}\left(
k\right) \right) $, it is enough \cite[Lemme 1.10.1]%
{Grothendieck-Tohoku-1957-MR0102537} to prove the existence of the dashed
arrow in any diagram of the form%
\begin{equation*}
%TCIMACRO{%
%\TeXButton{Triangle}{\begin{diagram}[size=3.0em,textflow]
%\mathcal{B} & \rInto & \mathcal{G} \\
%\dTo & \ldDashto \\
%\mathcal{J}_{\beta } \\
%\end{diagram}}}%
%BeginExpansion
\begin{diagram}[size=3.0em,textflow]
\mathcal{B} & \rInto & \mathcal{G} \\
\dTo & \ldDashto \\
\mathcal{J}_{\beta } \\
\end{diagram}%
%EndExpansion
\end{equation*}%
where $\mathcal{B}$ is a subsheaf of $\mathcal{G}$. Since%
\begin{equation*}
card\left( S\left( \mathcal{B}\right) \right) =card\left(
\dcoprod\limits_{U\in Ob\left( \mathbf{C}_{X}\right) }\mathcal{E}\left(
U\right) \right) <\beta ,
\end{equation*}%
there exists an $\alpha <\beta $, such that $\mathcal{B}\rightarrow \mathcal{%
J}_{\beta }$ factors through $\mathcal{J}_{\alpha }$.

Consider the commutative diagram%
\begin{equation*}
%TCIMACRO{%
%\TeXButton{Injectivity}{\begin{diagram}[size=3.0em,textflow]
%\mathcal{B} & \rInto & \mathcal{G} \\
%\dTo & & \dDashto & \rdDashto \rdDashto(4,2) \\
%\mathcal{J}_{\alpha} & \rTo & \left[ \mathcal{I}_{\alpha +1}\TeXButton{assigned}{{:=}}\left\langle \mathcal{P}\left( \mathcal{Q}_{\alpha +1}\right) ,T\right\rangle \right] & \rTo & \mathcal{J}_{\alpha +1}  & \rTo & \mathcal{J}_{\beta} \\
%\end{diagram}}}%
%BeginExpansion
\begin{diagram}[size=3.0em,textflow]
\mathcal{B} & \rInto & \mathcal{G} \\
\dTo & & \dDashto & \rdDashto \rdDashto(4,2) \\
\mathcal{J}_{\alpha} & \rTo & \left[ \mathcal{I}_{\alpha +1}\TeXButton{assigned}{{:=}}\left\langle \mathcal{P}\left( \mathcal{Q}_{\alpha +1}\right) ,T\right\rangle \right] & \rTo & \mathcal{J}_{\alpha +1}  & \rTo & \mathcal{J}_{\beta} \\
\end{diagram}%
%EndExpansion
\end{equation*}

The second vertical arrow exists, because $\mathcal{I}_{\alpha +1}$ is an
injective \textbf{pre}sheaf, and the morphism $\mathcal{B}\hookrightarrow 
\mathcal{G}$, being a \textbf{mono}morphism of sheaves, is a \textbf{mono}%
morphism of \textbf{pre}sheaves, as well.

(\textbf{\ref{Th-Cosheaf-homology-F-projective}}) The first condition in
Definition \ref{Def-F-projective} follows from (\ref%
{Th-Cosheaf-homology-surjection}). Let now%
\begin{equation*}
0\longrightarrow \mathcal{P}^{\prime }\longrightarrow \mathcal{P}%
\longrightarrow \mathcal{P}^{\prime \prime }\longrightarrow 0
\end{equation*}%
be an exact sequence with $\mathcal{P}$, $\mathcal{P}^{\prime \prime }\in
Q\left( \mathbf{CS}\left( X,\mathbf{Pro}\left( k\right) \right) \right) $.
For any injective $T\in \mathbf{Mod}\left( k\right) $, the sequence of 
\textbf{sheaves}%
\begin{equation*}
0\longrightarrow \left\langle \mathcal{P}^{\prime \prime },T\right\rangle
\longrightarrow \left\langle \mathcal{P},T\right\rangle \longrightarrow
\left\langle \mathcal{P}^{\prime },T\right\rangle \longrightarrow 0
\end{equation*}%
is exact in $\mathbf{S}\left( X,\mathbf{Mod}\left( k\right) \right) $, while 
$\left\langle \mathcal{P}^{\prime \prime },T\right\rangle $ and $%
\left\langle \mathcal{P},T\right\rangle $ are injective. Therefore the above
sequence splits, and%
\begin{equation*}
\left\langle \mathcal{P},T\right\rangle 
%TCIMACRO{\TeXButton{ISO}{\simeq}}%
%BeginExpansion
\simeq%
%EndExpansion
\left\langle \mathcal{P}^{\prime \prime },T\right\rangle \oplus \left\langle 
\mathcal{P}^{\prime },T\right\rangle .
\end{equation*}%
The sheaf $\left\langle \mathcal{P}^{\prime },T\right\rangle $, being a
direct summand of an injective sheaf, is injective, therefore the cosheaf $%
\mathcal{P}^{\prime }$ is quasi-projective,%
\begin{equation*}
\mathcal{P}^{\prime }\in Q\left( \mathbf{CS}\left( X,\mathbf{Pro}\left(
k\right) \right) \right) .
\end{equation*}%
The second condition in Definition \ref{Def-F-projective} is proved!

Apply the functor $\mathcal{A}\mapsto \mathcal{A}\left( U\right) $ to the
split exact sequence above, and get the following split exact sequences in $%
\mathbf{Mod}\left( k\right) $%
\begin{equation*}
%TCIMACRO{%
%\TeXButton{F-projective}{\begin{diagram}
%0 & \rTo & \left\langle \mathcal{P}^{\prime \prime },T\right\rangle \left( U\right) & \rTo & \left\langle \mathcal{P},T\right\rangle \left( U\right) & \rTo & \left\langle \mathcal{P}^{\prime},T\right\rangle \left( U\right) & \rTo & 0 \\
% & & \dTo^{\simeq} & & \dTo^{\simeq} & & \dTo^{\simeq} \\
%0 & \rTo & \left\langle \mathcal{P}^{\prime \prime } \left( U\right),T\right\rangle & \rTo & \left\langle \mathcal{P} \left( U\right),T\right\rangle & \rTo & \left\langle \mathcal{P}^{\prime} \left( U\right),T\right\rangle & \rTo & 0 \\
%\end{diagram}}}%
%BeginExpansion
\begin{diagram}
0 & \rTo & \left\langle \mathcal{P}^{\prime \prime },T\right\rangle \left( U\right) & \rTo & \left\langle \mathcal{P},T\right\rangle \left( U\right) & \rTo & \left\langle \mathcal{P}^{\prime},T\right\rangle \left( U\right) & \rTo & 0 \\
 & & \dTo^{\simeq} & & \dTo^{\simeq} & & \dTo^{\simeq} \\
0 & \rTo & \left\langle \mathcal{P}^{\prime \prime } \left( U\right),T\right\rangle & \rTo & \left\langle \mathcal{P} \left( U\right),T\right\rangle & \rTo & \left\langle \mathcal{P}^{\prime} \left( U\right),T\right\rangle & \rTo & 0 \\
\end{diagram}%
%EndExpansion
\end{equation*}%
It follows that the sequence%
\begin{equation*}
0\longrightarrow \mathcal{P}^{\prime }\left( U\right) \longrightarrow 
\mathcal{P}\left( U\right) \longrightarrow \mathcal{P}^{\prime \prime
}\left( U\right) \longrightarrow 0
\end{equation*}%
is exact in $\mathbf{Pro}\left( k\right) $, and the third condition for the $%
F$-projectivity is proved for the functor%
\begin{equation*}
F\left( \bullet \right) =\Gamma \left( U,\bullet \right) =\bullet \left(
U\right) .
\end{equation*}%
Consider now the following split exact sequences of presheaves%
\begin{equation*}
%TCIMACRO{%
%\TeXButton{Iota}{\begin{diagram}
%0 & \rTo & \iota \left\langle \mathcal{P}^{\prime \prime },T\right\rangle & \rTo & \iota \left\langle \mathcal{P},T\right\rangle \left( U\right) & \rTo & \iota \left\langle \mathcal{P}^{\prime},T\right\rangle & \rTo & 0 \\
% & & \dTo^{\simeq} & & \dTo^{\simeq} & & \dTo^{\simeq} \\
%0 & \rTo & \left\langle \iota \mathcal{P}^{\prime \prime },T\right\rangle & \rTo & \left\langle \iota \mathcal{P} ,T\right\rangle & \rTo & \left\langle \iota \mathcal{P}^{\prime} ,T\right\rangle & \rTo & 0 \\
%\end{diagram}}}%
%BeginExpansion
\begin{diagram}
0 & \rTo & \iota \left\langle \mathcal{P}^{\prime \prime },T\right\rangle & \rTo & \iota \left\langle \mathcal{P},T\right\rangle \left( U\right) & \rTo & \iota \left\langle \mathcal{P}^{\prime},T\right\rangle & \rTo & 0 \\
 & & \dTo^{\simeq} & & \dTo^{\simeq} & & \dTo^{\simeq} \\
0 & \rTo & \left\langle \iota \mathcal{P}^{\prime \prime },T\right\rangle & \rTo & \left\langle \iota \mathcal{P} ,T\right\rangle & \rTo & \left\langle \iota \mathcal{P}^{\prime} ,T\right\rangle & \rTo & 0 \\
\end{diagram}%
%EndExpansion
\end{equation*}%
It follows that the sequence%
\begin{equation*}
0\longrightarrow \iota \mathcal{P}^{\prime }\longrightarrow \iota \mathcal{P}%
\longrightarrow \iota \mathcal{P}^{\prime \prime }\longrightarrow 0
\end{equation*}%
is exact in $\mathbf{pCS}\left( X,\mathbf{Pro}\left( k\right) \right) $, and
the third condition for the $F$-projectivity is proved for the inclusion
functor%
\begin{equation*}
\iota :\mathbf{CS}\left( X,\mathbf{Pro}\left( k\right) \right)
\longrightarrow \mathbf{pCS}\left( X,\mathbf{Pro}\left( k\right) \right) .
\end{equation*}%
(\textbf{\ref{Th-Cosheaf-homology-pairing}}) Let%
\begin{equation*}
0\longleftarrow \mathcal{A}\longleftarrow \mathcal{P}_{0}\longleftarrow 
\mathcal{P}_{1}\longleftarrow \mathcal{P}_{2}\longleftarrow
...\longleftarrow \mathcal{P}_{n}\longleftarrow ...
\end{equation*}%
be a quasi-projective resolution. Apply the functor $\Gamma \left( U,\bullet
\right) =\bullet \left( U\right) $, and get a chain complex of pro-modules:%
\begin{equation*}
0\longleftarrow \mathcal{P}_{0}\left( U\right) \longleftarrow \mathcal{P}%
_{1}\left( U\right) \longleftarrow \mathcal{P}_{2}\left( U\right)
\longleftarrow ...\longleftarrow \mathcal{P}_{n}\left( U\right)
\longleftarrow ...
\end{equation*}%
Let $T\in \mathbf{Mod}\left( k\right) $ be injective. Apply $\left\langle
\bullet ,T\right\rangle $, and get an injective resolution of $\left\langle 
\mathcal{A},T\right\rangle $:%
\begin{equation*}
0\longrightarrow \left\langle \mathcal{A},T\right\rangle \longrightarrow
\left\langle \mathcal{P}_{0},T\right\rangle \longrightarrow \left\langle 
\mathcal{P}_{1},T\right\rangle \longrightarrow \left\langle \mathcal{P}%
_{2},T\right\rangle \longrightarrow ...\longrightarrow \left\langle \mathcal{%
P}_{n},T\right\rangle \longrightarrow ...
\end{equation*}%
It follows that%
\begin{equation*}
\left\langle L_{n}\Gamma \left( U,\mathcal{A}\right) ,T\right\rangle 
%TCIMACRO{\TeXButton{ISO}{\simeq}}%
%BeginExpansion
\simeq%
%EndExpansion
\left\langle H_{n}\left( \mathcal{P}_{\bullet }\left( U\right) \right)
,T\right\rangle 
%TCIMACRO{\TeXButton{ISO}{\simeq}}%
%BeginExpansion
\simeq%
%EndExpansion
H^{n}\left\langle \mathcal{P}_{\bullet }\left( U\right) ,T\right\rangle 
%TCIMACRO{\TeXButton{ISO}{\simeq}}%
%BeginExpansion
\simeq%
%EndExpansion
H^{n}\left( U,\left\langle \mathcal{A},T\right\rangle \right) .
\end{equation*}

(\textbf{\ref{Th-Cosheaf-homology-forgetting}}) Apply the inclusion functor $%
\iota $ to the quasi-projective resolution above, and get a chain complex of
precosheaves:%
\begin{equation*}
0\longleftarrow \mathcal{\iota P}_{0}\longleftarrow \mathcal{\iota P}%
_{1}\longleftarrow \mathcal{\iota P}_{2}\longleftarrow ...\longleftarrow 
\mathcal{\iota P}_{n}\longleftarrow ...
\end{equation*}%
The precosheaf $L_{n}\iota $ is defined by%
\begin{equation*}
\left( L_{n}\iota \right) \mathcal{A}%
%TCIMACRO{\TeXButton{assigned}{{:=}}}%
%BeginExpansion
{:=}%
%EndExpansion
H_{n}\left( \iota \mathcal{P}_{\bullet }\right)
\end{equation*}%
The functor%
\begin{equation*}
\mathcal{B}\longmapsto \mathcal{B}\left( U\right) :\mathbf{pCS}\left( X,%
\mathbf{Pro}\left( k\right) \right) \longrightarrow \mathbf{Pro}\left(
k\right)
\end{equation*}%
is exact, therefore%
\begin{equation*}
\left[ \left( L_{n}\iota \right) \mathcal{A}\right] \left( U\right) 
%TCIMACRO{\TeXButton{ISO}{\simeq}}%
%BeginExpansion
\simeq%
%EndExpansion
H_{n}\left( \iota \mathcal{P}_{\bullet }\left( U\right) \right) 
%TCIMACRO{\TeXButton{ISO}{\simeq}}%
%BeginExpansion
\simeq%
%EndExpansion
H_{n}\left( \mathcal{P}_{\bullet }\left( U\right) \right) 
%TCIMACRO{\TeXButton{ISO}{\simeq}}%
%BeginExpansion
\simeq%
%EndExpansion
H_{n}\left( U,\mathcal{A}\right) ,
\end{equation*}%
proving (b). Moreover,%
\begin{equation*}
\left\langle \left( L_{n}\iota \right) \mathcal{A},T\right\rangle 
%TCIMACRO{\TeXButton{ISO}{\simeq}}%
%BeginExpansion
\simeq%
%EndExpansion
\left\langle H_{n}\left( \iota \mathcal{P}_{\bullet }\right) ,T\right\rangle 
%TCIMACRO{\TeXButton{ISO}{\simeq}}%
%BeginExpansion
\simeq%
%EndExpansion
H^{n}\left\langle \iota \mathcal{P}_{\bullet },T\right\rangle 
%TCIMACRO{\TeXButton{ISO}{\simeq}}%
%BeginExpansion
\simeq%
%EndExpansion
\mathcal{H}^{n}\left\langle \mathcal{A},T\right\rangle ,
\end{equation*}%
proving (a).

(\textbf{\ref{Th-Cosheaf-homology-H-t-plus-trivial}}) It follows from \cite[%
Theorem 2.12(2, 3)]{Prasolov-Cosheafification-2016-zbMATH06684178} that%
\begin{equation*}
\left( \mathcal{H}_{t}\mathcal{A}\right) _{\#}\longrightarrow \left( 
\mathcal{H}_{t}\mathcal{A}\right) _{+}
\end{equation*}%
is an epimorphism. Therefore, it is enough to prove that $\left( \mathcal{H}%
_{t}\mathcal{A}\right) _{\#}=0$ for $t>0$. Apply the \textbf{exact} (due to
Theorem \ref{Th-Cosheafification} (\ref{Th-Cosheafification-plusplus-exact}%
)) functor $\left( {}\right) _{\#}$ to the chain complex%
\begin{equation*}
0\longleftarrow \mathcal{\iota P}_{0}\longleftarrow \mathcal{\iota P}%
_{1}\longleftarrow \mathcal{\iota P}_{2}\longleftarrow ...\longleftarrow 
\mathcal{\iota P}_{n}\longleftarrow ...
\end{equation*}%
Since $\left( {}\right) _{\#}\circ \iota =1_{\mathbf{CS}\left( X,\mathbf{Pro}%
\left( k\right) \right) }$, one gets an acyclic complex%
\begin{equation*}
\left( \iota \mathcal{P}_{\bullet }\right) _{\#}%
%TCIMACRO{\TeXButton{ISO}{\simeq}}%
%BeginExpansion
\simeq%
%EndExpansion
\left( \mathcal{P}_{\bullet }\right) .
\end{equation*}%
Therefore,%
\begin{equation*}
0=H_{t}\mathcal{P}_{\bullet }%
%TCIMACRO{\TeXButton{ISO}{\simeq}}%
%BeginExpansion
\simeq%
%EndExpansion
H_{t}\left[ \left( \iota \mathcal{P}_{\bullet }\right) _{\#}\right] 
%TCIMACRO{\TeXButton{ISO}{\simeq}}%
%BeginExpansion
\simeq%
%EndExpansion
\left[ H_{t}\left( \iota \mathcal{P}_{\bullet }\right) \right] _{\#}%
%TCIMACRO{\TeXButton{ISO}{\simeq}}%
%BeginExpansion
\simeq%
%EndExpansion
\left( \mathcal{H}_{t}\mathcal{A}\right) _{\#}
\end{equation*}%
for $t>0$.

(\textbf{\ref{Th-Cosheaf-homology-Spectral-sequence-R}}) Let $X_{\bullet
,\bullet }$ be the following bicomplex in $\mathbf{Pro}\left( k\right) $:%
\begin{equation*}
\left( X_{s,t},d,\delta \right) 
%TCIMACRO{\TeXButton{assigned}{{:=}}}%
%BeginExpansion
{:=}%
%EndExpansion
\left( \dbigoplus\limits_{\left( U_{0}\rightarrow U_{1}\rightarrow
...\rightarrow U_{s}\rightarrow U\right) \in \mathbf{C}_{R}}\mathcal{P}%
_{t}\left( U_{0}\right) ,d,\delta \right) ,
\end{equation*}%
where $\delta $ is inherited from the above quasi-projective resolution, and 
$d$ is as in Definition \ref{Def-Roos-complex}. Consider the two spectral
sequences%
\begin{eqnarray*}
^{ver}E_{s,t}^{2} &\implies &H_{s+t}\left( Tot_{\bullet }\left( X\right)
\right) , \\
^{hor}E_{s,t}^{2} &\implies &H_{s+t}\left( Tot_{\bullet }\left( X\right)
\right) .
\end{eqnarray*}%
Since $\mathcal{P}_{t}$ are quasi-projective cosheaves, thus
quasi-projective \textbf{pre}cosheaves, it follows that%
\begin{eqnarray*}
^{hor}E_{s,t}^{1} &=&~^{hor}H_{s}\left( X_{\bullet ,\bullet }\right)
=H_{s}\left( R,\mathcal{P}_{t}\right) =\left\{ 
\begin{array}{ccc}
H_{0}\left( R,\mathcal{P}_{t}\right) 
%TCIMACRO{\TeXButton{ISO}{\simeq}}%
%BeginExpansion
\simeq%
%EndExpansion
\mathcal{P}_{t}\left( U\right) & \text{if} & s=0, \\ 
0 & \text{if} & s\neq 0.%
\end{array}%
\right. \\
^{hor}E_{s,t}^{2} &=&\left\{ 
\begin{array}{ccc}
H_{t}\left( U,\mathcal{A}\right) & \text{if} & s=0, \\ 
0 & \text{if} & s\neq 0.%
\end{array}%
\right.
\end{eqnarray*}%
The spectral sequence degenerates from $E_{2}$ on, implying%
\begin{equation*}
H_{n}\left( Tot_{\bullet }\left( X\right) \right) 
%TCIMACRO{\TeXButton{ISO}{\simeq}}%
%BeginExpansion
\simeq%
%EndExpansion
H_{n}\left( U,\mathcal{A}\right) ,
\end{equation*}%
Furthermore,%
\begin{eqnarray*}
^{ver}E_{s,t}^{1} &=&~^{ver}H_{t}\left( X_{\bullet ,\bullet }\right)
=\dbigoplus\limits_{\left( U_{0}\rightarrow U_{1}\rightarrow ...\rightarrow
U_{s}\rightarrow U\right) \in \mathbf{C}_{R}}\mathcal{H}_{t}\mathcal{A}%
\left( U_{0}\right) , \\
^{ver}E_{s,t}^{2} &=&H_{s}\left( R,\mathcal{H}_{t}\mathcal{A}\right)
\implies H_{s+t}\left( Tot_{\bullet }\left( X\right) \right) 
%TCIMACRO{\TeXButton{ISO}{\simeq}}%
%BeginExpansion
\simeq%
%EndExpansion
H_{s+t}\left( U,\mathcal{A}\right) ,
\end{eqnarray*}%
proving (a).

Apply $\underleftarrow{\lim }$ over all covering sieves, to the above
spectral sequence, and get the desired spectral sequence%
\begin{equation*}
E_{s,t}^{2}=\check{H}_{s}\left( U,\mathcal{H}_{t}\mathcal{A}\right) \implies
H_{s+t}\left( U,\mathcal{A}\right) ,
\end{equation*}%
proving (b).

To prove (c), notice that%
\begin{eqnarray*}
&&\mathcal{H}_{0}\mathcal{A}%
%TCIMACRO{\TeXButton{ISO}{\simeq}}%
%BeginExpansion
\simeq%
%EndExpansion
\mathcal{A}, \\
&&E_{s,0}^{2}%
%TCIMACRO{\TeXButton{ISO}{\simeq}}%
%BeginExpansion
\simeq%
%EndExpansion
\check{H}_{s}\left( U,\mathcal{A}\right) , \\
E_{0,t}^{2} &=&0,~t>0.
\end{eqnarray*}%
It follows that%
\begin{equation*}
\check{H}_{0}\left( U,\mathcal{A}\right) 
%TCIMACRO{\TeXButton{ISO}{\simeq}}%
%BeginExpansion
\simeq%
%EndExpansion
E_{0,0}^{2}%
%TCIMACRO{\TeXButton{ISO}{\simeq}}%
%BeginExpansion
\simeq%
%EndExpansion
E_{0,0}^{\infty }%
%TCIMACRO{\TeXButton{ISO}{\simeq}}%
%BeginExpansion
\simeq%
%EndExpansion
H_{0}\left( U,\mathcal{A}\right) .
\end{equation*}%
Moreover, there is a short exact sequence%
\begin{equation*}
0\longrightarrow \left[ E_{0,1}^{\infty }=0\right] \longrightarrow
H_{1}\left( U,\mathcal{A}\right) \longrightarrow \left[ E_{1,0}^{\infty }=%
\check{H}_{1}\left( U,\mathcal{A}\right) \right] \longrightarrow 0,
\end{equation*}%
implying%
\begin{equation*}
H_{1}\left( U,\mathcal{A}\right) 
%TCIMACRO{\TeXButton{ISO}{\simeq}}%
%BeginExpansion
\simeq%
%EndExpansion
\check{H}_{1}\left( U,\mathcal{A}\right) .
\end{equation*}%
Finally,%
\begin{equation*}
E_{2,0}^{\infty }%
%TCIMACRO{\TeXButton{ISO}{\simeq}}%
%BeginExpansion
\simeq%
%EndExpansion
E_{2,0}^{3}%
%TCIMACRO{\TeXButton{ISO}{\simeq}}%
%BeginExpansion
\simeq%
%EndExpansion
\ker \left( E_{2,0}^{2}\longrightarrow \left[ E_{0,1}^{2}=0\right] \right) 
%TCIMACRO{\TeXButton{ISO}{\simeq}}%
%BeginExpansion
\simeq%
%EndExpansion
E_{2,0}^{2}%
%TCIMACRO{\TeXButton{ISO}{\simeq}}%
%BeginExpansion
\simeq%
%EndExpansion
\check{H}_{2}\left( U,\mathcal{A}\right) ,
\end{equation*}%
and there is, since $E_{0,2}^{\infty }=0$, a short exact sequence%
\begin{equation*}
0\longrightarrow E_{1,1}^{\infty }\longrightarrow H_{2}\left( U,\mathcal{A}%
\right) \longrightarrow \left[ E_{2,0}^{2}%
%TCIMACRO{\TeXButton{ISO}{\simeq}}%
%BeginExpansion
\simeq%
%EndExpansion
\check{H}_{2}\left( U,\mathcal{A}\right) \right] \longrightarrow 0,
\end{equation*}%
implying%
\begin{equation*}
H_{2}\left( U,\mathcal{A}\right) \twoheadrightarrow \check{H}_{2}\left( U,%
\mathcal{A}\right) .
\end{equation*}

(\textbf{\ref{Th-Cosheaf-homology-Spectral-sequence-Cech}}) Follows from
Proposition \ref{Prop-Two-Cech-equivalent}.
\end{proof}

\appendix

\section{Categories}

\subsection{Limits}

\begin{notation}
~

\begin{enumerate}
\item We shall denote \textbf{limits} (inverse/projective limits) by $%
\underleftarrow{\lim }$, and \textbf{colimits} (direct/inductive limits) by $%
\underrightarrow{\lim }$.

\item If $U$ is an object of a category $\mathbf{K}$, we shall usually write 
$U\in \mathbf{K}$ instead of $U\in Ob\left( \mathbf{K}\right) $.
\end{enumerate}
\end{notation}

\begin{definition}
\label{Def-(co)cone}~

\begin{enumerate}
\item An $\mathbf{I}$-\textbf{diagram} in $\mathbf{K}$ is a functor%
\begin{equation*}
D:\mathbf{I}\longrightarrow \mathbf{K}
\end{equation*}%
where $\mathbf{I}$ is a \textbf{small} category.

\item A \textbf{cone} (respectively \textbf{cocone}) of the diagram $D$\ is
a pair $\left( U,\varphi \right) $ where $U\in \mathbf{K}$, and $\varphi $
is a morphism of functors $\varphi :U^{const}\rightarrow D$ (respectively $%
D\rightarrow U^{const}$). Here $U^{const}$ is the constant functor:%
\begin{eqnarray*}
U^{const}\left( i\right) &=&U,i\in \mathbf{I}, \\
U^{const}\left( i\rightarrow j\right) &=&\mathbf{1}_{U}.
\end{eqnarray*}

\item A cone $\left( U,\varphi \right) $ \textbf{is a limit} iff $%
\underleftarrow{\lim }~D$ exists, and the evident morphism%
\begin{equation*}
U\longrightarrow \underleftarrow{\lim }~D
\end{equation*}%
is an isomorphism.

\item A cocone $\left( U,\varphi \right) $ \textbf{is a colimit} iff $%
\underrightarrow{\lim }~D$ exists, and the evident morphism%
\begin{equation*}
\underrightarrow{\lim }~D\longrightarrow U
\end{equation*}%
is an isomorphism.
\end{enumerate}
\end{definition}

\begin{definition}
\label{Def-Functor-category}Given two categories $\mathbf{I}$ and $\mathbf{K}
$ with $\mathbf{I}$ small, let $\mathbf{K}^{\mathbf{I}}$ be the category of $%
\mathbf{I}$-diagrams in $\mathbf{K}$.
\end{definition}

\begin{remark}
We will also consider functors $\mathbf{C\rightarrow D}$ where $\mathbf{C}$
is not small. However, such functors form a \textbf{quasi-category} $\mathbf{%
D}^{\mathbf{C}}$, because the morphisms $\mathbf{D}^{\mathbf{C}}\left(
F,G\right) $ form a \textbf{class}, but not in general a \textbf{set}.
\end{remark}

\begin{definition}
\label{Def-(co)complete}A category $\mathbf{K}$ is called \textbf{complete}
if it admits small limits, and \textbf{cocomplete} if it admits small
colimits.
\end{definition}

\begin{remark}
A complete category has a \textbf{terminal} object (a limit of an empty
diagram). A cocomplete category has an \textbf{initial} object (a colimit of
an empty diagram).
\end{remark}

\begin{definition}
\label{Def-exact-functors}A functor $f:\mathbf{C}\rightarrow \mathbf{D}$ is
called \textbf{left (right) exact} if it preserves \textbf{finite} limits
(colimits). $f$ is called \textbf{exact} if it is both left and right exact.
\end{definition}

\begin{definition}
\label{Def-(co)reflective}A subcategory $\mathbf{C\subseteq D}$ is called 
\textbf{reflective} (respectively \textbf{coreflective}) iff the inclusion $%
\mathbf{C\hookrightarrow D}$ is a right (respectively left) adjoint. The
left (respectively right) adjoint $\mathbf{D\rightarrow C}$ is called a 
\textbf{reflection} (respectively \textbf{coreflection}).
\end{definition}

\begin{definition}
\label{Def-Yoneda-embeddings}Given $U\in \mathbf{K}$, let%
\begin{equation*}
h_{U}:\mathbf{K}^{op}\longrightarrow \mathbf{Set},~h^{U}:\mathbf{K}%
\longrightarrow \mathbf{Set},
\end{equation*}%
be the following functors:%
\begin{eqnarray*}
&&h_{U}\left( V\right) 
%TCIMACRO{\TeXButton{assigned}{{:=}}}%
%BeginExpansion
{:=}%
%EndExpansion
Hom_{\mathbf{C}}\left( V,U\right) ,~h^{U}\left( V\right) 
%TCIMACRO{\TeXButton{assigned}{{:=}}}%
%BeginExpansion
{:=}%
%EndExpansion
Hom_{\mathbf{C}}\left( U,V\right) , \\
&&h_{U}\left( \alpha \right) 
%TCIMACRO{\TeXButton{assigned}{{:=}}}%
%BeginExpansion
{:=}%
%EndExpansion
\left[ \left( \gamma \in h_{U}\left( V\right) =Hom_{\mathbf{C}}\left(
V,U\right) \right) \longmapsto \left( \gamma \circ \alpha \in Hom_{\mathbf{C}%
}\left( V^{\prime },U\right) =h_{U}\left( V^{\prime }\right) \right) \right]
, \\
&&h^{U}\left( \beta \right) 
%TCIMACRO{\TeXButton{assigned}{{:=}}}%
%BeginExpansion
{:=}%
%EndExpansion
\left[ \left( \gamma \in h^{U}\left( V\right) =Hom_{\mathbf{C}}\left(
U,V\right) \right) \longmapsto \left( \beta \circ \gamma \in Hom_{\mathbf{C}%
}\left( U,V^{\prime }\right) =h^{U}\left( V^{\prime }\right) \right) \right]
,
\end{eqnarray*}%
where%
\begin{eqnarray*}
\left( \alpha :V^{\prime }\longrightarrow V\right) &\in &Hom_{\mathbf{C}%
}\left( V^{\prime },V\right) =Hom_{\mathbf{C}^{op}}\left( V,V^{\prime
}\right) , \\
\left( \beta :V\longrightarrow V^{\prime }\right) &\in &Hom_{\mathbf{C}%
}\left( V,V^{\prime }\right) .
\end{eqnarray*}
\end{definition}

\begin{remark}
\label{Rem-Yoneda-embeddings}~

\begin{enumerate}
\item The functors%
\begin{equation*}
h_{\bullet }:\mathbf{K}\longrightarrow \mathbf{Set}^{\mathbf{K}%
^{op}},~h^{\bullet }:\mathbf{K}^{op}\longrightarrow \mathbf{Set}^{\mathbf{K}%
},
\end{equation*}%
are full embeddings, called the \textbf{first} and the \textbf{second} 
\textbf{Yoneda embeddings}.

\item We will consider also the \textbf{third} Yoneda embedding, which is
dual to the second one:%
\begin{equation*}
\left( h^{\bullet }\right) ^{op}:\mathbf{K}=\left( \mathbf{K}^{op}\right)
^{op}\longrightarrow \left( \mathbf{Set}^{\mathbf{K}}\right) ^{op}.
\end{equation*}
\end{enumerate}
\end{remark}

\begin{definition}
\label{Def-Comma-General}Let%
\begin{equation*}
\varphi :\mathbf{C}\longrightarrow \mathbf{D}
\end{equation*}%
be a functor, and let $d\in \mathbf{D}$.

\begin{enumerate}
\item The \textbf{comma category} $\varphi \downarrow d$ is defined as
follows:%
\begin{eqnarray*}
&&Ob\left( \varphi \downarrow d\right) 
%TCIMACRO{\TeXButton{assigned}{{:=}}}%
%BeginExpansion
{:=}%
%EndExpansion
\left\{ \left( \varphi \left( c\right) \rightarrow d\right) \in Hom_{\mathbf{%
D}}\left( \varphi \left( c\right) ,d\right) \right\} , \\
&&Hom_{\varphi \downarrow d}\left( \left( \alpha _{1}:\varphi \left(
c_{1}\right) \rightarrow d\right) ,\left( \alpha _{2}:\varphi \left(
c_{2}\right) \rightarrow d\right) \right) 
%TCIMACRO{\TeXButton{assigned}{{:=}}}%
%BeginExpansion
{:=}%
%EndExpansion
\left\{ \beta :c_{1}\rightarrow c_{2}~|~\alpha _{2}\circ \varphi \left(
\beta \right) =\alpha _{1}\right\} .
\end{eqnarray*}

\item Another \textbf{comma category} 
\begin{equation*}
d\downarrow \varphi =\left( \varphi ^{op}\downarrow d\right) ^{op}
\end{equation*}
is defined as follows:%
\begin{eqnarray*}
&&Ob\left( d\downarrow \varphi \right) 
%TCIMACRO{\TeXButton{assigned}{{:=}}}%
%BeginExpansion
{:=}%
%EndExpansion
\left\{ \left( d\rightarrow \varphi \left( c\right) \right) \in Hom_{\mathbf{%
D}}\left( d,\varphi \left( c\right) \right) \right\} , \\
&&Hom_{\varphi \downarrow d}\left( \left( \alpha _{1}:d\rightarrow \varphi
\left( c_{1}\right) \right) ,\left( \alpha _{2}:d\rightarrow \varphi \left(
c_{2}\right) \right) \right) 
%TCIMACRO{\TeXButton{assigned}{{:=}}}%
%BeginExpansion
{:=}%
%EndExpansion
\left\{ \beta :c_{1}\rightarrow c_{2}~|~\varphi \left( \beta \right) \circ
\alpha _{1}=\alpha _{2}\right\} .
\end{eqnarray*}
\end{enumerate}
\end{definition}

\begin{definition}
\label{Def-Comma-U}Let $U\in \mathbf{C}$. The \textbf{comma category} $%
\mathbf{C}_{U}$ is defined as follows:%
\begin{equation*}
\mathbf{C}_{U}=\mathbf{1}_{\mathbf{C}}\downarrow U,
\end{equation*}%
i.e.%
\begin{eqnarray*}
&&Ob\left( \mathbf{C}_{U}\right) 
%TCIMACRO{\TeXButton{assigned}{{:=}}}%
%BeginExpansion
{:=}%
%EndExpansion
\left\{ \left( V\rightarrow U\right) \in Hom_{\mathbf{C}}\left( V,U\right)
\right\} , \\
&&Hom_{\mathbf{C}_{U}}\left( \left( \alpha _{1}:V_{1}\rightarrow U\right)
,\left( \alpha _{2}:V_{2}\rightarrow U\right) \right) 
%TCIMACRO{\TeXButton{assigned}{{:=}}}%
%BeginExpansion
{:=}%
%EndExpansion
\left\{ \beta :V_{1}\rightarrow V_{2}~|~\alpha _{2}\circ \beta =\alpha
_{1}\right\} .
\end{eqnarray*}
\end{definition}

\begin{definition}
\label{Def-Comma-R}Let $F\in \mathbf{Set}^{\mathbf{C}^{op}}$. The \textbf{%
comma category} $\mathbf{C}_{F}$ is defined as follows:%
\begin{eqnarray*}
&&Ob\left( \mathbf{C}_{F}\right) 
%TCIMACRO{\TeXButton{assigned}{{:=}}}%
%BeginExpansion
{:=}%
%EndExpansion
\left\{ \left( V,\alpha \right) ~|~V\in \mathbf{C},\alpha \in F\left(
V\right) \right\} , \\
&&Hom_{\mathbf{C}_{U}}\left( \left( V_{1},\alpha _{1}\right) ,\left(
V_{2},\alpha _{2}\right) \right) 
%TCIMACRO{\TeXButton{assigned}{{:=}}}%
%BeginExpansion
{:=}%
%EndExpansion
\left\{ \beta :V_{1}\rightarrow V_{2}~|~F\left( \beta \right) \left( \alpha
_{2}\right) =\alpha _{1}\right\} .
\end{eqnarray*}
\end{definition}

\begin{remark}
\label{Rem-CU-equivalent-ChU}The categories $\mathbf{C}_{U}$ and $\mathbf{C}%
_{h_{U}}$ are equivalent.
\end{remark}

\begin{definition}
\label{Def-Kan-extensions}Let $\mathbf{I}$ and $\mathbf{J}$ be small
categories and let $\mathbf{C}$ be an arbitrary category. For 
\begin{equation*}
\varphi :\mathbf{J}\longrightarrow \mathbf{I}
\end{equation*}%
denote by $\varphi _{\ast }$ the following functor:%
\begin{equation*}
\varphi _{\ast }:\mathbf{C}^{\mathbf{I}}\longrightarrow \mathbf{C}^{\mathbf{J%
}}~\left( \varphi _{\ast }\left( f\right) 
%TCIMACRO{\TeXButton{assigned}{{:=}}}%
%BeginExpansion
{:=}%
%EndExpansion
f\circ \varphi \right) ,
\end{equation*}%
where $f:\mathbf{I}\longrightarrow \mathbf{C}$ is an arbitrary diagram. Then
the following \textbf{left} adjoint ($\varphi ^{\dag }\dashv \varphi _{\ast
} $)%
\begin{equation*}
\varphi ^{\dag }:\mathbf{C}^{\mathbf{J}}\longrightarrow \mathbf{C}^{\mathbf{I%
}}
\end{equation*}%
to $\varphi _{\ast }$ (if exists!) is called the \textbf{left} Kan extension
of $\varphi $. The following \textbf{right} adjoint ($\varphi _{\ast }\dashv
\varphi ^{\ddag }$)%
\begin{equation*}
\varphi ^{\ddag }:\mathbf{C}^{\mathbf{J}}\longrightarrow \mathbf{C}^{\mathbf{%
I}}
\end{equation*}%
to $\varphi _{\ast }$ (if exists!) is called the \textbf{right} Kan
extension of $\varphi $. See \cite[Definition 2.3.1]%
{Kashiwara-Categories-MR2182076}.
\end{definition}

\begin{proposition}
\label{Prop-Kan-extensions}Let $\varphi :\mathbf{J}\longrightarrow \mathbf{I}
$ be a functor and $\beta \in \mathbf{C}^{\mathbf{J}}$.

\begin{enumerate}
\item \label{Prop-Kan-extensions-left}Assume that%
\begin{equation*}
\underset{\left( \varphi \left( j\right) \rightarrow i\right) \in \varphi
\downarrow i}{\underrightarrow{\lim }}\beta \left( j\right)
\end{equation*}%
exists in $\mathbf{C}$ for any $i\in \mathbf{I}$. Then $\varphi ^{\dag
}\beta $ exists, and we have%
\begin{equation*}
\varphi ^{\dag }\beta \left( i\right) =\underset{\left( \varphi \left(
j\right) \rightarrow i\right) \in \varphi \downarrow i}{\underrightarrow{%
\lim }}\beta \left( j\right)
\end{equation*}%
for $i\in \mathbf{I}$.

\item \label{Prop-Kan-extensions-right}Assume that%
\begin{equation*}
\underset{\left( i\rightarrow \varphi \left( j\right) \right) \in
i\downarrow \varphi }{\underleftarrow{\lim }}\beta \left( j\right)
\end{equation*}%
exists in $\mathbf{C}$ for any $i\in \mathbf{I}$. Then $\varphi ^{\ddag
}\beta $ exists, and we have%
\begin{equation*}
\varphi ^{\ddag }\beta \left( i\right) =\underset{\left( i\rightarrow
\varphi \left( j\right) \right) \in i\downarrow \varphi }{\underleftarrow{%
\lim }}\beta \left( j\right)
\end{equation*}%
for $i\in \mathbf{I}$.

\item \label{Prop-Kan-extensions-projective}Assume that $\mathbf{C}$ is
abelian, and that $\varphi ^{\dag }$ exists. Then $\varphi ^{\dag }$
converts projective objects of $\mathbf{C}^{\mathbf{J}}$ into projective
objects of $\mathbf{C}^{\mathbf{I}}$.

\item \label{Prop-Kan-extensions-injective}Assume that $\mathbf{C}$ is
abelian, and that $\varphi ^{\ddag }$ exists. Then $\varphi ^{\ddag }$
converts injective objects of $\mathbf{C}^{\mathbf{J}}$ into injective
objects of $\mathbf{C}^{\mathbf{I}}$.
\end{enumerate}
\end{proposition}

\begin{proof}
For (\ref{Prop-Kan-extensions-left}) and (\ref{Prop-Kan-extensions-right})
see \cite[Theorem 2.3.3]{Kashiwara-Categories-MR2182076}.

(\ref{Prop-Kan-extensions-projective}) $\varphi _{\ast }$ is clearly exact.
If $\mathcal{A}\in \mathbf{C}^{\mathbf{J}}$ is projective, then the functor%
\begin{equation*}
Hom_{\mathbf{C}^{\mathbf{I}}}\left( \varphi ^{\dag }\mathcal{A},\bullet
\right) 
%TCIMACRO{\TeXButton{ISO}{\simeq}}%
%BeginExpansion
\simeq%
%EndExpansion
Hom_{\mathbf{C}^{\mathbf{J}}}\left( \mathcal{A},\varphi _{\ast }\left(
\bullet \right) \right) :\mathbf{C}^{\mathbf{I}}\longrightarrow \mathbf{Ab}
\end{equation*}%
is exact, therefore $\varphi ^{\dag }\mathcal{A}$ is projective.

(\ref{Prop-Kan-extensions-injective}) If $\mathcal{A}\in \mathbf{C}^{\mathbf{%
J}}$ is injective, then the functor%
\begin{equation*}
Hom_{\mathbf{C}^{\mathbf{I}}}\left( \bullet ,\varphi ^{\ddag }\mathcal{A}%
\right) 
%TCIMACRO{\TeXButton{ISO}{\simeq}}%
%BeginExpansion
\simeq%
%EndExpansion
Hom_{\mathbf{C}^{\mathbf{J}}}\left( \varphi _{\ast }\left( \bullet \right) ,%
\mathcal{A}\right) :\mathbf{C}^{\mathbf{I}}\longrightarrow \mathbf{Ab}
\end{equation*}%
is exact, therefore $f^{\ddag }\mathcal{A}$ is injective.
\end{proof}

\subsection{Pro-objects}

The main reference is \cite[Chapter 6]{Kashiwara-Categories-MR2182076} where
the $\mathbf{Ind}$-objects are considered. The $\mathbf{Pro}$-objects used
in this paper are dual to the $\mathbf{Ind}$-objects:%
\begin{equation*}
\mathbf{Pro}\left( \mathbf{C}\right) 
%TCIMACRO{\TeXButton{ISO}{\simeq}}%
%BeginExpansion
\simeq%
%EndExpansion
\left( \mathbf{Ind}\left( \mathbf{C}^{op}\right) \right) ^{op}.
\end{equation*}

\begin{definition}
A category $\mathbf{I}$ is called \textbf{filtered} iff:

\begin{enumerate}
\item It is not empty.

\item For every two objects $i,j\in \mathbf{I}$ there exists an object $k$
and two morphisms%
\begin{eqnarray*}
i &\longrightarrow &k, \\
j &\longrightarrow &k.
\end{eqnarray*}

\item For every two parallel morphisms%
\begin{eqnarray*}
u &:&i\longrightarrow j, \\
v &:&i\longrightarrow j,
\end{eqnarray*}%
there exists an object $k$ and a morphism%
\begin{equation*}
w:j\longrightarrow k,
\end{equation*}%
such that $w\circ u=w\circ v$. A category $\mathbf{I}$ is called \textbf{%
cofiltered} if $\mathbf{I}^{op}$ is filtered. A diagram $D:\mathbf{%
I\rightarrow K}$ is called (co)filtered if $\mathbf{I}$ is a (co)filtered
category.
\end{enumerate}
\end{definition}

See, e.g., \cite[Chapter IX.1]{Mac-Lane-Categories-1998-MR1712872} for
filtered, and \cite[Chapter I.1.4]{Mardesic-Segal-MR676973} for cofiltered
categories.

\begin{remark}
In \cite{Kashiwara-Categories-MR2182076}, such categories and diagrams are
called \textbf{(co)filtrant}.
\end{remark}

\begin{example}
\label{Ex-(co)filtered-poset}For any poset $\left( X,\leq \right) $ one can
define the category $\mathbf{Cat}\left( X\right) $ with%
\begin{equation*}
Ob\left( \mathbf{Cat}\left( X\right) \right) =X,
\end{equation*}%
where each set $Hom_{\mathbf{Cat}\left( X\right) }\left( x,y\right) $
consists of one object $\left( x,y\right) $ if $x\leq y$, and is empty
otherwise.

The poset $X$ is called \textbf{directed} iff $X\neq \varnothing $, and%
\begin{equation*}
\forall x,y\in X~\left[ \exists z\left( x\leq z\&y\leq z\right) \right] .
\end{equation*}%
The poset $X$ is called \textbf{codirected} iff $X\neq \varnothing $, and%
\begin{equation*}
\forall x,y\in X~\left[ \exists z\left( z\leq x\&z\leq y\right) \right] .
\end{equation*}%
It is easy to see that $\mathbf{Cat}\left( X\right) $ is (co)filtered iff $X$
is (co)directed.
\end{example}

\begin{definition}
\label{Def-Pro-category}\label{Def-Pro-C}Let $\mathbf{K}$ be a category. The
pro-category $\mathbf{Pro}\left( \mathbf{K}\right) $ (see \cite[Definition
6.1.1]{Kashiwara-Categories-MR2182076}, \cite[Remark I.1.4]%
{Mardesic-Segal-MR676973}, or \cite[Appendix]{Artin-Mazur-MR883959}) the the
full subcategory of $\left( \mathbf{Set}^{\mathbf{K}}\right) ^{op}$
consisting of functors that are cofiltered limits of representable functors,
i.e. limits of diagrams of the form%
\begin{equation*}
\mathbf{I}\overset{\mathbf{X}}{\longrightarrow }\mathbf{K}\overset{\left(
h^{\bullet }\right) ^{op}}{\longrightarrow }\left( \mathbf{Set}^{\mathbf{K}%
}\right) ^{op}
\end{equation*}%
where $\mathbf{I}$ is a cofiltered category, $\mathbf{X}:\mathbf{I}%
\rightarrow \mathbf{K}$ is a diagram, and $\left( h^{\bullet }\right) ^{op}$
is the third Yoneda embedding. We will simply denote such diagrams by $%
\mathbf{X}=\left( X_{i}\right) _{i\in \mathbf{I}}$.
\end{definition}

\begin{remark}
\label{Rem-Pro-objects-morphisms}See \cite[Lemma 6.1.2 and formula (2.6.4)]%
{Kashiwara-Categories-MR2182076}:

\begin{enumerate}
\item Let two pro-objects be defined by the diagrams $\mathbf{X}=\left(
X_{i}\right) _{i\in \mathbf{I}}$ and $\mathbf{Y}=\left( Y_{j}\right) _{j\in 
\mathbf{J}}$. Then%
\begin{equation*}
Hom_{\mathbf{Pro}\left( \mathbf{K}\right) }\left( \mathbf{X},\mathbf{Y}%
\right) =~\underset{j\in \mathbf{J}}{\underleftarrow{\lim }}~\underset{i\in 
\mathbf{I}}{\underrightarrow{\lim }}~Hom_{\mathbf{K}}\left(
X_{i},Y_{j}\right) .
\end{equation*}

\item $\mathbf{Pro}\left( \mathbf{K}\right) $ is indeed a category even
though $\left( \mathbf{Set}^{\mathbf{K}}\right) ^{op}$ is a quasi-category: $%
Hom_{\mathbf{Pro}\left( \mathbf{K}\right) }\left( \mathbf{X},\mathbf{Y}%
\right) $ is a \textbf{set} for any $\mathbf{X}$ and $\mathbf{Y}$.
\end{enumerate}
\end{remark}

\begin{remark}
\label{Rem-Trivial-pro-object}\label{Rem-Rudimentary}The category $\mathbf{K}
$ is a full subcategory of $\mathbf{Pro}\left( \mathbf{K}\right) $: any
object $X\in \mathbf{K}$ gives rise to the singleton%
\begin{equation*}
\left( X\right) \in \mathbf{Pro}\left( \mathbf{K}\right)
\end{equation*}%
with a trivial index category $\mathbf{I=}\left( \left\{ i\right\} ,\mathbf{1%
}_{i}\right) $. A pro-object $\mathbf{X}$ is called \textbf{rudimentary} 
\cite[\S I.1.1]{Mardesic-Segal-MR676973} iff it is isomorphic to an object
of $\mathbf{K}$:%
\begin{equation*}
\mathbf{X}%
%TCIMACRO{\TeXButton{ISO}{\simeq}}%
%BeginExpansion
\simeq%
%EndExpansion
Z\in \mathbf{K}\subseteq \mathbf{Pro}\left( \mathbf{K}\right) .
\end{equation*}
\end{remark}

The proposition below allows us to recognize rudimentary pro-objects:

\begin{proposition}
\label{Prop-Trivial-pro-object}\label{Prop-Rudimentary-pro-object}Let%
\begin{equation*}
\mathbf{X}=\left( X_{i}\right) _{i\in \mathbf{I}}\in \mathbf{Pro}\left( 
\mathbf{K}\right) ,
\end{equation*}%
and $Z\in \mathbf{K}$. Then $\mathbf{X}%
%TCIMACRO{\TeXButton{ISO}{\simeq}}%
%BeginExpansion
\simeq%
%EndExpansion
Z$ iff there exist an $i_{0}\in \mathbf{I}$ and a morphism $\tau
_{0}:X_{i_{0}}\rightarrow Z$ satisfying the property: for any morphism $%
s:i\rightarrow i_{0}$, there exist a morphism $\sigma :Z\rightarrow X_{i}$
and a morphism $t:j\rightarrow i$ satisfying%
\begin{eqnarray*}
\tau _{0}\circ X\left( s\right) \circ \sigma &=&\mathbf{1}_{Z}, \\
\sigma \circ \tau _{0}\circ X\left( s\right) \circ X\left( t\right)
&=&X\left( t\right) .
\end{eqnarray*}
\end{proposition}

\begin{proof}
The statement is dual to \cite[Proposition 6.2.1]%
{Kashiwara-Categories-MR2182076}.
\end{proof}

\begin{corollary}
\label{Cor-Zero-pro-object}Let%
\begin{equation*}
\mathbf{X}=\left( X_{i}\right) _{i\in \mathbf{I}}\in \mathbf{Pro}\left(
k\right) .
\end{equation*}%
Then $\mathbf{X}$ is a zero object in $\mathbf{Pro}\left( k\right) $ iff for
any $i\in \mathbf{I}$ there exists a $t:j\rightarrow i$ with $X\left(
t\right) =0$.
\end{corollary}

\begin{remark}
\label{Rem-Level-morphisms}Remark \ref{Rem-Pro-objects-morphisms} allows the
following description of morphisms in the pro-category: any%
\begin{equation*}
f\in Hom_{\mathbf{Pro}\left( \mathbf{K}\right) }\left( \left( X_{i}\right)
_{i\in \mathbf{I}},\left( Y_{j}\right) _{j\in \mathbf{J}}\right) =~\underset{%
j\in \mathbf{J}}{\underleftarrow{\lim }}~\underset{i\in \mathbf{I}}{%
\underrightarrow{\lim }}~Hom_{\mathbf{K}}\left( X_{i},Y_{j}\right)
\end{equation*}%
can be represented (not uniquely!) by a triple%
\begin{equation*}
\left( \varphi ,\lambda ,\left( f_{j}\right) _{j\in \mathbf{J}}\right) ,
\end{equation*}%
where%
\begin{eqnarray*}
\varphi &:&Ob\left( \mathbf{J}\right) \longrightarrow Ob\left( \mathbf{I}%
\right) , \\
\lambda &=&\left[ \alpha \longmapsto \left[ \varphi \left( j_{1}\right) 
\overset{\lambda _{1}\left( \alpha \right) }{\longleftarrow }\Lambda \left(
\alpha \right) \overset{\lambda _{0}\left( \alpha \right) }{\longrightarrow }%
\varphi \left( j_{0}\right) \right] \right] :Mor\left( \mathbf{J}\right)
\longrightarrow Ob\left( \mathbf{I}\right) \times Mor\left( \mathbf{I}%
\right) \times Mor\left( \mathbf{I}\right) ,
\end{eqnarray*}%
are functions, and%
\begin{equation*}
\left( f_{j}:X_{\varphi \left( j\right) }\longrightarrow Y_{j}\right) _{j\in 
\mathbf{J}}
\end{equation*}%
is a family of morphisms, such that the following diagram%
\begin{equation*}
%TCIMACRO{%
%\TeXButton{Pro-morphism}{\begin{diagram}[size=3.0em,textflow]
%X_{\varphi \left( j_{1}\right)} & \lTo^{X\left( \lambda _{1}\left( \alpha \right) \right)} & X_{\Lambda \left( \alpha \right)} & \rTo^{X\left( \lambda _{0}\left( \alpha \right) \right)} & X_{\varphi \left( j_{0}\right)} \\
%\dTo & & & & \dTo \\
%Y_{j_{1}} & & \lTo^{Y\left( \alpha \right)} & & Y_{j_{0}} \\
%\end{diagram}}}%
%BeginExpansion
\begin{diagram}[size=3.0em,textflow]
X_{\varphi \left( j_{1}\right)} & \lTo^{X\left( \lambda _{1}\left( \alpha \right) \right)} & X_{\Lambda \left( \alpha \right)} & \rTo^{X\left( \lambda _{0}\left( \alpha \right) \right)} & X_{\varphi \left( j_{0}\right)} \\
\dTo & & & & \dTo \\
Y_{j_{1}} & & \lTo^{Y\left( \alpha \right)} & & Y_{j_{0}} \\
\end{diagram}%
%EndExpansion
\end{equation*}%
commutes for any $\alpha :j_{0}\rightarrow j_{1}$ in $\mathbf{J}$ (see \cite[%
\S I.1.1]{Mardesic-Segal-MR676973} and \cite[\S A.3]{Artin-Mazur-MR883959}).
It is known that such a morphism is equivalent to a \textbf{level morphism}
(Definition \ref{Def-Level-morphisms}). Moreover, any \textbf{finite}
diagram of pro-objects \textbf{without loops} is equivalent to a level
diagram (see Definition \ref{Def-Level-morphisms} and Proposition \ref%
{Th-Level-morphisms}). However, it is not in general possible to
\textquotedblleft levelize\textquotedblright\ the \textbf{whole} set $Hom_{%
\mathbf{Pro}\left( \mathbf{K}\right) }\left( \mathbf{X},\mathbf{Y}\right) $
(or an \textbf{infinite} diagram, or a diagram \textbf{with loops}) in $%
\mathbf{Pro}\left( \mathbf{K}\right) $.
\end{remark}

\begin{definition}
\label{Def-Level-morphisms}~

\begin{enumerate}
\item A morphism 
\begin{equation*}
f\in Hom_{\mathbf{Pro}\left( \mathbf{K}\right) }\left( \mathbf{X}=\left(
X_{i}\right) _{i\in \mathbf{I}},\mathbf{Y}=\left( Y_{j}\right) _{j\in 
\mathbf{J}}\right)
\end{equation*}%
is called a \textbf{level morphism} (compare to \cite[\S I.1.3]%
{Mardesic-Segal-MR676973}) iff $\mathbf{I}=\mathbf{J}$, and there is a
morphism%
\begin{equation*}
\gamma :\left( X_{i}\right) _{i\in \mathbf{I}}\longrightarrow \left(
Y_{i}\right) _{i\in \mathbf{I}}:\mathbf{I}\longrightarrow \mathbf{K}
\end{equation*}%
of functors, generating $f$, i.e. such that the following diagram%
\begin{equation*}
%TCIMACRO{%
%\TeXButton{Level morphism}{\begin{diagram}[h=2.0em,w=4.0em,textflow]
%\mathbf{X} & \rTo & \mathbf{Y} \\
%\dTo_{\simeq} & & \dTo^{\simeq} \\
%\underleftarrow{\lim }_{i\in \mathbf{I}}\left( h^{X_{i}}\right) ^{op} & \rTo & \underleftarrow{\lim }_{i\in \mathbf{I}}\left( h^{Y_{i}}\right) ^{op} \\
%\end{diagram}}}%
%BeginExpansion
\begin{diagram}[h=2.0em,w=4.0em,textflow]
\mathbf{X} & \rTo & \mathbf{Y} \\
\dTo_{\simeq} & & \dTo^{\simeq} \\
\underleftarrow{\lim }_{i\in \mathbf{I}}\left( h^{X_{i}}\right) ^{op} & \rTo & \underleftarrow{\lim }_{i\in \mathbf{I}}\left( h^{Y_{i}}\right) ^{op} \\
\end{diagram}%
%EndExpansion
\end{equation*}%
where%
\begin{equation*}
\underleftarrow{\lim }_{i\in \mathbf{I}}\left( h^{X_{i}}\right) ^{op},%
\underleftarrow{\lim }_{i\in \mathbf{I}}\left( h^{Y_{i}}\right) ^{op}\in
\left( \mathbf{Set}^{\mathbf{K}}\right) ^{op},
\end{equation*}%
is commutative. In the notations of Remark \ref{Rem-Level-morphisms} it
means that:%
\begin{eqnarray*}
\varphi &=&\mathbf{1}_{Ob\left( \mathbf{I}\right) }:Ob\left( \mathbf{I}%
\right) \longrightarrow Ob\left( \mathbf{I}\right) \mathbf{,} \\
\lambda \left( \alpha :j_{0}\rightarrow j_{1}\right) &=&\left[ \varphi
\left( j_{1}\right) =j_{1}\overset{\alpha }{\longleftarrow }j_{0}\overset{%
\mathbf{1}_{j_{0}}}{\longrightarrow }j_{0}=\varphi \left( j_{0}\right) %
\right] , \\
f_{i} &=&\gamma _{i},i\in \mathbf{I.}
\end{eqnarray*}

\item A family%
\begin{equation*}
\left( f_{s}:\mathbf{X}_{s}=\left( X_{si}\right) _{i\in \mathbf{I}%
_{s}}\longrightarrow \mathbf{Y}_{s}=\left( Y_{sj}\right) _{j\in \mathbf{J}%
_{s}}\right)
\end{equation*}%
of morphisms in $\mathbf{Pro}\left( \mathbf{K}\right) $ is called a \textbf{%
level family} iff for some $\mathbf{H}$ and for all $s$,%
\begin{equation*}
\mathbf{I}_{s}=\mathbf{J}_{s}=\mathbf{H},
\end{equation*}%
and there is a family of functors%
\begin{equation*}
\alpha _{s}:\left( X_{si}\right) _{i\in \mathbf{H}}\longrightarrow \left(
Y_{si}\right) _{i\in \mathbf{H}},
\end{equation*}%
such that $\alpha _{s}$ generates $f_{s}$ for all $s$.

\item A diagram%
\begin{equation*}
D:\mathbf{G}\longrightarrow \mathbf{Pro}\left( \mathbf{K}\right)
\end{equation*}%
in $\mathbf{Pro}\left( \mathbf{K}\right) $ is called a \textbf{level diagram}
iff for some $\mathbf{H}$ and for all $g\in Ob\left( \mathbf{G}\right) $,%
\begin{equation*}
D\left( g\right) =\left( X_{gi}\right) _{i\in \mathbf{H}},
\end{equation*}%
and there is a diagram%
\begin{equation*}
\alpha :\mathbf{G}\times \mathbf{H}\longrightarrow \mathbf{K},
\end{equation*}%
such that for each 
\begin{equation*}
\left( \beta :g_{1}\longrightarrow g_{2}\right) \in Hom_{\mathbf{G}}\left(
g_{1},g_{2}\right)
\end{equation*}%
the morphism%
\begin{equation*}
\alpha \left( \beta \right) :\alpha \left( g_{1}\times \bullet \right)
\longrightarrow \alpha \left( g_{2}\times \bullet \right) :\mathbf{K}^{%
\mathbf{H}}\longrightarrow \mathbf{K}^{\mathbf{H}}
\end{equation*}%
generates the morphism%
\begin{equation*}
f_{\alpha }:D\left( g_{1}\right) \longrightarrow D\left( g_{2}\right) .
\end{equation*}
\end{enumerate}
\end{definition}

\begin{proposition}
\label{Th-Level-morphisms}\label{Prop-Level-morphisms}Let 
\begin{equation*}
D:\mathbf{G}\longrightarrow \mathbf{Pro}\left( \mathbf{K}\right)
\end{equation*}%
be a diagram in $\mathbf{Pro}\left( \mathbf{K}\right) $, where $\mathbf{G}$
is finite, and \textbf{does not have loops}. Then the diagram is isomorphic
to a level diagram, i.e. $D%
%TCIMACRO{\TeXButton{ISO}{\simeq}}%
%BeginExpansion
\simeq%
%EndExpansion
D^{\prime }$, where%
\begin{equation*}
D^{\prime }:\mathbf{G}\longrightarrow \mathbf{Pro}\left( \mathbf{K}\right)
\end{equation*}%
is a level diagram.
\end{proposition}

\begin{proof}
See \cite[Proposition A.3.3]{Artin-Mazur-MR883959} or \cite[dual to
Proposition 6.4.1]{Kashiwara-Categories-MR2182076}.
\end{proof}

\begin{remark}
See examples of such \textquotedblleft levelization\textquotedblright\ for
one morphism \cite[dual to Corollary 6.1.14]{Kashiwara-Categories-MR2182076}%
, and for a pair of parallel morphisms \cite[dual to Corollary 6.1.15]%
{Kashiwara-Categories-MR2182076}.
\end{remark}

Below are other useful properties of pro-objects.

\begin{proposition}
\label{Prop-Pro-objects-properties}~Let $\mathbf{K}$ be a cocomplete
category. In (\ref{Prop-Pro-objects-properties-Complete}-\ref%
{Prop-Pro-objects-properties-Cofiltered-limits-exact}) below assume that $%
\mathbf{K}$ admits finite limits.

\begin{enumerate}
\item \label{Prop-Pro-objects-properties-Convert-limits-to-colimits}For any $%
\mathbf{Y}\in \mathbf{Pro}\left( \mathbf{K}\right) $ the functor $Hom_{%
\mathbf{Pro}\left( \mathbf{K}\right) }\left( \bullet ,\mathbf{Y}\right) $
converts cofiltered limits into filtered colimits: for a diagram $\left( 
\mathbf{X}_{i}\right) _{i\in \mathbf{I}}$ in $\mathbf{Pro}\left( \mathbf{K}%
\right) $, where $\mathbf{I}$ is cofiltered,%
\begin{equation*}
Hom_{\mathbf{Pro}\left( \mathbf{K}\right) }\left( \underset{i\in \mathbf{I}}{%
\underleftarrow{\lim }}\mathbf{X}_{i},\mathbf{Y}\right) 
%TCIMACRO{\TeXButton{ISO}{\simeq}}%
%BeginExpansion
\simeq%
%EndExpansion
\underset{i\in \mathbf{I}^{op}}{\underrightarrow{\lim }}\left( Hom_{\mathbf{%
Pro}\left( \mathbf{K}\right) }\left( \mathbf{X}_{i},\mathbf{Y}\right)
\right) .
\end{equation*}

\item \label{Prop-Pro-objects-properties-Cocomplete}$\mathbf{Pro}\left( 
\mathbf{K}\right) $ is cocomplete.

\item \label{Prop-Pro-objects-properties-Complete}$\mathbf{Pro}\left( 
\mathbf{K}\right) $ is complete.

\item \label{Prop-Pro-objects-properties-Cofiltered-limits-exact}Cofiltered
limits are exact in $\mathbf{Pro}\left( \mathbf{K}\right) $: for a double
diagram $\left( \mathbf{X}_{i,j}\right) _{i\in \mathbf{I},j\in \mathbf{J}}$
in $\mathbf{Pro}\left( \mathbf{K}\right) $, where $\mathbf{I}$ is
cofiltered, and $\mathbf{J}$ is finite,%
\begin{eqnarray*}
&&\underset{i\in \mathbf{I}}{\underleftarrow{\lim }}~\underset{j\in \mathbf{J%
}}{\underrightarrow{\lim }}~\mathbf{X}_{i,j}%
%TCIMACRO{\TeXButton{ISO}{\simeq}}%
%BeginExpansion
\simeq%
%EndExpansion
\underset{j\in \mathbf{J}}{\underrightarrow{\lim }}~\underset{i\in \mathbf{I}%
}{\underleftarrow{\lim }}~\mathbf{X}_{i,j}, \\
&&\underset{i\in \mathbf{I}}{\underleftarrow{\lim }}~\underset{j\in \mathbf{J%
}}{\underleftarrow{\lim }}~\mathbf{X}_{i,j}%
%TCIMACRO{\TeXButton{ISO}{\simeq}}%
%BeginExpansion
\simeq%
%EndExpansion
\underset{j\in \mathbf{J}}{\underleftarrow{\lim }}~\underset{i\in \mathbf{I}}%
{\underleftarrow{\lim }}~\mathbf{X}_{i,j},
\end{eqnarray*}
\end{enumerate}
\end{proposition}

\begin{proof}
(\textbf{\ref{Prop-Pro-objects-properties-Convert-limits-to-colimits}}) The
statement is dual to \cite[Theorem 6.1.8]{Kashiwara-Categories-MR2182076}.

(\textbf{\ref{Prop-Pro-objects-properties-Cocomplete}}) See \cite[dual to
Corollary 6.1.17]{Kashiwara-Categories-MR2182076}.

(\textbf{\ref{Prop-Pro-objects-properties-Complete}}) See \cite[dual to
Proposition 6.1.18]{Kashiwara-Categories-MR2182076}.

(\textbf{\ref{Prop-Pro-objects-properties-Cofiltered-limits-exact}}) The
statement is dual to \cite[Proposition 6.1.19]%
{Kashiwara-Categories-MR2182076}.
\end{proof}

\subsection{Pairings}

\begin{definition}
\label{Not-Pairings-Hom(Pro-k)-short}\label{Def-Pairings-functors} \label%
{Def-Pairings-pro-modules}Let $\mathbf{D}$ be a small category. Various
bifunctors are defined below:

\begin{enumerate}
\item \label{Def-Pairings-pro-modules-Hom(Pro-k)}%
\begin{equation*}
\left\langle \bullet ,\bullet \right\rangle :\mathbf{Pro}\left( k\right)
^{op}\times \mathbf{Mod}\left( k\right) \longrightarrow \mathbf{Mod}\left(
k\right) .
\end{equation*}%
If%
\begin{equation*}
\mathbf{A}=\left( A_{i}\right) _{i\in \mathbf{I}}\in \mathbf{Pro}\left(
k\right)
\end{equation*}%
is a pro-module, and $G\in \mathbf{Mod}\left( k\right) $, let%
\begin{equation*}
\left\langle \mathbf{A},G\right\rangle 
%TCIMACRO{\TeXButton{assigned}{{:=}}}%
%BeginExpansion
{:=}%
%EndExpansion
Hom_{\mathbf{Pro}\left( k\right) }\left( \mathbf{A},G\right) =%
\underrightarrow{\lim }_{i\in \mathbf{I}}Hom_{\mathbf{Mod}\left( k\right)
}\left( A_{i},G\right) \in \mathbf{Mod}\left( k\right) .
\end{equation*}

\item \label{Def-Pairings-functors-Hom(Pro-k)}%
\begin{equation*}
\left\langle \bullet ,\bullet \right\rangle :\mathbf{pCS}\left( \mathbf{D},%
\mathbf{Pro}\left( k\right) \right) ^{op}\times \mathbf{Mod}\left( k\right)
\longrightarrow \mathbf{pS}\left( \mathbf{D},\mathbf{Mod}\left( k\right)
\right) .
\end{equation*}%
If%
\begin{equation*}
\mathcal{A}:\mathbf{D}\longrightarrow \mathbf{Pro}\left( k\right)
\end{equation*}%
is a functor, and $G\in \mathbf{Mod}\left( k\right) $, let%
\begin{eqnarray*}
\left\langle \mathcal{A},G\right\rangle &=&Hom_{\mathbf{Pro}\left( k\right)
}\left( \mathcal{A},G\right) 
%TCIMACRO{\TeXButton{assigned}{{:=}}}%
%BeginExpansion
{:=}%
%EndExpansion
\left[ U\longmapsto Hom_{\mathbf{Pro}\left( k\right) }\left( \mathcal{A}%
\left( U\right) ,G\right) \right] , \\
\left\langle \mathcal{A},G\right\rangle &:&\mathbf{D}^{op}\longrightarrow 
\mathbf{Mod}\left( k\right) .
\end{eqnarray*}

\item \label{Def-Pairings-functors-Set}%
\begin{equation*}
\bullet \otimes _{\mathbf{Set}}\bullet :\mathbf{K\times Set}\longrightarrow 
\mathbf{K.}
\end{equation*}%
If $A\in \mathbf{K}$ (say, $\mathbf{K=Mod}\left( k\right) $ or $\mathbf{K=Pro%
}\left( k\right) $), and $B\in \mathbf{Set}$, let%
\begin{equation*}
A\otimes _{\mathbf{Set}}B=B\otimes _{\mathbf{Set}}A=\dcoprod\limits_{B}A
\end{equation*}%
be the coproduct in $\mathbf{K}$ of $B$ copies of $A$.

\item \label{Def-Pairings-functors-Pro(Set)}\label{Def-Pro(Set)-Ten-Set-K}%
\begin{equation*}
\bullet \otimes _{\mathbf{Set}}\bullet :\mathbf{K\times Pro}\left( \mathbf{%
Set}\right) \longrightarrow \mathbf{Pro}\left( \mathbf{K}\right) \mathbf{.}
\end{equation*}%
Let $\mathbf{Y}=\left( Y_{i}\right) _{i\in \mathbf{I}}\in \mathbf{Pro}\left( 
\mathbf{Set}\right) $, and $X\in \mathbf{K}$. Define%
\begin{equation*}
X\otimes _{\mathbf{Set}}\mathbf{Y}=\mathbf{Y}\otimes _{\mathbf{Set}}X\in 
\mathbf{Pro}\left( \mathbf{K}\right)
\end{equation*}%
by%
\begin{equation*}
X\otimes _{\mathbf{Set}}\mathbf{Y}=\left( X\otimes _{\mathbf{Set}%
}Y_{i}\right) _{i\in \mathbf{I}}.
\end{equation*}

\item \label{Def-Pairings-functors-Set-X}%
\begin{equation*}
\bullet \otimes _{\mathbf{Set}^{\mathbf{D}}}\bullet :\mathbf{pCS}\left( 
\mathbf{D},\mathbf{Pro}\left( \mathbf{K}\right) \right) \mathbf{\times 
\mathbf{pS}\left( D,\mathbf{Set}\right) }\longrightarrow \mathbf{Pro}\left( 
\mathbf{K}\right) \mathbf{.}
\end{equation*}%
If%
\begin{eqnarray*}
\mathcal{A} &:&\mathbf{D}\longrightarrow \mathbf{Pro}\left( \mathbf{K}%
\right) , \\
\mathcal{B} &:&\mathbf{D}^{op}\longrightarrow \mathbf{Set},
\end{eqnarray*}%
are functors, let 
\begin{equation*}
\mathcal{A}\otimes _{\mathbf{Set}^{\mathbf{D}}}\mathcal{B}\in \mathbf{Pro}%
\left( \mathbf{K}\right)
\end{equation*}%
be the \textbf{coend }\cite[Chapter IX.6]{Mac-Lane-Categories-1998-MR1712872}
of the bifunctor $\left( U,V\right) \mapsto \mathcal{A}\left( U\right)
\otimes _{\mathbf{Set}}\mathcal{B}\left( V\right) $, i.e.%
\begin{equation*}
\mathcal{A\otimes }_{\mathbf{Set}^{\mathbf{C}}}\mathcal{B}%
%TCIMACRO{\TeXButton{assigned}{{:=}}}%
%BeginExpansion
{:=}%
%EndExpansion
%TCIMACRO{\TeXButton{coker}{\coker}}%
%BeginExpansion
\coker%
%EndExpansion
\left( \dcoprod\limits_{U\rightarrow V}\mathcal{A}\left( U\right) \otimes _{%
\mathbf{Set}}\mathcal{B}\left( V\right) \rightrightarrows \dcoprod\limits_{U}%
\mathcal{A}\left( U\right) \otimes _{\mathbf{Set}}\mathcal{B}\left( U\right)
\right) .
\end{equation*}

\item \label{Def-Pairings-functors-Hom-Set-X}%
\begin{equation*}
Hom_{\mathbf{Set}^{\mathbf{D}}}\left( \bullet ,\bullet \right) :\mathbf{pS}%
\left( \mathbf{D},\mathbf{Set}\right) ^{op}\times \mathbf{pS}\left( \mathbf{D%
},\mathbf{K}\right)
\end{equation*}%
If%
\begin{eqnarray*}
\mathcal{A} &:&\mathbf{D}^{op}\longrightarrow \mathbf{K}, \\
\mathcal{B} &:&\mathbf{D}^{op}\longrightarrow \mathbf{Set},
\end{eqnarray*}%
are functors, let 
\begin{equation*}
Hom_{\mathbf{Set}^{\mathbf{D}}}\left( \mathcal{B},\mathcal{A}\right) \in 
\mathbf{Mod}\left( k\right)
\end{equation*}%
be the \textbf{end }\cite[Chapter IX.6]{Mac-Lane-Categories-1998-MR1712872}
of the bifunctor $\left( U,V\right) \mapsto Hom_{\mathbf{Set}}\left( 
\mathcal{B}\left( U\right) ,\mathcal{A}\left( V\right) \right) $, i.e.%
\begin{equation*}
Hom_{\mathbf{Set}^{\mathbf{C}}}\left( \mathcal{B},\mathcal{A}\right) 
%TCIMACRO{\TeXButton{assigned}{{:=}}}%
%BeginExpansion
{:=}%
%EndExpansion
\ker \left( \dprod\limits_{U}Hom_{\mathbf{Set}}\left( \mathcal{B}\left(
U\right) ,\mathcal{A}\left( U\right) \right) \rightrightarrows
\dprod\limits_{U\rightarrow V}Hom_{\mathbf{Set}}\left( \mathcal{B}\left(
U\right) ,\mathcal{A}\left( V\right) \right) \right) .
\end{equation*}
\end{enumerate}
\end{definition}

\subsection{Quasi-projective pro-modules}

\begin{definition}
\label{Def-quasi-projective-pro-module}\label{Def-quasi-projective}A
pro-module $\mathbf{P}$ is called \textbf{quasi-projective} iff the functor%
\begin{equation*}
Hom_{\mathbf{Pro}\left( k\right) }\left( \mathbf{P},\bullet \right) :\mathbf{%
Mod}\left( k\right) \longrightarrow \mathbf{Mod}\left( k\right)
\end{equation*}%
is exact (see \cite[dual to Definition 15.2.1]%
{Kashiwara-Categories-MR2182076}).
\end{definition}

\begin{proposition}
\label{Prop-quasi-projective}A pro-module $\mathbf{P}$ is quasi-projective
iff it is isomorphic to a pro-module $\left( Q_{i}\right) _{i\in \mathbf{I}}$
where all modules $Q_{i}\in \mathbf{Mod}\left( k\right) $ are projective.
\end{proposition}

\begin{proof}
The statement is dual to \cite[Proposition 15.2.3]%
{Kashiwara-Categories-MR2182076}.
\end{proof}

\begin{remark}
\label{Rem-not-enough-projectives}The category $\mathbf{Pro}\left( k\right) $
does not have enough projectives (compare with \cite[Corollary 15.1.3]%
{Kashiwara-Categories-MR2182076}). However, it has enough quasi-projectives
(see Proposition \ref{Prop-Pro-modules-properties}(\ref%
{Prop-enough-quasi-projectives}) below).
\end{remark}

\begin{definition}
\label{Def-quasi-noetherian}A commutative ring $k$ is called \textbf{%
quasi-noetherian} iff 
\begin{equation*}
\left\langle \mathbf{P},T\right\rangle =Hom_{\mathbf{Pro}\left( k\right)
}\left( \mathbf{P},T\right)
\end{equation*}%
is an injective $k$-module for any quasi-projective pro-module $\mathbf{P}$
and an injective $k$-module $T$.
\end{definition}

\begin{proposition}
\label{Prop-Noetherian-Quasi-noetherian}A noetherian ring is
quasi-noetherian.
\end{proposition}

\begin{proof}
See \cite[Proposition 2.28]%
{Prasolov-universal-coefficients-formula-2013-MR3095217}.
\end{proof}

\begin{proposition}
\label{Prop-Pro(K)-abelian}If $\mathbf{K}$ is an abelian category, then $%
\mathbf{Pro}\left( \mathbf{K}\right) $ is an abelian category as well.
\end{proposition}

\begin{proof}
See \cite[dual to Theorem 8.6.5(i)]{Kashiwara-Categories-MR2182076}.
\end{proof}

\begin{notation}
\label{Not-pro-modules-cogenerators}For a $k$-module $M$, denote by $M^{\ast
}$ the following $k$-module:%
\begin{equation*}
M^{\ast }%
%TCIMACRO{\TeXButton{assigned}{{:=}}}%
%BeginExpansion
{:=}%
%EndExpansion
Hom_{\mathbb{Z}}\left( M,\mathbb{Q}/\mathbb{Z}\right) .
\end{equation*}
\end{notation}

\begin{proposition}
\label{Prop-Pro-modules-properties}~

\begin{enumerate}
\item \label{Prop-Pro-modules-properties-abelian-(co)cocomplete} \label%
{Prop-Pro-modules-properties-abelian-(co)complete}The category $\mathbf{Pro}%
\left( k\right) $ is abelian, complete and cocomplete, and satisfies both
the $AB3$ and $AB3^{\ast }$ axioms (\cite[1.5]%
{Grothendieck-Tohoku-1957-MR0102537}, \cite[Ch. 5.8]%
{Bucur-Deleanu-1968-Introduction-categories-functors-MR0236236}).

\item \label{Prop-Pro-modules-properties-colimit}For any diagram%
\begin{equation*}
\mathbf{X}:\mathbf{I}\longrightarrow \mathbf{Pro}\left( k\right)
\end{equation*}%
and any $T\in \mathbf{Mod}\left( k\right) $ (not necessarily injective!)%
\begin{equation*}
\left\langle \underrightarrow{\lim }_{i\in \mathbf{I}}\mathbf{X}%
_{i},T\right\rangle 
%TCIMACRO{\TeXButton{ISO}{\simeq}}%
%BeginExpansion
\simeq%
%EndExpansion
\underleftarrow{\lim }_{i\in \mathbf{I}}\left\langle \mathbf{X}%
_{i},T\right\rangle
\end{equation*}%
in $\mathbf{Mod}\left( k\right) $.

\item \label{Prop-Pro-modules-properties-cofiltered-limit}For any \textbf{%
cofiltered} diagram%
\begin{equation*}
\mathbf{X}:\mathbf{I}\longrightarrow \mathbf{Pro}\left( k\right)
\end{equation*}%
and any $T\in \mathbf{Mod}\left( k\right) $ (not necessarily injective!)%
\begin{equation*}
\left\langle \underleftarrow{\lim }_{i\in \mathbf{I}}\mathbf{X}%
_{i},T\right\rangle 
%TCIMACRO{\TeXButton{ISO}{\simeq}}%
%BeginExpansion
\simeq%
%EndExpansion
\underrightarrow{\lim }_{i\in \mathbf{I}}\left\langle \mathbf{X}%
_{i},T\right\rangle
\end{equation*}%
in $\mathbf{Mod}\left( k\right) $.

\item \label{Prop-Pro-modules-properties-product}For any family $\left( 
\mathbf{X}_{i}\right) _{i\in I}$ in $\mathbf{Pro}\left( k\right) $ and any $%
T\in \mathbf{Mod}\left( k\right) $ (not necessarily injective!)%
\begin{equation*}
\left\langle \dprod\limits_{i\in I}\mathbf{X}_{i},T\right\rangle 
%TCIMACRO{\TeXButton{ISO}{\simeq}}%
%BeginExpansion
\simeq%
%EndExpansion
\dbigoplus\limits_{i\in I}\left\langle \mathbf{X}_{i},T\right\rangle
\end{equation*}%
in $\mathbf{Mod}\left( k\right) $.

\item \label{Prop-Pro-modules-properties-limit}For \textbf{any} (not
necessarily cofiltered\textbf{)} diagram%
\begin{equation*}
\mathbf{X}:\mathbf{I}\longrightarrow \mathbf{Pro}\left( k\right)
\end{equation*}%
and any injective $T\in \mathbf{Mod}\left( k\right) $%
\begin{equation*}
\left\langle \underleftarrow{\lim }_{i\in \mathbf{I}}\mathbf{X}%
_{i},T\right\rangle 
%TCIMACRO{\TeXButton{ISO}{\simeq}}%
%BeginExpansion
\simeq%
%EndExpansion
\underrightarrow{\lim }_{i\in \mathbf{I}}\left\langle \mathbf{X}%
_{i},T\right\rangle
\end{equation*}%
in $\mathbf{Mod}\left( k\right) $.

\item \label{Prop-enough-quasi-projectives}\label%
{Prop-Pro-modules-properties-quasi-projective}For an arbitrary pro-module $%
\mathbf{M}\in \mathbf{Pro}\left( k\right) $, there exists a functorial
surjection%
\begin{equation*}
\mathbf{F}\left( \mathbf{M}\right) \twoheadrightarrow \mathbf{M},
\end{equation*}%
where $\mathbf{F}\left( \mathbf{M}\right) $ is quasi-projective.

\item \label{Prop-Pro-modules-properties-Zero}Let $\mathbf{M}\in \mathbf{Pro}%
\left( k\right) $. Then $\mathbf{M}%
%TCIMACRO{\TeXButton{ISO}{\simeq}}%
%BeginExpansion
\simeq%
%EndExpansion
\mathbf{0}$ iff $\left\langle \mathbf{M},T\right\rangle =0$ for any
injective $T\in \mathbf{Mod}\left( k\right) $.

\item \label{Prop-Pro-modules-properties-Homology}Let%
\begin{equation*}
\mathcal{E=}\left( \mathbf{M}\overset{\alpha }{\longleftarrow }\mathbf{N}%
\overset{\beta }{\longleftarrow }\mathbf{K}\right)
\end{equation*}%
be a sequence of morphisms in $\mathbf{Pro}\left( k\right) $ with $\beta
\circ \alpha =0$, and let $T\in \mathbf{Mod}\left( k\right) $ be injective.
Then%
\begin{equation*}
H\left( \mathcal{E}\right) 
%TCIMACRO{\TeXButton{assigned}{{:=}}}%
%BeginExpansion
{:=}%
%EndExpansion
\frac{\ker \left( \alpha \right) }{im\left( \beta \right) }
\end{equation*}%
satisfies%
\begin{equation*}
\left\langle H\left( \mathcal{E}\right) ,T\right\rangle 
%TCIMACRO{\TeXButton{ISO}{\simeq}}%
%BeginExpansion
\simeq%
%EndExpansion
H\left( \left\langle \mathcal{E},T\right\rangle \right) 
%TCIMACRO{\TeXButton{assigned}{{:=}}}%
%BeginExpansion
{:=}%
%EndExpansion
\frac{\ker \left( \left\langle \beta ,T\right\rangle \right) }{im\left(
\left\langle \alpha ,T\right\rangle \right) }.
\end{equation*}

\item \label{Prop-Pro-modules-properties-Exact}Let%
\begin{equation*}
\mathcal{E=}\left( \mathbf{M}\overset{\alpha }{\longleftarrow }\mathbf{N}%
\overset{\beta }{\longleftarrow }\mathbf{K}\right)
\end{equation*}%
be a sequence of morphisms in $\mathbf{Pro}\left( k\right) $ with $\beta
\circ \alpha =0$. Then $\mathcal{E}$ is exact iff the sequence%
\begin{equation*}
\left\langle \mathbf{M},T\right\rangle \overset{\left\langle \alpha
,T\right\rangle }{\longrightarrow }\left\langle \mathbf{N},T\right\rangle 
\overset{\left\langle \beta ,T\right\rangle }{\longrightarrow }\left\langle 
\mathbf{K},T\right\rangle
\end{equation*}%
is exact in $\mathbf{Mod}\left( k\right) $ for all injective $T\in \mathbf{%
Mod}\left( k\right) $.

\item \label{Prop-Pro-modules-properties-Injective}Let $T\in \mathbf{Mod}%
\left( k\right) $ be an injective module. Then the corresponding rudimentary
(Remark \ref{Rem-Rudimentary}) pro-module $T$ is an injective object of $%
\mathbf{Pro}\left( k\right) $.

\item \label{Prop-Pro-modules-properties-AB4}The category $\mathbf{Pro}%
\left( k\right) $ satisfies the $AB4$ axiom (\cite[1.5]%
{Grothendieck-Tohoku-1957-MR0102537}, \cite[Ch. 5.8]%
{Bucur-Deleanu-1968-Introduction-categories-functors-MR0236236}).

\item \label{Prop-Pro-modules-properties-AB4*}The category $\mathbf{Pro}%
\left( k\right) $ satisfies the $AB4^{\ast }$ axiom (\cite[1.5]%
{Grothendieck-Tohoku-1957-MR0102537}, \cite[Ch. 5.8]%
{Bucur-Deleanu-1968-Introduction-categories-functors-MR0236236}).

\item \label{Prop-Pro-modules-properties-cofiltered-limits-exact}\label%
{Prop-Pro-modules-properties-AB5*}The category $\mathbf{Pro}\left( k\right) $
satisfies the $AB5^{\ast }$ axiom (\cite[1.5]%
{Grothendieck-Tohoku-1957-MR0102537}, \cite[Ch. 5.8]%
{Bucur-Deleanu-1968-Introduction-categories-functors-MR0236236}): cofiltered
limits are exact in the category $\mathbf{Pro}\left( k\right) $.

\item \label{Prop-Pro-modules-properties-cogenerators}The \textbf{class} (%
\textbf{not} a set)%
\begin{equation*}
\mathfrak{G}=\left\{ \mathbf{G}\left( S\right) ~|~S\in \mathbf{Set}\right\}
\subseteq \mathbf{Pro}\left( k\right) ,
\end{equation*}%
where $\mathbf{G}\left( S\right) $ is the rudimentary pro-module (Remark \ref%
{Rem-Rudimentary}) corresponding to the $k$-module%
\begin{equation*}
\dprod\limits_{S}k^{\ast }=\dprod\limits_{S}Hom_{\mathbf{Ab}}\left( k,%
\mathbb{Q}/\mathbb{Z}\right)
\end{equation*}%
forms a class of cogenerators (\cite[1.9]{Grothendieck-Tohoku-1957-MR0102537}%
, \cite[Ch. 5.9]%
{Bucur-Deleanu-1968-Introduction-categories-functors-MR0236236}) of the
category $\mathbf{Pro}\left( k\right) $.
\end{enumerate}
\end{proposition}

\begin{proof}
(\textbf{\ref{Prop-Pro-modules-properties-abelian-(co)cocomplete}}) It
follows from Proposition \ref{Prop-Pro(K)-abelian} that $\mathbf{Pro}\left(
k\right) $ is abelian. Due to Proposition \ref{Prop-Pro-objects-properties} (%
\ref{Prop-Pro-objects-properties-Cocomplete}, \ref%
{Prop-Pro-objects-properties-Complete}), $\mathbf{Pro}\left( k\right) $ is
complete and cocomplete. $AB3$ and $AB3^{\ast }$ follow immediately.

(\textbf{\ref{Prop-Pro-modules-properties-colimit}}) Follows from the
definition of a colimit.

(\textbf{\ref{Prop-Pro-modules-properties-cofiltered-limit}}) Follows from
Proposition \ref{Prop-Pro-objects-properties} (\ref%
{Prop-Pro-objects-properties-Convert-limits-to-colimits}).

(\textbf{\ref{Prop-Pro-modules-properties-product}}) Let%
\begin{equation*}
\mathbf{Fin}\left( I\right) =\mathbf{Cat}\left( X\left( I\right) \right)
^{op}
\end{equation*}%
(see Example \ref{Ex-(co)filtered-poset}) where $X\left( I\right) $ is the
set of \textbf{finite} subsets of $I$, ordered by inclusion. Then $X\left(
I\right) $ is a directed poset, and $\mathbf{Fin}\left( I\right) $ is a
cofiltered category (see Example \ref{Ex-(co)filtered-poset} again). It is
easy to check that%
\begin{equation*}
\dprod\limits_{i\in I}\mathbf{X}_{i}%
%TCIMACRO{\TeXButton{ISO}{\simeq}}%
%BeginExpansion
\simeq%
%EndExpansion
\underleftarrow{\lim }_{A\in \mathbf{Fin}\left( I\right) }\left[
\dprod\limits_{j\in A}\mathbf{X}_{j}\right] .
\end{equation*}%
It follows from the statement (\ref%
{Prop-Pro-modules-properties-cofiltered-limit}) of our theorem that%
\begin{eqnarray*}
&&\left\langle \dprod\limits_{i\in I}\mathbf{X}_{i}%
%TCIMACRO{\TeXButton{ISO}{\simeq}}%
%BeginExpansion
\simeq%
%EndExpansion
,T\right\rangle 
%TCIMACRO{\TeXButton{ISO}{\simeq}}%
%BeginExpansion
\simeq%
%EndExpansion
\underrightarrow{\lim }_{A\in \mathbf{Fin}\left( I\right) ^{op}}\left\langle
\dprod\limits_{j\in A}\mathbf{X}_{j},T\right\rangle 
%TCIMACRO{\TeXButton{ISO}{\simeq} }%
%BeginExpansion
\simeq
%EndExpansion
\\
&&%
%TCIMACRO{\TeXButton{ISO}{\simeq}}%
%BeginExpansion
\simeq%
%EndExpansion
\underrightarrow{\lim }_{A\in \mathbf{Fin}\left( I\right)
^{op}}\dbigoplus\limits_{j\in A}\left\langle \mathbf{X}_{j},T\right\rangle 
%TCIMACRO{\TeXButton{ISO}{\simeq}}%
%BeginExpansion
\simeq%
%EndExpansion
\dbigoplus\limits_{j\in I}\left\langle \mathbf{X}_{i},T\right\rangle .
\end{eqnarray*}

(\textbf{\ref{Prop-Pro-modules-properties-limit}}) Limits in any category
can be constructed as combinations of products and kernels. Let $T\in 
\mathbf{Mod}\left( k\right) $. It follows from (\ref%
{Prop-Pro-modules-properties-product}) that the pairing $\left\langle
\bullet ,T\right\rangle $ converts products into coproducts. If $T$ is
injective, then $\left\langle \bullet ,T\right\rangle $ converts kernels
into cokernels. Finally, $\left\langle \bullet ,T\right\rangle $ converts
arbitrary limits into colimits.

(\textbf{\ref{Prop-Pro-modules-properties-quasi-projective}}) The statement
is dual to the rather complicated Theorem 15.2.5 from \cite%
{Kashiwara-Categories-MR2182076}. However, the proof is much simpler in our
case. Given $\mathbf{M}=\left( M_{i}\right) _{i\in \mathbf{I}}$, let 
\begin{equation*}
\mathbf{F}\left( \mathbf{M}\right) =\left( Q_{i}\right) _{i\in \mathbf{I}},
\end{equation*}%
where $Q_{i}=F\left( M_{i}\right) $ is the free $k$-module generated by the
set of symbols $\left( \left[ m\right] \right) _{m\in M}$. A family of
epimorphisms 
\begin{equation*}
f_{i}:Q_{i}\longrightarrow M_{i}~\left( f_{i}\left( \dsum\limits_{j}\alpha
_{j}\left[ m_{j}\right] \right) =\dsum\limits_{j}\alpha _{j}m_{j},\alpha
_{j}\in k,m_{j}\in M_{i}\right) ,
\end{equation*}%
defines an epimorphism $f:\mathbf{F}\left( \mathbf{M}\right) \rightarrow 
\mathbf{M}$.%
\begin{equation*}
\mathbf{F}\left( \mathbf{M}\right) =\left( F\left( M_{i}\right) \right)
_{j\in \mathbf{I}}
\end{equation*}%
is quasi-projective (Proposition \ref{Prop-quasi-projective}), and the
epimorphism $F\left( \mathbf{M}\right) \twoheadrightarrow \mathbf{M}$ is as
desired.

(\textbf{\ref{Prop-Pro-modules-properties-Zero}}) The \textquotedblleft only
if\textquotedblright\ part is trivial. Assume now that $\mathbf{M}$ is 
\textbf{not} isomorphic to $\mathbf{0}$. Since%
\begin{equation*}
\mathbf{Pro}\left( k\right) \hookrightarrow \left( \mathbf{Set}^{\mathbf{Mod}%
\left( k\right) }\right) ^{op}
\end{equation*}%
is a full embedding (by definition!), there exists a $N\in \mathbf{Mod}%
\left( k\right) $ with%
\begin{equation*}
\left\langle \mathbf{M},N\right\rangle =Hom_{\mathbf{Pro}\left( k\right)
}\left( \mathbf{M},N\right) \neq 0.
\end{equation*}%
Choose an embedding $N\hookrightarrow T$ into an injective $k$-module. Then%
\begin{equation*}
\left\langle \mathbf{M},N\right\rangle \longrightarrow \left\langle \mathbf{M%
},T\right\rangle
\end{equation*}%
is a monomorphism. It follows that $\left\langle \mathbf{M},T\right\rangle
\neq 0$ as well.

(\textbf{\ref{Prop-Pro-modules-properties-Homology}}) Due to Proposition \ref%
{Prop-Level-morphisms}, one can assume that $\mathcal{E}$ is a level diagram:%
\begin{equation*}
%TCIMACRO{%
%\TeXButton{Level diagram}{\begin{diagram}
%\left( M_{i}\right) _{i\in \mathbf{I}} & \lTo^{\left( \alpha _{i}\right) } & \left( N_{i}\right) _{i\in \mathbf{I}} & \lTo^{\left( \beta _{i}\right) } & \left( K_{i}\right) _{i\in \mathbf{I}}.
%\end{diagram}}}%
%BeginExpansion
\begin{diagram}
\left( M_{i}\right) _{i\in \mathbf{I}} & \lTo^{\left( \alpha _{i}\right) } & \left( N_{i}\right) _{i\in \mathbf{I}} & \lTo^{\left( \beta _{i}\right) } & \left( K_{i}\right) _{i\in \mathbf{I}}.
\end{diagram}%
%EndExpansion
\end{equation*}%
Since $T$ is injective, the sequences%
\begin{equation*}
\left\langle \mathcal{E}_{i},T\right\rangle =\left[ 
%TCIMACRO{%
%\TeXButton{Level diagram}{\begin{diagram}
%\left\langle M_{i},T\right\rangle & \rTo^{\left\langle \alpha _{i},T\right\rangle } & \left\langle N_{i},T\right\rangle & \rTo^{\left\langle \beta _{i},T\right\rangle } & \left\langle K_{i},T\right\rangle
%\end{diagram}}}%
%BeginExpansion
\begin{diagram}
\left\langle M_{i},T\right\rangle & \rTo^{\left\langle \alpha _{i},T\right\rangle } & \left\langle N_{i},T\right\rangle & \rTo^{\left\langle \beta _{i},T\right\rangle } & \left\langle K_{i},T\right\rangle
\end{diagram}%
%EndExpansion
\right]
\end{equation*}%
satisfy%
\begin{equation*}
H\left\langle \mathcal{E}_{i},T\right\rangle 
%TCIMACRO{\TeXButton{ISO}{\simeq}}%
%BeginExpansion
\simeq%
%EndExpansion
\left\langle H\left( \mathcal{E}_{i}\right) ,T\right\rangle .
\end{equation*}%
The category $\mathbf{I}^{op}$ is filtered, and filtered colimits are exact
in the category $\mathbf{Mod}\left( k\right) $, therefore%
\begin{equation*}
\left\langle H\left( \mathcal{E}\right) ,T\right\rangle 
%TCIMACRO{\TeXButton{ISO}{\simeq}}%
%BeginExpansion
\simeq%
%EndExpansion
\underrightarrow{\lim }_{i\in \mathbf{I}^{op}}\left\langle H\left( \mathcal{E%
}_{i}\right) ,T\right\rangle 
%TCIMACRO{\TeXButton{ISO}{\simeq}}%
%BeginExpansion
\simeq%
%EndExpansion
\underrightarrow{\lim }_{i\in \mathbf{I}^{op}}H\left\langle \mathcal{E}%
_{i},T\right\rangle 
%TCIMACRO{\TeXButton{ISO}{\simeq}}%
%BeginExpansion
\simeq%
%EndExpansion
H\left( \left\langle \mathcal{E},T\right\rangle \right) .
\end{equation*}

(\textbf{\ref{Prop-Pro-modules-properties-Exact}}) It follows from the
statement (\ref{Prop-Pro-modules-properties-Homology}) of our theorem that%
\begin{equation*}
\left\langle H\left( \mathcal{E}\right) ,T\right\rangle 
%TCIMACRO{\TeXButton{ISO}{\simeq}}%
%BeginExpansion
\simeq%
%EndExpansion
H\left( \left\langle \mathcal{E},T\right\rangle \right) .
\end{equation*}%
Applying (\ref{Prop-Pro-modules-properties-Zero}) of our theorem, one gets%
\begin{equation*}
H\left( \mathcal{E}\right) =\mathbf{0}\iff \forall \text{(injective }T\text{)%
}\left[ H\left( \left\langle \mathcal{E},T\right\rangle \right) =0\right] ,
\end{equation*}%
therefore $\mathcal{E}$ is exact iff $\left\langle \mathcal{E}%
,T\right\rangle $ is exact for all injective $T\in \mathbf{Mod}\left(
k\right) $.

(\textbf{\ref{Prop-Pro-modules-properties-Injective}}) Follows easily from (%
\ref{Prop-Pro-modules-properties-Exact}).

(\textbf{\ref{Prop-Pro-modules-properties-AB4}}) Let%
\begin{equation*}
\left( f_{i}:A_{i}\longrightarrow B_{i}\right) _{i\in I}
\end{equation*}%
be a family of monomorphisms, and $T\in \mathbf{Mod}\left( k\right) $ be
injective. Then all the homomorphisms%
\begin{equation*}
\left\langle f_{i},T\right\rangle :\left\langle B_{i},T\right\rangle
\longrightarrow \left\langle A_{i},T\right\rangle
\end{equation*}%
are epimorphisms in $\mathbf{Mod}\left( k\right) $. Therefore, the
homomorphism%
\begin{equation*}
\left\langle \dbigoplus\limits_{i\in I}f_{i},T\right\rangle
=\dprod\limits_{i\in I}\left\langle f_{i},T\right\rangle
:\dprod\limits_{i\in I}\left\langle B_{i},T\right\rangle =\left\langle
\dbigoplus\limits_{i\in I}B_{i},T\right\rangle \longrightarrow \left\langle
\dbigoplus\limits_{i\in I}A_{i},T\right\rangle =\dprod\limits_{i\in
I}\left\langle B_{i},T\right\rangle
\end{equation*}%
is an epimorphism in $\mathbf{Mod}\left( k\right) $ for any injective $T$.
It follows that $\dbigoplus\limits_{i\in I}f_{i}$ is a monomorphism in $%
\mathbf{Pro}\left( k\right) $.

(\textbf{\ref{Prop-Pro-modules-properties-AB4*}}) Let%
\begin{equation*}
\left( f_{i}:A_{i}\longrightarrow B_{i}\right) _{i\in I}
\end{equation*}%
be a family of epimorphisms, and $T\in \mathbf{Mod}\left( k\right) $ be
injective. Then all the homomorphisms%
\begin{equation*}
\left\langle f_{i},T\right\rangle :\left\langle B_{i},T\right\rangle
\longrightarrow \left\langle A_{i},T\right\rangle
\end{equation*}%
are monomorphisms in $\mathbf{Mod}\left( k\right) $. Therefore, the
homomorphism%
\begin{equation*}
\left\langle \dprod\limits_{i\in I}f_{i},T\right\rangle
=\dbigoplus\limits_{i\in I}\left\langle f_{i},T\right\rangle
:\dbigoplus\limits_{i\in I}\left\langle B_{i},T\right\rangle =\left\langle
\dprod\limits_{i\in I}B_{i},T\right\rangle \longrightarrow \left\langle
\dprod\limits_{i\in I}A_{i},T\right\rangle =\dbigoplus\limits_{i\in
I}\left\langle B_{i},T\right\rangle
\end{equation*}%
is a monomorphism in $\mathbf{Mod}\left( k\right) $ for any injective $T$.
It follows that $\dprod\limits_{i\in I}f_{i}$ is an epimorphism in $\mathbf{%
Pro}\left( k\right) $.

(\textbf{\ref{Prop-Pro-modules-properties-cofiltered-limits-exact}}) Follows
from Proposition \ref{Prop-Pro-objects-properties} (\ref%
{Prop-Pro-objects-properties-Cofiltered-limits-exact}).

(\textbf{\ref{Prop-Pro-modules-properties-cogenerators}}) Since%
\begin{equation*}
Hom_{\mathbf{Mod}\left( k\right) }\left( \bullet ,k^{\ast }\right) 
%TCIMACRO{\TeXButton{ISO}{\simeq}}%
%BeginExpansion
\simeq%
%EndExpansion
Hom_{\mathbf{Ab}}\left( \bullet ,\mathbb{Q}/\mathbb{Z}\right) ,
\end{equation*}%
and $\mathbb{Q}/\mathbb{Z}$ is a cogenerator in the category $\mathbf{Ab}$, $%
k^{\ast }$ is an \textbf{injective} cogenerator in $\mathbf{Mod}\left(
k\right) $. In fact, $k^{\ast }$ is \textbf{injective} in $\mathbf{Pro}%
\left( k\right) $ as well. Indeed, it is enough to apply part (\ref%
{Prop-Pro-modules-properties-Injective}) of our theorem to $T=k^{\ast }$.

Let now%
\begin{equation*}
f:M\twoheadrightarrow N
\end{equation*}%
be a non-trivial (not an isomorphism!) epimorphism in $\mathbf{Pro}\left(
k\right) $. Let%
\begin{equation*}
K=\ker f\neq 0.
\end{equation*}%
We can assume that $f$ and $h:K\rightarrowtail M$ are level morphisms:%
\begin{equation*}
%TCIMACRO{%
%\TeXButton{Level-morphisms}{\begin{diagram}
%0 & \rTo & \left( K_{i}{{:=}}\ker f_{i}\right) _{i\in \mathbf{I}} & \rTo^{h=\left( h_{i}\right)} & \left( M_{i}\right) _{i\in \mathbf{I}} & \rTo^{f=\left( f_{i}\right)} & \left( N_{i}\right) _{i\in \mathbf{I}}. 
%\end{diagram}
%}}%
%BeginExpansion
\begin{diagram}
0 & \rTo & \left( K_{i}{{:=}}\ker f_{i}\right) _{i\in \mathbf{I}} & \rTo^{h=\left( h_{i}\right)} & \left( M_{i}\right) _{i\in \mathbf{I}} & \rTo^{f=\left( f_{i}\right)} & \left( N_{i}\right) _{i\in \mathbf{I}}. 
\end{diagram}
%
%EndExpansion
\end{equation*}%
Due to Corollary \ref{Cor-Zero-pro-object}, there exists an $i\in \mathbf{I}$%
, such that $K\left( t\right) \neq 0$ for any $t:j\rightarrow i$. It follows
that $K_{i}\neq 0$. Let%
\begin{equation*}
S=\left\{ \left( t:j\rightarrow i\right) \in \mathbf{I}\right\} ,
\end{equation*}%
and let%
\begin{equation*}
\mathbf{G}\left( S\right) \in \mathbf{Pro}\left( k\right)
\end{equation*}%
be the rudimentary pro-module corresponding to $\dprod\limits_{t\in
S}k^{\ast }$. Due to (\ref{Prop-Pro-modules-properties-Injective}), $\mathbf{%
G}\left( S\right) $ is an injective pro-module. Since $k^{\ast }$ is an
injective cogenerator for $\mathbf{Mod}\left( k\right) $, we can for each $%
\left( t:j\rightarrow i\right) \in S$, choose a homomorphism%
\begin{equation*}
\varphi _{t}:K_{i}\longrightarrow k^{\ast },
\end{equation*}%
such that the composition%
\begin{equation*}
\varphi _{t}\circ K\left( t\right)
\end{equation*}%
is nonzero. Let%
\begin{equation*}
\varphi =\left( \dprod\limits_{t\in S}\varphi _{t}\right)
:K_{i}\longrightarrow \dprod\limits_{t\in S}k^{\ast }.
\end{equation*}%
The corresponding morphism%
\begin{equation*}
\Phi :\mathbf{K}\longrightarrow \mathbf{G}\left( S\right)
\end{equation*}%
is \textbf{nonzero}. Indeed, if it is zero, then there exists a $%
t:j\rightarrow i$ with%
\begin{equation*}
\Phi \circ K\left( t\right) =0.
\end{equation*}%
However,%
\begin{equation*}
\pi _{t}\circ \Phi \circ K\left( t\right) =\varphi _{t}\circ K\left(
t\right) \neq 0,
\end{equation*}%
where%
\begin{equation*}
\pi _{t}:\dprod\limits_{t\in S}k^{\ast }\longrightarrow k^{\ast }
\end{equation*}%
is the $t$-th projection. Denote by the same letter $\varphi _{i}$ the
corresponding morphism%
\begin{equation*}
\left( \varphi _{i}:K\longrightarrow k^{\ast }\right) \in \mathbf{Pro}\left(
k\right) .
\end{equation*}%
The morphism $\Phi $ can be extended, due to injectivity of $\mathbf{G}%
\left( S\right) $, to a morphism%
\begin{equation*}
\Psi :M\longrightarrow \mathbf{G}\left( S\right) .
\end{equation*}%
Since the composition $\Phi =\Psi \circ h$ is nonzero, the morphism $\Psi $ 
\textbf{cannot} be factored through $N$.
\end{proof}

\subsection{Derived categories}

We use here the \textquotedblleft classical\textquotedblright\ definition of
an $F$-projective category. The subcategories, which are called
\textquotedblleft $F$-projective\textquotedblright\ in \cite[Definition
13.3.4]{Kashiwara-Categories-MR2182076}, will be called \textbf{weak }$F$%
\textbf{-projective} in this paper.

\begin{definition}
\label{Def-F-projective}Let 
\begin{equation*}
F:\mathbf{C}\longrightarrow \mathbf{E}
\end{equation*}%
be a right exact additive functor of abelian categories, and let $\mathbf{P}$
be a full additive subcategory of $\mathbf{C}$. Then:

$\mathbf{P}$ is called \textbf{weak }$F$\textbf{-projective} if $\mathbf{P}$
satisfies the definition of an $F$-projective subcategory in \cite[%
Definition 13.3.4]{Kashiwara-Categories-MR2182076}.

$\mathbf{P}$ is called $F$-\textbf{projective} if it satisfies the following
conditions:

\begin{enumerate}
\item \label{Def-F-projective-generating}The category $\mathbf{P}$ is 
\textbf{generating} in $\mathbf{C}$ (i.e. for any object $X\in \mathbf{C}$
there exists an epimorphism $P\twoheadrightarrow X$ with $P\in \mathbf{P}$);

\item \label{Def-F-projective-closed-by-extensions}For any exact sequence%
\begin{equation*}
0\longrightarrow X^{\prime }\longrightarrow X\longrightarrow X^{\prime
\prime }\longrightarrow 0
\end{equation*}%
in $\mathbf{C}$ with $X$, $X^{\prime \prime }\in \mathbf{P}$, we have $%
X^{\prime }\in \mathbf{P}$;

\item \label{Def-F-projective-acyclic}For any exact sequence%
\begin{equation*}
0\longrightarrow X^{\prime }\longrightarrow X\longrightarrow X^{\prime
\prime }\longrightarrow 0
\end{equation*}%
in $\mathbf{C}$ with $X$, $X^{\prime \prime }\in \mathbf{P}$, the sequence%
\begin{equation*}
0\longrightarrow F\left( X^{\prime }\right) \longrightarrow F\left( X\right)
\longrightarrow F\left( X^{\prime \prime }\right) \longrightarrow 0
\end{equation*}%
is exact.
\end{enumerate}
\end{definition}

\begin{notation}
\label{Not-Derived-category}For an abelian category $\mathbf{E}$, let:

\begin{enumerate}
\item $C\left( \mathbf{E}\right) $ denote the category of \textbf{bounded
below} chain complexes in $\mathbf{E}$;

\item a \textbf{qis} denote a \textbf{quasi-isomorphism} in $C\left( \mathbf{%
E}\right) $, i.e. a homomorphism%
\begin{equation*}
X_{\bullet }\longrightarrow Y_{\bullet }
\end{equation*}%
inducing an isomorphism of the homologies;

\item a complex $X_{\bullet }$ be \textbf{qis} to $Y_{\bullet }$ iff there
is a qis $X_{\bullet }\rightarrow Y_{\bullet }$;

\item $K\left( \mathbf{E}\right) $ denote the homotopy category of $C\left( 
\mathbf{E}\right) $, i.e. morphisms%
\begin{equation*}
X_{\bullet }\longrightarrow Y_{\bullet }
\end{equation*}%
in $K\left( \mathbf{E}\right) $ are \textbf{classes} of homotopic maps $%
X_{\bullet }\rightarrow Y_{\bullet }$;

\item $D\left( \mathbf{E}\right) $ denote the corresponding derived category
of $K\left( \mathbf{E}\right) $, i.e.%
\begin{equation*}
D\left( \mathbf{E}\right) =K\left( \mathbf{E}\right) /N\left( \mathbf{E}%
\right)
\end{equation*}%
where $N\left( \mathbf{E}\right) $ is the full subcategory of $K\left( 
\mathbf{E}\right) $ consisting of complexes qis to $\mathbf{0}$.
\end{enumerate}
\end{notation}

\begin{proposition}
\label{Prop-Left-derived}Let $F:\mathbf{C}\rightarrow \mathbf{E}$ be an
additive functor of abelian categories, and let $\mathbf{P}$ be a full
additive subcategory of $\mathbf{C}$. Assume $\mathbf{P}$ is $F$-projective.
Then:

\begin{enumerate}
\item $\mathbf{P}$ is weak $F$-projective.

\item The left satellite%
\begin{equation*}
LF:D(\mathbf{C})\longrightarrow D(\mathbf{E})
\end{equation*}%
exists, and%
\begin{equation*}
LF(X_{\bullet })%
%TCIMACRO{\TeXButton{ISO}{\simeq}}%
%BeginExpansion
\simeq%
%EndExpansion
F(Y_{\bullet })
\end{equation*}%
for any qis%
\begin{equation*}
Y_{\bullet }\longrightarrow X_{\bullet }
\end{equation*}%
with $Y_{\bullet }\in K\left( \mathbf{P}\right) $.
\end{enumerate}
\end{proposition}

\begin{proof}
Follows from \cite[dual to Proposition 13.3.5 and Corollary 13.3.8]%
{Kashiwara-Categories-MR2182076}.
\end{proof}

Using $F$-projective subcategories, one can define left \textbf{satellites}
of the functor $F$.

\begin{definition}
\label{Def-Left-derived-functors}\label{Def-Left-satellites}In the
conditions of Proposition \ref{Prop-Left-derived} let $X\in \mathbf{C}$.
Considering $X$ as a complex concentrated in degree $0$, take a qis $%
P_{\bullet }\longrightarrow X$, i.e. a \textbf{resolution}%
\begin{equation*}
0\longleftarrow X\longleftarrow P_{0}\longleftarrow P_{1}\longleftarrow
P_{2}\longleftarrow ...\longleftarrow P_{n}\longleftarrow ...
\end{equation*}%
with $P_{\bullet }\in K\left( \mathbf{P}\right) $. Define%
\begin{equation*}
L_{n}F\left( X\right) 
%TCIMACRO{\TeXButton{assigned}{{:=}}}%
%BeginExpansion
{:=}%
%EndExpansion
H_{n}\left( P_{\bullet }\right) .
\end{equation*}%
It is easy to check that $L_{n}F$, $n\geq 0$, are additive functors%
\begin{equation*}
L_{n}F:\mathbf{C\longrightarrow E,}
\end{equation*}%
that $L_{n}F=\mathbf{0}$ if $n<0$, and that $L_{0}F%
%TCIMACRO{\TeXButton{ISO}{\simeq}}%
%BeginExpansion
\simeq%
%EndExpansion
F$ if $F$ is right exact.

The functors $L_{n}F$ are called the \textbf{left satellites} of $F$.
\end{definition}

\subsection{Bicomplexes}

In this section, $\mathbf{K}$ is assumed to be an abelian category. We
consider only \textbf{first quadrant chain} bicomplexes.

\begin{definition}
\label{Def-Bicomplex}A \textbf{bicomplex} in $\mathbf{K}$ is a collection%
\begin{equation*}
X_{\bullet ,\bullet }=\left( X_{s,t},d_{s,t},\delta _{s,t}\right) _{s,t\in 
\mathbb{Z}}
\end{equation*}%
of objects and morphisms%
\begin{eqnarray*}
X_{s,t} &\in &\mathbf{K,} \\
d_{s,t} &\in &Hom_{\mathbf{K}}\left( X_{s+1,t},X_{s,t}\right) , \\
\delta _{s,t} &\in &Hom_{\mathbf{K}}\left( X_{s,t+1},X_{s,t}\right) ,
\end{eqnarray*}%
such that for all $s,t\in \mathbb{Z}$%
\begin{eqnarray*}
X_{s,t} &=&0\text{ if }s<0\text{ or }t<0, \\
d_{s-1,t}\circ d_{s,t} &=&0, \\
\delta _{s,t-1}\circ \delta _{s,t} &=&0, \\
d_{s-1,t-1}\circ \delta _{s,t-1} &=&\delta _{s-1,t-1}\circ d_{s-1,t}.
\end{eqnarray*}
\end{definition}

\begin{definition}
\label{Def-total-complex}If $\left( X_{\bullet ,\bullet },d,\delta \right) $
be a bicomplex, let $Tot_{\bullet }\left( X\right) $ be the following chain
complex:%
\begin{equation*}
Tot_{n}\left( X\right)
=\dbigoplus\limits_{s+t=n}X_{s,t}=\dbigoplus\limits_{s+t=n}X_{s,t}%
%TCIMACRO{\TeXButton{ISO}{\simeq}}%
%BeginExpansion
\simeq%
%EndExpansion
\dprod\limits_{s+t=n}X_{s,t}
\end{equation*}%
with the differential%
\begin{equation*}
\partial _{n}:Tot_{n+1}\left( X\right) \longrightarrow Tot_{n}\left(
X\right) ,
\end{equation*}%
given by%
\begin{equation*}
\partial _{n}\circ \iota _{s,t}=\iota _{s-1,t}\circ d+\left( -1\right)
^{s}\iota _{s,t-1}\circ \delta ,
\end{equation*}%
where%
\begin{equation*}
\iota _{s,t}:X_{s,t}\rightarrowtail Tot_{n}\left( X\right)
\end{equation*}%
is the natural embedding into the coproduct.
\end{definition}

\begin{theorem}
\label{Th-Spectral-sequence}Let $\left( X_{\bullet ,\bullet },d,\delta
\right) $ be a first quadrant bicomplex in $\mathbf{K}$. All objects below
depend \textbf{functorially} on $X_{\bullet ,\bullet }$, and all morphisms
are \textbf{natural} in $X_{\bullet ,\bullet }$.

\begin{enumerate}
\item \label{Th-Spectral-sequence-exact-couple}There exist two families ($%
r\geq 1$) of (\textbf{vertical} and \textbf{horizontal}) bigraded derived
exact couples, and two corresponding spectral sequences ($i^{r}$, $j^{r}$,
and $k^{r}$ have bidegrees indicated on the corresponding diagrams):

\begin{enumerate}
\item 
\begin{equation*}
%TCIMACRO{%
%\TeXButton{Vertical-exact-couple}{\begin{diagram}[size=3.0em,textflow]
%{^{ver}D^{r}} &    & \rTo^{i^{r}}_{\left( 1,-1\right) }  &       & ^{ver}D^{r} \\
%  & \luTo^{k^{r}}_{\left( -r,r-1\right)}  &     & \ldTo^{j^{r}}_{\left( 0,0\right)}  \\
%  &        & ^{ver}E^{r} \\
%\end{diagram}}}%
%BeginExpansion
\begin{diagram}[size=3.0em,textflow]
{^{ver}D^{r}} &    & \rTo^{i^{r}}_{\left( 1,-1\right) }  &       & ^{ver}D^{r} \\
  & \luTo^{k^{r}}_{\left( -r,r-1\right)}  &     & \ldTo^{j^{r}}_{\left( 0,0\right)}  \\
  &        & ^{ver}E^{r} \\
\end{diagram}%
%EndExpansion
\end{equation*}%
where%
\begin{eqnarray*}
^{ver}D_{s,t}^{r} &\neq &0\text{ only if }s,s+t\geq 0, \\
^{ver}E_{s,t}^{r} &\neq &0\text{ only if }s,t\geq 0, \\
^{ver}E_{s,t}^{r} &=&\left( ^{ver}E_{s,t}^{r},~^{ver}d^{r}=j^{r}\circ
k^{r}:~^{ver}E_{s,t}^{r}\longrightarrow ~^{ver}E_{s-r,t+r-1}^{r}\right) , \\
&&^{ver}E_{s,t}^{r+1}%
%TCIMACRO{\TeXButton{ISO}{\simeq}}%
%BeginExpansion
\simeq%
%EndExpansion
H\left( ^{ver}E_{s,t}^{r},~^{ver}d^{r}\right) .
\end{eqnarray*}

\item 
\begin{equation*}
%TCIMACRO{%
%\TeXButton{Horisontal-exact-couple}{\begin{diagram}[size=3.0em,textflow]
%{^{hor}D^{r}} &    & \rTo^{i^{r}}_{\left( -1,1\right) }  &       & ^{hor}D^{r} \\
%  & \luTo^{k^{r}}_{\left( r-1,-r\right)}  &     & \ldTo^{j^{r}}_{\left( 0,0\right)}  \\
%  &        & ^{hor}E^{r} \\
%\end{diagram}}}%
%BeginExpansion
\begin{diagram}[size=3.0em,textflow]
{^{hor}D^{r}} &    & \rTo^{i^{r}}_{\left( -1,1\right) }  &       & ^{hor}D^{r} \\
  & \luTo^{k^{r}}_{\left( r-1,-r\right)}  &     & \ldTo^{j^{r}}_{\left( 0,0\right)}  \\
  &        & ^{hor}E^{r} \\
\end{diagram}%
%EndExpansion
\end{equation*}%
where%
\begin{eqnarray*}
^{hor}D_{s,t}^{r} &\neq &0\text{ only if }t,s+t\geq 0, \\
^{hor}E_{s,t}^{r} &\neq &0\text{ only if }s,t\geq 0, \\
^{hor}E_{s,t}^{r} &=&\left( ^{hor}E_{s,t}^{r},~^{hor}d^{r}=j^{r}\circ
k^{r}:~^{hor}E_{s,t}^{r}\longrightarrow ~^{hor}E_{s+r-1,t-r}^{r}\right) , \\
&&^{hor}E_{s,t}^{r+1}%
%TCIMACRO{\TeXButton{ISO}{\simeq}}%
%BeginExpansion
\simeq%
%EndExpansion
H\left( ^{hor}E_{s,t}^{r},~^{hor}d^{r}\right) .
\end{eqnarray*}
\end{enumerate}

\item \label{Th-Spectral-sequence-E-0}We introduce an extra entry $E^{0}$:

\begin{enumerate}
\item 
\begin{equation*}
^{ver}E_{\bullet ,\bullet }^{0}%
%TCIMACRO{\TeXButton{assigned}{{:=}}}%
%BeginExpansion
{:=}%
%EndExpansion
\left( X_{\bullet ,\bullet },d^{0}=\delta \right) .
\end{equation*}

\item 
\begin{equation*}
^{hor}E_{\bullet ,\bullet }^{0}%
%TCIMACRO{\TeXButton{assigned}{{:=}}}%
%BeginExpansion
{:=}%
%EndExpansion
\left( X_{\bullet ,\bullet },d^{0}=d\right) .
\end{equation*}
\end{enumerate}

\item \label{Th-Spectral-sequence-E-1}~

\begin{enumerate}
\item 
\begin{equation*}
^{ver}E_{\bullet ,t}^{1}%
%TCIMACRO{\TeXButton{ISO}{\simeq}}%
%BeginExpansion
\simeq%
%EndExpansion
\left( ^{ver}H_{t}\left( X_{\bullet ,\bullet }\right)
,d^{1}=d|_{^{ver}H\left( X_{\bullet ,\bullet }\right) }\right) .
\end{equation*}

\item 
\begin{equation*}
^{hor}E_{s,\bullet }^{1}%
%TCIMACRO{\TeXButton{ISO}{\simeq}}%
%BeginExpansion
\simeq%
%EndExpansion
\left( ^{hor}H_{s}\left( X_{\bullet ,\bullet }\right) ,d^{1}=\delta
|_{^{hor}H\left( X_{\bullet ,\bullet }\right) }\right) .
\end{equation*}
\end{enumerate}

\item \label{Th-Spectral-sequence-E-2}~

\begin{enumerate}
\item 
\begin{equation*}
^{ver}E_{s,t}^{2}%
%TCIMACRO{\TeXButton{ISO}{\simeq}}%
%BeginExpansion
\simeq%
%EndExpansion
~^{hor}H_{s}\left( ^{ver}H_{t}\left( X_{\bullet ,\bullet }\right) \right) .
\end{equation*}

\item 
\begin{equation*}
^{hor}E_{s,t}^{2}%
%TCIMACRO{\TeXButton{ISO}{\simeq}}%
%BeginExpansion
\simeq%
%EndExpansion
~^{ver}H_{t}\left( ^{hor}H_{s}\left( X_{\bullet ,\bullet }\right) \right) .
\end{equation*}
\end{enumerate}

\item \label{Th-Spectral-sequence-D-stabilizes}~

\begin{enumerate}
\item For each pair $\left( s,t\right) $ the sequence $^{ver}D^{r}$
stabilizes:%
\begin{equation*}
^{ver}D_{s,t}^{r}\longrightarrow ~^{ver}D_{s,t}^{r+1}%
%TCIMACRO{\TeXButton{assigned back}{{=:}}}%
%BeginExpansion
{=:}%
%EndExpansion
~^{ver}D_{s,t}^{\infty }
\end{equation*}%
is an isomorphism whenever $r\gg 0$.

\item For each pair $\left( s,t\right) $ the sequence $^{hor}D^{r}$
stabilizes:%
\begin{equation*}
^{hor}D_{s,t}^{r}\longrightarrow ~^{hor}D_{s,t}^{r+1}%
%TCIMACRO{\TeXButton{assigned back}{{=:}}}%
%BeginExpansion
{=:}%
%EndExpansion
~^{hor}D_{s,t}^{\infty }
\end{equation*}%
is an isomorphism whenever $r\gg 0$.
\end{enumerate}

\item \label{Th-Spectral-sequence-E-stabilizes}~

\begin{enumerate}
\item For each pair $\left( s,t\right) $ the sequence $^{ver}E^{r}$
stabilizes:%
\begin{equation*}
^{ver}E_{s,t}^{r}\longrightarrow ~^{ver}E_{s,t}^{r+1}%
%TCIMACRO{\TeXButton{assigned back}{{=:}}}%
%BeginExpansion
{=:}%
%EndExpansion
~^{ver}E_{s,t}^{\infty }
\end{equation*}%
is an isomorphism whenever $r\gg 0$.

\item For each pair $\left( s,t\right) $ the sequence $^{hor}E^{r}$
stabilizes:%
\begin{equation*}
^{hor}E_{s,t}^{r}\longrightarrow ~^{hor}E_{s,t}^{r+1}%
%TCIMACRO{\TeXButton{assigned back}{{=:}}}%
%BeginExpansion
{=:}%
%EndExpansion
~^{hor}E_{s,t}^{\infty }
\end{equation*}%
is an isomorphism whenever $r\gg 0$.
\end{enumerate}

\item \label{Th-Spectral-sequence-convergence}The two spectral sequences
converge to $H_{\bullet }\left( Tot_{\bullet }\left( X\right) \right) $ in
the following sense:

\begin{enumerate}
\item For each $n\geq 0$, the sequence below consists of monomorphisms%
\begin{equation*}
\left[ 0=~^{ver}D_{-1,n+1}^{\infty }\right] \rightarrowtail
~^{ver}D_{0,n}^{\infty }\rightarrowtail ~^{ver}D_{1,n-1}^{\infty
}\rightarrowtail ...\rightarrowtail ~^{ver}D_{n,0}^{\infty }%
%TCIMACRO{\TeXButton{ISO}{\simeq}}%
%BeginExpansion
\simeq%
%EndExpansion
H_{n}\left( Tot_{\bullet }\left( X\right) \right) ,
\end{equation*}%
and for each $s$, $t$%
\begin{equation*}
%TCIMACRO{\TeXButton{coker}{\coker}}%
%BeginExpansion
\coker%
%EndExpansion
\left( ^{ver}D_{s-1,t+1}^{\infty }\rightarrowtail ~^{ver}D_{s,t}^{\infty
}\right) 
%TCIMACRO{\TeXButton{ISO}{\simeq}}%
%BeginExpansion
\simeq%
%EndExpansion
~^{ver}E_{s,t}^{\infty }.
\end{equation*}

\item For each $n\geq 0$, the sequence below consists of monomorphisms%
\begin{equation*}
\left[ 0=~^{hor}D_{n+1,-1}^{\infty }\right] \rightarrowtail
~^{hor}D_{n,0}^{\infty }\rightarrowtail ~^{hor}D_{n-1,1}^{\infty
}\rightarrowtail ...\rightarrowtail ~^{hor}D_{0,n}^{\infty }%
%TCIMACRO{\TeXButton{ISO}{\simeq}}%
%BeginExpansion
\simeq%
%EndExpansion
H_{n}\left( Tot_{\bullet }\left( X\right) \right) ,
\end{equation*}%
and for each $s$, $t$%
\begin{equation*}
%TCIMACRO{\TeXButton{coker}{\coker}}%
%BeginExpansion
\coker%
%EndExpansion
\left( ^{hor}D_{s+1,t-1}^{\infty }\rightarrowtail ~^{hor}D_{s,t}^{\infty
}\right) 
%TCIMACRO{\TeXButton{ISO}{\simeq}}%
%BeginExpansion
\simeq%
%EndExpansion
~^{hor}E_{s,t}^{\infty }.
\end{equation*}
\end{enumerate}

\item \label{Th-Spectral-sequence-isomorphism}Let $f_{\bullet ,\bullet
}:X_{\bullet ,\bullet }\rightarrow Y_{\bullet ,\bullet }$ be a morphism of
bicomplexes, and let $r\geq 1$.

\begin{enumerate}
\item If for some $r$%
\begin{equation*}
^{ver}E_{s,t}^{r}\left( f\right) :~^{ver}E_{s,t}^{r}\left( X\right)
\longrightarrow ~^{ver}E_{s,t}^{r}\left( Y\right)
\end{equation*}%
is an isomorphism for all $s$, $t$, then%
\begin{equation*}
H_{n}\left( Tot_{\bullet }\left( f\right) \right) :H_{n}\left( Tot_{\bullet
}\left( X\right) \right) \longrightarrow H_{n}\left( Tot_{\bullet }\left(
Y\right) \right)
\end{equation*}%
is an isomorphism for all $n$.

\item If for some $r$%
\begin{equation*}
^{hor}E_{s,t}^{r}\left( f\right) :~^{hor}E_{s,t}^{r}\left( X\right)
\longrightarrow ~^{hor}E_{s,t}^{r}\left( Y\right)
\end{equation*}%
is an isomorphism for all $s$, $t$, then%
\begin{equation*}
H_{n}\left( Tot_{\bullet }\left( f\right) \right) :H_{n}\left( Tot_{\bullet
}\left( X\right) \right) \longrightarrow H_{n}\left( Tot_{\bullet }\left(
Y\right) \right)
\end{equation*}%
is an isomorphism for all $n$.
\end{enumerate}

\item \label{Th-Spectral-sequence-edge-homomorphism}~

\begin{enumerate}
\item For all $r$, $1\leq r\leq \infty $, and all $n$,%
\begin{equation*}
^{ver}D_{0,n}^{r}%
%TCIMACRO{\TeXButton{ISO}{\simeq}}%
%BeginExpansion
\simeq%
%EndExpansion
~^{ver}E_{0,n}^{r}.
\end{equation*}%
The composition%
\begin{equation*}
^{ver}H_{n}\left( X_{0,\bullet }\right)
=~^{ver}E_{0,n}^{1}\twoheadrightarrow ~^{ver}E_{0,n}^{\infty }%
%TCIMACRO{\TeXButton{ISO}{\simeq}}%
%BeginExpansion
\simeq%
%EndExpansion
~^{ver}D_{0,n}^{\infty }\rightarrowtail H_{n}\left( Tot_{\bullet }\left(
X\right) \right)
\end{equation*}%
is induced (up to sign) by the embedding of complexes $X_{0,\bullet
}\hookrightarrow Tot_{\bullet }\left( X\right) $.

Let $\varphi _{n}$ be the composition%
\begin{equation*}
H_{n}\left( Tot_{\bullet }\left( X\right) \right) 
%TCIMACRO{\TeXButton{ISO}{\simeq}}%
%BeginExpansion
\simeq%
%EndExpansion
~^{ver}D_{n,0}^{\infty }\twoheadrightarrow ~^{ver}E_{n,0}^{\infty
}\rightarrowtail ~^{ver}E_{n,0}^{2}.
\end{equation*}%
Then the following diagram commutes (up to sign):%
\begin{equation*}
%TCIMACRO{%
%\TeXButton{Vertical-edge-homomorphism}{\begin{diagram}[w=1.0em,h=1.5em,textflow]
%{Tot_{n+1}\left( X\right) } & \rTo^{\partial } & {Tot_{n}\left( X\right) } & \rTo & {{\coker}\ \partial } & \lInto & {H_{n}\left( Tot_{\bullet }\left( X\right) \right) } \\
%\dTo & & \dTo \\
%{X_{n+1,0}} & & {X_{n,0}}  & & \dTo & & \dTo_{\varphi _{n}} \\
%\dTo & & \dTo \\
%{\left[ \TeXButton{coker}{\coker}\delta _{n+1,0}=~^{ver}E_{n+1,0}^{1}\right] } & \rTo^{d|_{E_{n+1,0}^{1}}} & {\left[ {\coker}\ \delta _{n,0}=~^{ver}E_{n,0}^{1}\right] }  & \rTo & {\TeXButton{coker}{\coker}d|_{E_{n+1,0}^{1}}} & \lInto & {^{ver}E^{2}_{n,0}} \\
%\end{diagram}}}%
%BeginExpansion
\begin{diagram}[w=1.0em,h=1.5em,textflow]
{Tot_{n+1}\left( X\right) } & \rTo^{\partial } & {Tot_{n}\left( X\right) } & \rTo & {{\coker}\ \partial } & \lInto & {H_{n}\left( Tot_{\bullet }\left( X\right) \right) } \\
\dTo & & \dTo \\
{X_{n+1,0}} & & {X_{n,0}}  & & \dTo & & \dTo_{\varphi _{n}} \\
\dTo & & \dTo \\
{\left[ \TeXButton{coker}{\coker}\delta _{n+1,0}=~^{ver}E_{n+1,0}^{1}\right] } & \rTo^{d|_{E_{n+1,0}^{1}}} & {\left[ {\coker}\ \delta _{n,0}=~^{ver}E_{n,0}^{1}\right] }  & \rTo & {\TeXButton{coker}{\coker}d|_{E_{n+1,0}^{1}}} & \lInto & {^{ver}E^{2}_{n,0}} \\
\end{diagram}%
%EndExpansion
\end{equation*}

\item For all $r$, $1\leq r\leq \infty $, and all $n$,%
\begin{equation*}
^{hor}D_{n,0}^{r}%
%TCIMACRO{\TeXButton{ISO}{\simeq}}%
%BeginExpansion
\simeq%
%EndExpansion
~^{hor}E_{n,0}^{r}.
\end{equation*}%
The composition%
\begin{equation*}
^{hor}H_{n}\left( X_{\bullet ,0}\right)
=~^{hor}E_{n,0}^{1}\twoheadrightarrow ~^{hor}E_{n,0}^{\infty }%
%TCIMACRO{\TeXButton{ISO}{\simeq}}%
%BeginExpansion
\simeq%
%EndExpansion
~^{hor}D_{n,0}^{\infty }\rightarrowtail H_{n}\left( Tot_{\bullet }\left(
X\right) \right)
\end{equation*}%
is induced (up to sign) by the inclusion of complexes $X_{\bullet
,0}\hookrightarrow Tot_{\bullet }\left( X\right) $.

Let $\psi _{n}$ be the composition%
\begin{equation*}
H_{n}\left( Tot_{\bullet }\left( X\right) \right) 
%TCIMACRO{\TeXButton{ISO}{\simeq}}%
%BeginExpansion
\simeq%
%EndExpansion
~^{hor}D_{0,n}^{\infty }\twoheadrightarrow ~^{hor}E_{0,n}^{\infty
}\rightarrowtail ~^{hor}E_{0,n}^{2}.
\end{equation*}%
Then the following diagram commutes (up to sign):%
\begin{equation*}
%TCIMACRO{%
%\TeXButton{Horisontal-edge-homomorphism}{\begin{diagram}[w=1.0em,h=1.5em,textflow]
%{Tot_{n+1}\left( X\right) } & \rTo^{\partial } & {Tot_{n}\left( X\right) } & \rTo & {{\coker}\ \partial } & \lInto & {H_{n}\left( Tot_{\bullet }\left( X\right) \right) } \\
%\dTo & & \dTo \\
%{X_{0,n+1}} & & {X_{0,n}}  & & \dTo & & \dTo_{\psi _{n}} \\
%\dTo & & \dTo \\
%{\left[ {\coker}\ d_{0,n+1}=~^{hor}E_{0,n+1}^{1}\right] } & \rTo^{\delta |_{E_{0,n+1}^{1}}} & {\left[ \TeXButton{coker}{\coker}d_{0,n}=~^{hor}E_{0,n}^{1}\right] } & \rTo & {\TeXButton{coker}{\coker}d|_{E_{0,n+1}^{1}}} & \lInto & {^{hor}E^{2}_{0,n}} \\
%\end{diagram}}}%
%BeginExpansion
\begin{diagram}[w=1.0em,h=1.5em,textflow]
{Tot_{n+1}\left( X\right) } & \rTo^{\partial } & {Tot_{n}\left( X\right) } & \rTo & {{\coker}\ \partial } & \lInto & {H_{n}\left( Tot_{\bullet }\left( X\right) \right) } \\
\dTo & & \dTo \\
{X_{0,n+1}} & & {X_{0,n}}  & & \dTo & & \dTo_{\psi _{n}} \\
\dTo & & \dTo \\
{\left[ {\coker}\ d_{0,n+1}=~^{hor}E_{0,n+1}^{1}\right] } & \rTo^{\delta |_{E_{0,n+1}^{1}}} & {\left[ \TeXButton{coker}{\coker}d_{0,n}=~^{hor}E_{0,n}^{1}\right] } & \rTo & {\TeXButton{coker}{\coker}d|_{E_{0,n+1}^{1}}} & \lInto & {^{hor}E^{2}_{0,n}} \\
\end{diagram}%
%EndExpansion
\end{equation*}
\end{enumerate}
\end{enumerate}
\end{theorem}

\begin{proof}
The proof of various forms of this theorem is scattered around several
papers and books. See \cite%
{Eckmann-Hilton-1966-Exact-couples-in-an-abelian-category-MR0191937}, \cite[%
Chapter 5]{Weibel-1994-An-introduction-to-homological-algebra-MR1269324}, 
\cite[\S III.7]{Manin-2003-Methods-of-homological-algebra-MR1950475}, and 
\cite[Theorem 12.5.4 and Corollary 12.5.5(3)]{Kashiwara-Categories-MR2182076}%
.
\end{proof}

\section{Topologies}

\subsection{Grothendieck topologies}

\begin{definition}
\label{Def-Sieve}Let $\mathbf{C}$ be a category. A \textbf{sieve} $R$ over $%
U\in \mathbf{C}$ is a subfunctor $R\subseteq h_{U}$ of 
\begin{equation*}
h_{U}=Hom_{\mathbf{C}}\left( \bullet ,U\right) :\mathbf{C}%
^{op}\longrightarrow \mathbf{Set.}
\end{equation*}
\end{definition}

\begin{remark}
Compare with \cite[Definition 16.1.1]{Kashiwara-Categories-MR2182076}.
\end{remark}

\begin{definition}
\label{Def-Site}A Grothendieck site (or simply a \textbf{site}) $X$ is a
pair $\left( \mathbf{C}_{X},Cov\left( X\right) \right) $ where $\mathbf{C}%
_{X}$ is a category, and%
\begin{equation*}
Cov\left( X\right) =\dbigcup\limits_{U\in \mathbf{C}_{X}}Cov\left( U\right) ,
\end{equation*}%
where $Cov\left( U\right) $ are the sets of \textbf{covering sieves} over $U$%
, satisfying the axioms GT1-GT4 from \cite[Definition 16.1.2]%
{Kashiwara-Categories-MR2182076}, or, equivalently, the axioms T1-T3 from 
\cite[Definition II.1.1]{SGA4-1-MR0354652}. The site is called \textbf{small}
iff $\mathbf{C}_{X}$ is a small category.
\end{definition}

\begin{remark}
The class (or a set, if $X$ is small) $Cov\left( X\right) $ is called the 
\textbf{topology} on $X$.
\end{remark}

\begin{notation}
\label{Note-Comma-over-site}Given $U\in \mathbf{C}_{X}$, and $R\in Cov\left(
X\right) $, denote simply%
\begin{equation*}
\mathbf{C}_{U}%
%TCIMACRO{\TeXButton{assigned}{{:=}}}%
%BeginExpansion
{:=}%
%EndExpansion
\left( \mathbf{C}_{X}\right) _{U},~\mathbf{C}_{R}%
%TCIMACRO{\TeXButton{assigned}{{:=}}}%
%BeginExpansion
{:=}%
%EndExpansion
\left( \mathbf{C}_{X}\right) _{R},
\end{equation*}%
where $\left( \mathbf{C}_{X}\right) _{U}$ and $\left( \mathbf{C}_{X}\right)
_{R}$ are the comma-categories defined earlier in Definition \ref%
{Def-Comma-U} and Definition \ref{Def-Comma-R}.
\end{notation}

\begin{definition}
\label{Def-Pretopology}We say that the topology on a small site $X$ is
induced by a \textbf{pretopology} if each object $U\in \mathbf{C}_{X}$ is
supplied with base-changeable (Definition \ref{Def-Quarrable}) \textbf{covers%
} $\left\{ U_{i}\rightarrow U\right\} _{i\in I}$, satisfying \cite[%
Definition II.1.3]{SGA4-1-MR0354652} (compare to \cite[Definition 16.1.5]%
{Kashiwara-Categories-MR2182076}), and the covering sieves $R\in Cov\left(
X\right) $ are \textbf{generated} by covers:%
\begin{equation*}
R=R_{\left\{ U_{i}\rightarrow U\right\} }\subseteq h_{U},
\end{equation*}%
where $R_{\left\{ U_{i}\rightarrow U\right\} }\left( V\right) $ consists of
morphisms $\left( V\rightarrow U\right) \in h_{U}\left( V\right) $ admitting
a decomposition%
\begin{equation*}
\left( V\rightarrow U\right) =\left( V\rightarrow U_{i}\rightarrow U\right) .
\end{equation*}
\end{definition}

\begin{remark}
We use the word \textbf{covers} for general sites, and reserve the word 
\textbf{coverings} for open coverings of topological spaces.
\end{remark}

\begin{proposition}
\label{Prop-A-ten-Set-CX-R}Let $G\in \mathbf{Mod}\left( k\right) $, let $%
\mathcal{A}\in \mathbf{pCS}\left( X,\mathbf{Pro}\left( k\right) \right) $,
and let $R\subseteq h_{U}$ be a sieve. Then:

\begin{enumerate}
\item \label{Prop-A-ten-Set-CX-R-Hom-to-G}\label{Prop-Hom-K-A-ten-Set-CX-G}%
\begin{eqnarray*}
&&Hom_{\mathbf{Pro}\left( k\right) }\left( \mathcal{A}\otimes _{\mathbf{Set}%
^{\mathbf{C}_{X}}}R,G\right) 
%TCIMACRO{\TeXButton{ISO}{\simeq}}%
%BeginExpansion
\simeq%
%EndExpansion
Hom_{\mathbf{Set}^{\left( \mathbf{C}_{X}\right) ^{op}}}\left( R,Hom_{\mathbf{%
K}}\left( \mathcal{A},G\right) \right) 
%TCIMACRO{\TeXButton{ISO}{\simeq} }%
%BeginExpansion
\simeq
%EndExpansion
\\
&&%
%TCIMACRO{\TeXButton{ISO}{\simeq}}%
%BeginExpansion
\simeq%
%EndExpansion
\underset{\left( V\rightarrow U\right) \in \mathbf{C}_{R}}{\underleftarrow{%
\lim }}Hom_{\mathbf{K}}\left( \mathcal{A}\left( V\right) ,G\right) 
%TCIMACRO{\TeXButton{ISO}{\simeq}}%
%BeginExpansion
\simeq%
%EndExpansion
Hom_{\mathbf{K}}\left( \underset{\left( V\rightarrow U\right) \in \mathbf{C}%
_{R}}{\underrightarrow{\lim }}\mathcal{A}\left( V\right) ,G\right)
\end{eqnarray*}%
naturally in $G$, $\mathcal{A}$ and $R$. The presheaf of $k$-\textbf{modules}
$Hom_{\mathbf{Pro}\left( k\right) }\left( \mathcal{A},G\right) $ is
introduced in Definition \ref{Def-Pairings-functors}(\ref%
{Def-Pairings-functors-Hom(Pro-k)}).

\item \label{Prop-A-ten-Set-CX-R-lim-CR}%
\begin{equation*}
\mathcal{A}\otimes _{\mathbf{Set}^{\mathbf{C}_{X}}}R%
%TCIMACRO{\TeXButton{ISO}{\simeq}}%
%BeginExpansion
\simeq%
%EndExpansion
\underset{\left( V\rightarrow U\right) \in \mathbf{C}_{R}}{\underrightarrow{%
\lim }}\mathcal{A}\left( V\right) .
\end{equation*}
\end{enumerate}
\end{proposition}

\begin{proof}
See \cite[Proposition 2.3]{Prasolov-Cosheafification-2016-zbMATH06684178}
\end{proof}

\begin{example}
\label{Site-TOP}Let $X$ be a topological space. We will call the site $%
OPEN\left( X\right) $ below the \textbf{standard site} for $X$:%
\begin{equation*}
OPEN\left( X\right) =\left( \mathbf{C}_{OPEN\left( X\right) },Cov\left(
OPEN\left( X\right) \right) \right) .
\end{equation*}%
$\mathbf{C}_{OPEN\left( X\right) }$ has open subsets of $X$ as objects and
inclusions $U\subseteq V$ as morphisms. The pretopology on $OPEN\left(
X\right) $ consists of open coverings%
\begin{equation*}
\left\{ U_{i}\subseteq U\right\} _{i\in I}\in \mathbf{C}_{OPEN\left(
X\right) }.
\end{equation*}%
The corresponding topology consists of sieves $R_{\left\{ U_{i}\subseteq
U\right\} }\subseteq h_{U}$ where%
\begin{equation*}
\left( V\subseteq U\right) \in R_{\left\{ U_{i}\subseteq U\right\} }\left(
U\right) \iff \exists i\in I~\left( V\subseteq U_{i}\right) .
\end{equation*}
\end{example}

\begin{remark}
\label{Denote-standard-site-simply}We will always denote the standard site $%
OPEN\left( X\right) $ simply by $X$.
\end{remark}

\begin{example}
\label{Site-NORM}Let again $X$ be a topological space. Consider the site%
\begin{equation*}
NORM\left( X\right) =\left( \mathbf{C}_{NORM\left( X\right) },Cov\left(
NORM\left( X\right) \right) \right)
\end{equation*}%
where $\mathbf{C}_{NORM\left( X\right) }=\mathbf{C}_{X}$, while the
pretopology on $NORM\left( X\right) $ consists of \textbf{normal}
(Definition \ref{Def-Normal-covering}) coverings $\left\{ U_{i}\subseteq
U\right\} $.
\end{example}

See Conjecture \ref{Conj-Satellites-H}.

\begin{example}
\label{Ex-Site-ETALE}\label{Site-ETALE}Let $X$ be a noetherian scheme, and
define the site $X^{et}$ by: $\mathbf{C}_{X^{et}}$ is the category of
schemes $Y/X$ \'{e}tale, finite type, while the pretopology on $X^{et}$
consists of finite surjective families of maps. See \cite[Example 1.1.6]%
{Artin-GT}, or \cite[II.1.2]{Tamme-MR1317816}.
\end{example}

Let $X=\left( \mathbf{C}_{X}\mathbf{,}Cov\left( X\right) \right) $ be a
small site (Definition \ref{Def-Site}), and let $\mathbf{D}$ and $\mathbf{K}$
be categories. Assume that $\mathbf{D}$ is small, and $\mathbf{K}$ is
complete (Definition \ref{Def-(co)complete}).

\begin{definition}
\label{Def-(Pre)sheaves}~

\begin{enumerate}
\item A \textbf{presheaf} $\mathcal{A}$ on $\mathbf{D}$ with values in $%
\mathbf{K}$ is a functor $\mathcal{A}:\mathbf{D}^{op}\rightarrow \mathbf{K}$.

\item A presheaf $\mathcal{A}$ on $X$ with values in $\mathbf{K}$ is a
functor $\mathcal{A}:\left( \mathbf{C}_{X}\right) ^{op}\rightarrow \mathbf{K}
$.

\item A presheaf $\mathcal{A}$ on $X$ is \textbf{separated} provided 
\begin{equation*}
\mathcal{A}\left( U\right) 
%TCIMACRO{\TeXButton{ISO}{\simeq}}%
%BeginExpansion
\simeq%
%EndExpansion
Hom_{\mathbf{Set}^{\mathbf{C}_{X}}}\left( h_{U},\mathcal{A}\right)
\longrightarrow Hom_{\mathbf{Set}^{\mathbf{C}_{X}}}\left( R,\mathcal{A}%
\right) 
%TCIMACRO{\TeXButton{ISO}{\simeq}}%
%BeginExpansion
\simeq%
%EndExpansion
\underset{\left( V\rightarrow U\right) \in \mathbf{C}_{R}}{\underleftarrow{%
\lim }}\mathcal{A}\left( V\right)
\end{equation*}%
is a monomorphism for any $U\in \mathbf{C}_{X}$ and for any covering sieve
(Definition \ref{Def-Sieve} and \ref{Def-Site}) $R$ over $U$. The pairing $%
Hom_{\mathbf{Set}^{\mathbf{C}_{X}}}\left( \bullet ,\bullet \right) $ is
introduced in Definition \ref{Def-Pairings-functors}(\ref%
{Def-Pairings-functors-Hom-Set-X}).

\item A presheaf $\mathcal{A}$ on $X$ is a \textbf{sheaf} provided 
\begin{equation*}
\mathcal{A}\left( U\right) 
%TCIMACRO{\TeXButton{ISO}{\simeq}}%
%BeginExpansion
\simeq%
%EndExpansion
Hom_{\mathbf{Set}^{\mathbf{C}_{X}}}\left( h_{U},\mathcal{A}\right)
\longrightarrow Hom_{\mathbf{Set}^{\mathbf{C}_{X}}}\left( R,\mathcal{A}%
\right) 
%TCIMACRO{\TeXButton{ISO}{\simeq}}%
%BeginExpansion
\simeq%
%EndExpansion
\underset{\left( V\rightarrow U\right) \in \mathbf{C}_{R}}{\underleftarrow{%
\lim }}\mathcal{A}\left( V\right)
\end{equation*}%
is an isomorphism for any $U\in \mathbf{C}_{X}$ and for any covering sieve $%
R $ over $U$.
\end{enumerate}
\end{definition}

\begin{remark}
\label{Rem-Hom-Set-CX}The isomorphisms%
\begin{equation*}
Hom_{\mathbf{Set}^{\mathbf{C}_{X}}}\left( R,\mathcal{A}\right) 
%TCIMACRO{\TeXButton{ISO}{\simeq}}%
%BeginExpansion
\simeq%
%EndExpansion
\underset{\left( V\rightarrow U\right) \in \mathbf{C}_{R}}{\underleftarrow{%
\lim }}\mathcal{A}\left( V\right)
\end{equation*}%
and%
\begin{equation*}
\mathcal{A}\left( U\right) 
%TCIMACRO{\TeXButton{ISO}{\simeq}}%
%BeginExpansion
\simeq%
%EndExpansion
Hom_{\mathbf{Set}^{\mathbf{C}_{X}}}\left( h_{U},\mathcal{A}\right)
\end{equation*}%
follow from \cite[Proposition B.6]%
{Prasolov-Cosheafification-2016-zbMATH06684178}, because the comma-category $%
\mathbf{C}_{U}%
%TCIMACRO{\TeXButton{ISO}{\simeq}}%
%BeginExpansion
\simeq%
%EndExpansion
\mathbf{C}_{h_{U}}$ (Definition \ref{Def-Comma-U} and Remark \ref%
{Rem-CU-equivalent-ChU}) has a terminal object $\left( U,\mathbf{1}%
_{U}\right) $.
\end{remark}

\begin{notation}
\label{Not-(Pre)sheaves-categories}Denote by $\mathbf{S}\left( X,\mathbf{K}%
\right) $ the category of sheaves, and by $\mathbf{pS}\left( X,\mathbf{K}%
\right) $ ($\mathbf{pS}\left( \mathbf{D},\mathbf{K}\right) $) the category
of presheaves on $X$ (on $\mathbf{D}$) with values in $\mathbf{K}$.
\end{notation}

\begin{remark}
\label{Rem-Compare-sheaves-with-cosheaves}Compare to Definition \ref%
{Def-(Pre)cosheaves} and Notation \ref{Not-(Pre)cosheaves-categories}.
\end{remark}

\subsection{%
%TCIMACRO{\TeXButton{Cech }{\u{C}ech} }%
%BeginExpansion
\u{C}ech
%EndExpansion
(co)homology}

\begin{definition}
\label{Def-Quarrable}A morphism $V\rightarrow U$ in a category $\mathbf{D}$
is called \textbf{base-changeable} (\textquotedblleft
quarrable\textquotedblright\ in (\cite[Def. II.1.3]{SGA4-1-MR0354652}), iff
for every other morphism $U^{\prime }\rightarrow U$ the fiber product $V%
\underset{U}{\times }U^{\prime }$ exists.
\end{definition}

\begin{definition}
\label{Def-Cech-complex}Let $\mathbf{D}$ and $\mathbf{K}$ be categories.
Assume that $\mathbf{D}$ is small and $\mathbf{K}$ is abelian. Let $\left\{
U_{i}\rightarrow U\right\} $ be a family of base-changeable morphisms in $%
\mathbf{K}$. For a pre(co)sheaf%
\begin{equation*}
\mathcal{A}\in \mathbf{pCS}\left( \mathbf{D},\mathbf{K}\right) \text{
(respectively }\mathcal{B}\in \mathbf{pS}\left( \mathbf{D},\mathbf{K}\right) 
\text{)}
\end{equation*}%
on $\mathbf{D}$ with values in $\mathbf{K}$, define the following \textbf{%
%TCIMACRO{\TeXButton{Cech }{\u{C}ech} }%
%BeginExpansion
\u{C}ech
%EndExpansion
chain complex} $\check{C}_{\bullet }$ and the\textbf{\ 
%TCIMACRO{\TeXButton{Cech }{\u{C}ech} }%
%BeginExpansion
\u{C}ech
%EndExpansion
cochain complex} $\check{C}^{\bullet }$. Assume that $\mathbf{K}$ is
complete in the case of a presheaf, and cocomplete in the case of a
precosheaf: 
\begin{eqnarray*}
&&\check{C}^{\bullet }\left( \left\{ U_{i}\rightarrow U\right\} \mathbf{,~}%
\mathcal{B}\right) 
%TCIMACRO{\TeXButton{assigned}{{:=}}}%
%BeginExpansion
{:=}%
%EndExpansion
\left( \check{C}^{n}\left( \left\{ U_{i}\rightarrow U\right\} \mathbf{,~}%
\mathcal{B}\right) ,d^{n}\right) _{n\geq 0}, \\
&&\check{C}_{\bullet }\left( \left\{ U_{i}\rightarrow U\right\} \mathbf{,~}%
\mathcal{A}\right) 
%TCIMACRO{\TeXButton{assigned}{{:=}}}%
%BeginExpansion
{:=}%
%EndExpansion
\left( \check{C}_{n}\left( \left\{ U_{i}\rightarrow U\right\} \mathbf{,~}%
\mathcal{A}\right) ,d_{n}\right) _{n\geq 0},
\end{eqnarray*}%
where%
\begin{eqnarray*}
\check{C}^{n}\left( \left\{ U_{i}\rightarrow U\right\} \mathbf{,~}\mathcal{B}%
\right) &=&\dprod\limits_{i_{0},i_{1},...,i_{n}\in I}\mathcal{B}\left(
U_{i_{0}}\underset{U}{\times }U_{i_{1}}\underset{U}{\times }...\underset{U}{%
\times }U_{i_{n}}\right) , \\
\check{C}_{n}\left( \left\{ U_{i}\rightarrow U\right\} \mathbf{,~}\mathcal{A}%
\right) &=&\dbigoplus\limits_{i_{0},i_{1},...,i_{n}\in I}\mathcal{A}\left(
U_{i_{0}}\underset{U}{\times }U_{i_{1}}\underset{U}{\times }...\underset{U}{%
\times }U_{i_{n}}\right) , \\
d^{n} &=&\dsum\limits_{k=0}^{n+1}\left( -1\right) ^{k}d_{\left( k\right)
}^{n}, \\
d_{n} &=&\dsum\limits_{k=0}^{n+1}\left( -1\right) ^{k}d_{n}^{\left( k\right)
},
\end{eqnarray*}%
$d_{\left( k\right) }^{n}:\check{C}^{n}\rightarrow \check{C}^{n+1}$ are
defined by the compositions%
\begin{eqnarray*}
&&\left[ \pi _{i_{0},i_{1},...,i_{n},i_{n+1}}\right] \circ d_{\left(
k\right) }^{n}%
%TCIMACRO{\TeXButton{assigned}{{:=}}}%
%BeginExpansion
{:=}%
%EndExpansion
\left[ \dprod\limits_{i_{0},i_{1},...,i_{n}\in I}\mathcal{B}\left( U_{i_{0}}%
\underset{U}{\times }U_{i_{1}}\underset{U}{\times }...\underset{U}{\times }%
U_{i_{n}}\right) \right. \\
&&%
%TCIMACRO{%
%\TeXButton{Arrow pi NEW}{\begin{diagram}[size=3.0em,textflow]
% & \rTo^{\pi _{i_{0},...,\widehat{i_{k}},...,i_{n}}} & {\mathcal{B}\left( U_{i_{0}}\underset{U}{\times }...\underset{U}{\times }\widehat{U_{i_{k}}}\underset{U}{\times }...\underset{U}{\times }U_{i_{n}}\underset{U}{\times }U_{i_{n+1}}\right)} \\
%\end{diagram}} }%
%BeginExpansion
\begin{diagram}[size=3.0em,textflow]
 & \rTo^{\pi _{i_{0},...,\widehat{i_{k}},...,i_{n}}} & {\mathcal{B}\left( U_{i_{0}}\underset{U}{\times }...\underset{U}{\times }\widehat{U_{i_{k}}}\underset{U}{\times }...\underset{U}{\times }U_{i_{n}}\underset{U}{\times }U_{i_{n+1}}\right)} \\
\end{diagram}
%EndExpansion
\\
&&\left. 
%TCIMACRO{%
%\TeXButton{Arrow B NEW}{\begin{diagram}[size=3.0em,textflow]
% & \rTo^{\mathcal{B}\left( \sigma _{k,i_{0},i_{1},...,i_{n},i_{n+1}}\right)}  & {\mathcal{B}}\left( U_{i_{0}}\underset{U}{\times }U_{i_{1}}\underset{U}{\times }...\underset{U}{\times }U_{i_{n}}\underset{U}{\times }U_{i_{n+1}}\right) \\
%\end{diagram}}}%
%BeginExpansion
\begin{diagram}[size=3.0em,textflow]
 & \rTo^{\mathcal{B}\left( \sigma _{k,i_{0},i_{1},...,i_{n},i_{n+1}}\right)}  & {\mathcal{B}}\left( U_{i_{0}}\underset{U}{\times }U_{i_{1}}\underset{U}{\times }...\underset{U}{\times }U_{i_{n}}\underset{U}{\times }U_{i_{n+1}}\right) \\
\end{diagram}%
%EndExpansion
\right] ,
\end{eqnarray*}%
and%
\begin{eqnarray*}
\pi _{i_{0},i_{1},...,i_{n},i_{n+1}} &:&\left[ \dprod%
\limits_{i_{0},i_{1},...,i_{n+1}\in I}\mathcal{B}\left( U_{i_{0}}\underset{U}%
{\times }U_{i_{1}}\underset{U}{\times }...\underset{U}{\times }%
U_{i_{n+1}}\right) \right] \longrightarrow \mathcal{B}\left( U_{i_{0}}%
\underset{U}{\times }U_{i_{1}}\underset{U}{\times }...\underset{U}{\times }%
U_{i_{n+1}}\right) , \\
\pi _{i_{0},...,\widehat{i_{k}},...,i_{n}} &:&\left[ \dprod%
\limits_{i_{0},i_{1},...,i_{n}\in I}\mathcal{B}\left( U_{i_{0}}\underset{U}{%
\times }U_{i_{1}}\underset{U}{\times }...\underset{U}{\times }%
U_{i_{n}}\right) \right] \longrightarrow \mathcal{B}\left( U_{i_{0}}\underset%
{U}{\times }...\underset{U}{\times }\widehat{U_{i_{k}}}\underset{U}{\times }%
...\underset{U}{\times }U_{i_{n+1}}\right) , \\
\sigma _{k,i_{0},i_{1},...,i_{n},i_{n+1}} &:&U_{i_{0}}\underset{U}{\times }%
U_{i_{1}}\underset{U}{\times }...\underset{U}{\times }U_{i_{n}}\underset{U}{%
\times }U_{i_{n+1}}\longrightarrow U_{i_{0}}\underset{U}{\times }...\underset%
{U}{\times }\widehat{U_{i_{k}}}\underset{U}{\times }...\underset{U}{\times }%
U_{i_{n+1}},
\end{eqnarray*}%
are the natural projections.

$d_{n}^{\left( k\right) }:\check{C}_{n+1}\rightarrow \check{C}_{n}$ are
defined dually to $d_{\left( k\right) }^{n}$, by the compositions%
\begin{eqnarray*}
&&d_{n}^{\left( k\right) }\circ \left[ \rho _{i_{0},i_{1},...,i_{n+1}}\right]
%TCIMACRO{\TeXButton{assigned}{{:=}}}%
%BeginExpansion
{:=}%
%EndExpansion
\left[ \mathcal{A}\left( U_{i_{0}}\underset{U}{\times }U_{i_{1}}\underset{U}{%
\times }...\underset{U}{\times }U_{i_{n}}\underset{U}{\times }%
U_{i_{n+1}}\right) \right. \\
&&%
%TCIMACRO{%
%\TeXButton{Arrow A NEW}{\begin{diagram}[size=3.0em,textflow]
% & \rTo^{\mathcal{A}\left( \sigma _{k,i_{0},i_{1},...,i_{n},i_{n+1}}\right)}  & \mathcal{A}\left( U_{i_{0}}\underset{U}{\times }...\underset{U}{\times }\widehat{U_{i_{k}}}\underset{U}{\times }...\underset{U}{\times }U_{i_{n+1}}\right) \\
%\end{diagram}} }%
%BeginExpansion
\begin{diagram}[size=3.0em,textflow]
 & \rTo^{\mathcal{A}\left( \sigma _{k,i_{0},i_{1},...,i_{n},i_{n+1}}\right)}  & \mathcal{A}\left( U_{i_{0}}\underset{U}{\times }...\underset{U}{\times }\widehat{U_{i_{k}}}\underset{U}{\times }...\underset{U}{\times }U_{i_{n+1}}\right) \\
\end{diagram}
%EndExpansion
\\
&&\left. 
%TCIMACRO{%
%\TeXButton{Arrow ro NEW}{\begin{diagram}[size=3.0em,textflow]
% & \rTo^{\rho _{i_{0},...,\widehat{i_{k}},...,i_{n+1}}} & {\dbigoplus\limits_{i_{0},i_{1},...,i_{n}\in {I}}\mathcal{A}\left( U_{i_{0}}\underset{U}{\times }U_{i_{1}}\underset{U}{\times }...\underset{U}{\times }U_{i_{n}}\right)} \\
%\end{diagram}}}%
%BeginExpansion
\begin{diagram}[size=3.0em,textflow]
 & \rTo^{\rho _{i_{0},...,\widehat{i_{k}},...,i_{n+1}}} & {\dbigoplus\limits_{i_{0},i_{1},...,i_{n}\in {I}}\mathcal{A}\left( U_{i_{0}}\underset{U}{\times }U_{i_{1}}\underset{U}{\times }...\underset{U}{\times }U_{i_{n}}\right)} \\
\end{diagram}%
%EndExpansion
\right] ,
\end{eqnarray*}%
where%
\begin{eqnarray*}
\rho _{i_{0},i_{1},...,i_{n},i_{n+1}} &:&\mathcal{A}\left( U_{i_{0}}\underset%
{U}{\times }U_{i_{1}}\underset{U}{\times }...\underset{U}{\times }%
U_{i_{n+1}}\right) \longrightarrow \left[ \dbigoplus%
\limits_{i_{0},i_{1},...,i_{n+1}\in I}\mathcal{A}\left( U_{i_{0}}\underset{U}%
{\times }U_{i_{1}}\underset{U}{\times }...\underset{U}{\times }%
U_{i_{n+1}}\right) \right] , \\
\rho _{i_{0},...,\widehat{i_{k}},...,i_{n}} &:&\mathcal{A}\left( U_{i_{0}}%
\underset{U}{\times }...\underset{U}{\times }\widehat{U_{i_{k}}}\underset{U}{%
\times }...\underset{U}{\times }U_{i_{n+1}}\right) \longrightarrow \left[
\dbigoplus\limits_{i_{0},i_{1},...,i_{n}\in I}\mathcal{A}\left( U_{i_{0}}%
\underset{U}{\times }U_{i_{1}}\underset{U}{\times }...\underset{U}{\times }%
U_{i_{n}}\right) \right] ,
\end{eqnarray*}%
are the natural embeddings.
\end{definition}

\begin{definition}
\label{Def-Roos-complex}Let $\mathbf{D}$ and $\mathbf{K}$ be categories.
Assume that $\mathbf{D}$ is small and $\mathbf{K}$ is abelian. Let $%
R\subseteq h_{U}$ be a sieve on $\mathbf{D}$. For a pre(co)sheaf%
\begin{equation*}
\mathcal{A}\in \mathbf{pCS}\left( \mathbf{D},\mathbf{K}\right) \text{
(respectively }\mathcal{B}\in \mathbf{pS}\left( \mathbf{D},\mathbf{K}\right) 
\text{)}
\end{equation*}%
on $\mathbf{D}$ with values in $\mathbf{K}$, define the following \textbf{%
Roos chain complex} $^{Roos}C_{\bullet }$ and the\textbf{\ Roos cochain
complex} $^{Roos}C^{\bullet }$ (see \cite%
{Roos-1961-Sur-les-foncteurs-derives-de-lim-MR0132091} and \cite%
{Noebeling-1962-Uber-die-Derivierten-des-Limes-MR0138666}). Assume that $%
\mathbf{K}$ is complete in the case of a presheaf, and cocomplete in the
case of a precosheaf: 
\begin{eqnarray*}
&&^{Roos}C^{\bullet }\left( R\mathbf{,~}\mathcal{B}\right) 
%TCIMACRO{\TeXButton{assigned}{{:=}}}%
%BeginExpansion
{:=}%
%EndExpansion
\left( ^{Roos}C^{n}\left( R\mathbf{,~}\mathcal{B}\right) ,d^{n}\right)
_{n\geq 0}, \\
&&^{Roos}C_{\bullet }\left( R\mathbf{,~}\mathcal{A}\right) 
%TCIMACRO{\TeXButton{assigned}{{:=}}}%
%BeginExpansion
{:=}%
%EndExpansion
\left( ^{Roos}C_{n}\left( R\mathbf{,~}\mathcal{A}\right) ,d_{n}\right)
_{n\geq 0},
\end{eqnarray*}%
where%
\begin{equation*}
\left\langle i_{0},i_{1},...,i_{n}\right\rangle 
%TCIMACRO{\TeXButton{assigned}{{:=}}}%
%BeginExpansion
{:=}%
%EndExpansion
\left[ U_{0}\overset{i_{0}}{\longrightarrow }U_{1}\overset{i_{1}}{%
\longrightarrow }...\overset{i_{n-1}}{\longrightarrow }U_{n}\overset{i_{n}}{%
\longrightarrow }U\right] \in \mathbf{C}_{R},
\end{equation*}%
\begin{eqnarray*}
^{Roos}C^{n}\left( R\mathbf{,~}\mathcal{B}\right) 
&=&\dprod\limits_{\left\langle i_{0},i_{1},...,i_{n}\right\rangle \in 
\mathbf{C}_{R}}\mathcal{B}\left( U_{0}\right) , \\
^{Roos}C_{n}\left( R\mathbf{,~}\mathcal{A}\right) 
&=&\dbigoplus\limits_{\left\langle i_{0},i_{1},...,i_{n}\right\rangle \in 
\mathbf{C}_{R}}\mathcal{A}\left( U_{0}\right) , \\
d^{n} &=&\dsum\limits_{k=0}^{n+1}\left( -1\right) ^{k}d_{\left( k\right)
}^{n}, \\
d_{n} &=&\dsum\limits_{k=0}^{n+1}\left( -1\right) ^{k}d_{n}^{\left( k\right)
},
\end{eqnarray*}%
$d_{\left( k\right) }^{n}:~^{Roos}C^{n}\rightarrow ~^{Roos}C^{n+1}$ are
defined by the compositions%
\begin{eqnarray*}
&&\left[ \pi _{\left\langle i_{0},i_{1},...,i_{n+1}\right\rangle }\right]
\circ d_{\left( k\right) }^{n}%
%TCIMACRO{\TeXButton{assigned}{{:=}} }%
%BeginExpansion
{:=}
%EndExpansion
\\
&&\left[ 
%TCIMACRO{%
%\TeXButton{Arrow pi 1}{\begin{diagram}
%\left( \dprod\limits_{\left\langle i_{0},i_{1},...,i_{n}\right\rangle }\mathcal{B}\left( U_{0}\right) \right) & \rTo^{\pi _{\left\langle i_{0},...,i_{k}\circ i_{k-1},...,i_{n+1}\right\rangle }} & \mathcal{B}\left( U_{0}\right)
%\end{diagram}}}%
%BeginExpansion
\begin{diagram}
\left( \dprod\limits_{\left\langle i_{0},i_{1},...,i_{n}\right\rangle }\mathcal{B}\left( U_{0}\right) \right) & \rTo^{\pi _{\left\langle i_{0},...,i_{k}\circ i_{k-1},...,i_{n+1}\right\rangle }} & \mathcal{B}\left( U_{0}\right)
\end{diagram}%
%EndExpansion
\right] ,
\end{eqnarray*}%
if $k\neq 0$,%
\begin{eqnarray*}
&&\left[ \pi _{\left\langle i_{0},i_{1},...,i_{n+1}\right\rangle }\right]
\circ d_{\left( 0\right) }^{n}%
%TCIMACRO{\TeXButton{assigned}{{:=}} }%
%BeginExpansion
{:=}
%EndExpansion
\\
&&\left[ 
%TCIMACRO{%
%\TeXButton{Arrow pi 2}{\begin{diagram}
%\left( \dprod\limits_{\left\langle i_{0},i_{1},...,i_{n}\right\rangle }\mathcal{B}\left( U_{0}\right) \right) & \rTo^{\pi _{\left\langle i_{1},i_{2},...,i_{n+1}\right\rangle }} & \mathcal{B}\left( U_{1}\right) & \rTo^{\mathcal{B}\left( i_{0}\right) } & \mathcal{B}\left( U_{0}\right)
%\end{diagram}}}%
%BeginExpansion
\begin{diagram}
\left( \dprod\limits_{\left\langle i_{0},i_{1},...,i_{n}\right\rangle }\mathcal{B}\left( U_{0}\right) \right) & \rTo^{\pi _{\left\langle i_{1},i_{2},...,i_{n+1}\right\rangle }} & \mathcal{B}\left( U_{1}\right) & \rTo^{\mathcal{B}\left( i_{0}\right) } & \mathcal{B}\left( U_{0}\right)
\end{diagram}%
%EndExpansion
\right] ,
\end{eqnarray*}%
if $k=0$, and%
\begin{eqnarray*}
\pi _{\left\langle i_{0},i_{1},...,i_{n+1}\right\rangle } &:&\left[
\dprod\limits_{\left\langle i_{0},i_{1},...,i_{n+1}\right\rangle }\mathcal{B}%
\left( U_{0}\right) \right] \longrightarrow \mathcal{B}\left( U_{0}\right) ,
\\
\pi _{\left\langle i_{0},i_{1},...,i_{n}\right\rangle } &:&\left[
\dprod\limits_{\left\langle i_{0},i_{1},...,i_{n}\right\rangle }\mathcal{B}%
\left( U_{0}\right) \right] \longrightarrow \mathcal{B}\left( U_{0}\right) ,
\end{eqnarray*}%
are the natural projections.

$d_{n}^{\left( k\right) }:~^{Roos}C_{n+1}\rightarrow ~^{Roos}C_{n}$ are
defined dually to $d_{\left( k\right) }^{n}$, by the compositions%
\begin{eqnarray*}
&&d_{n}^{\left( k\right) }\circ \left[ \rho _{\left\langle
i_{0},i_{1},...,i_{n+1}\right\rangle }\right] 
%TCIMACRO{\TeXButton{assigned}{{:=}} }%
%BeginExpansion
{:=}
%EndExpansion
\\
&&\left[ 
%TCIMACRO{%
%\TeXButton{Arrow A}{\begin{diagram}
%\mathcal{A}\left( U_{0}\right) & \rTo^{\rho _{\left\langle i_{0},...,i_{k+1}\circ i_{k},...,i_{n+1}\right\rangle }} & \left( \dbigoplus\limits_{\left\langle i_{0},i_{1},...,i_{n}\right\rangle }\mathcal{A}\left( U_{0}\right) \right)
%\end{diagram}}}%
%BeginExpansion
\begin{diagram}
\mathcal{A}\left( U_{0}\right) & \rTo^{\rho _{\left\langle i_{0},...,i_{k+1}\circ i_{k},...,i_{n+1}\right\rangle }} & \left( \dbigoplus\limits_{\left\langle i_{0},i_{1},...,i_{n}\right\rangle }\mathcal{A}\left( U_{0}\right) \right)
\end{diagram}%
%EndExpansion
\right] ,
\end{eqnarray*}%
if $k\neq 0$,%
\begin{eqnarray*}
&&d_{n}^{\left( k\right) }\circ \left[ \rho _{\left\langle
i_{0},i_{1},...,i_{n+1}\right\rangle }\right] 
%TCIMACRO{\TeXButton{assigned}{{:=}} }%
%BeginExpansion
{:=}
%EndExpansion
\\
&&\left[ 
%TCIMACRO{%
%\TeXButton{Arrow A}{\begin{diagram}
%\mathcal{A}\left( U_{0}\right) & \rTo^{\mathcal{A}\left( i_{0}\right) } & \mathcal{A}\left( U_{1}\right) & \rTo^{\rho _{\left\langle i_{1},i_{2},...,i_{n+1}\right\rangle }} & \left( \dbigoplus\limits_{\left\langle i_{0},i_{1},...,i_{n}\right\rangle }\mathcal{A}\left( U_{0}\right) \right)
%\end{diagram}}}%
%BeginExpansion
\begin{diagram}
\mathcal{A}\left( U_{0}\right) & \rTo^{\mathcal{A}\left( i_{0}\right) } & \mathcal{A}\left( U_{1}\right) & \rTo^{\rho _{\left\langle i_{1},i_{2},...,i_{n+1}\right\rangle }} & \left( \dbigoplus\limits_{\left\langle i_{0},i_{1},...,i_{n}\right\rangle }\mathcal{A}\left( U_{0}\right) \right)
\end{diagram}%
%EndExpansion
\right] ,
\end{eqnarray*}%
if $k=0$, where%
\begin{eqnarray*}
\rho _{\left\langle i_{0},i_{1},...,i_{n+1}\right\rangle } &:&\mathcal{A}%
\left( U_{0}\right) \longrightarrow \left[ \dbigoplus\limits_{\left\langle
i_{0},i_{1},...,i_{n+1}\right\rangle }\mathcal{A}\left( U_{0}\right) \right]
, \\
\rho _{\left\langle i_{0},i_{1},...,i_{n}\right\rangle } &:&\mathcal{A}%
\left( U_{0}\right) \longrightarrow \left[ \dbigoplus\limits_{\left\langle
i_{0},i_{1},...,i_{n}\right\rangle }\mathcal{A}\left( U_{0}\right) \right] ,
\end{eqnarray*}%
are the natural embeddings.
\end{definition}

\begin{definition}
\label{Def-Cech-homology}Let $X=\left( \mathbf{C}_{X}\mathbf{,}Cov\left(
X\right) \right) $ be a small site, $\mathcal{A}$ a precosheaf, and $%
\mathcal{B}$ a presheaf on $X$:%
\begin{eqnarray*}
\mathcal{A} &\in &\mathbf{pCS}\left( X,\mathbf{Pro}\left( k\right) \right) ,
\\
\mathcal{B} &\in &\mathbf{pS}\left( X,\mathbf{Mod}\left( k\right) \right) .
\end{eqnarray*}%
Let also $R$ be a sieve on $X$, and $\left\{ V_{i}\rightarrow V\right\} $ be
a family of base-changeable morphisms in $\mathbf{C}_{X}$.

\begin{enumerate}
\item \label{Def-Cech-homology-H0(R)}%
\begin{eqnarray*}
&&H_{0}\left( R,\mathcal{A}\right) 
%TCIMACRO{\TeXButton{assigned}{{:=}}}%
%BeginExpansion
{:=}%
%EndExpansion
\mathcal{A}\otimes _{\mathbf{Set}^{\mathbf{C}_{X}}}R%
%TCIMACRO{\TeXButton{ISO}{\simeq}}%
%BeginExpansion
\simeq%
%EndExpansion
\underset{\left( V\rightarrow U\right) \in \mathbf{C}_{R}}{\underrightarrow{%
\lim }}\mathcal{A}\left( V\right) , \\
&&H^{0}\left( R,\mathcal{B}\right) 
%TCIMACRO{\TeXButton{assigned}{{:=}}}%
%BeginExpansion
{:=}%
%EndExpansion
Hom_{\mathbf{Set}^{\mathbf{C}_{X}}}\left( R,\mathcal{B}\right) 
%TCIMACRO{\TeXButton{ISO}{\simeq}}%
%BeginExpansion
\simeq%
%EndExpansion
\underset{\left( V\rightarrow U\right) \in \mathbf{C}_{R}}{\underleftarrow{%
\lim }}\mathcal{B}\left( V\right) ,
\end{eqnarray*}%
see Definition \ref{Def-Pairings-functors}(\ref{Def-Pairings-functors-Set-X},%
\ref{Def-Pairings-functors-Hom-Set-X}), Notation \ref{Note-Comma-over-site},
Proposition \ref{Prop-A-ten-Set-CX-R}(\ref{Prop-A-ten-Set-CX-R-lim-CR}) and
Remark \ref{Rem-Hom-Set-CX};

\item \label{Def-Cech-homology-Hn(R)}%
\begin{eqnarray*}
&&H_{n}\left( R,\mathcal{A}\right) 
%TCIMACRO{\TeXButton{assigned}{{:=}}}%
%BeginExpansion
{:=}%
%EndExpansion
H_{n}\left( ^{Roos}\check{C}_{\bullet }\left( R\mathbf{,}\mathcal{A}\right)
\right) , \\
&&H^{n}\left( R,\mathcal{B}\right) 
%TCIMACRO{\TeXButton{assigned}{{:=}}}%
%BeginExpansion
{:=}%
%EndExpansion
H^{n}\left( ^{Roos}\check{C}^{\bullet }\left( R\mathbf{,}\mathcal{B}\right)
\right) ;
\end{eqnarray*}

\item \label{Def-Cech-homology-Hn(Ui)}%
\begin{eqnarray*}
&&H_{n}\left( \left\{ V_{i}\rightarrow V\right\} ,\mathcal{A}\right) 
%TCIMACRO{\TeXButton{assigned}{{:=}}}%
%BeginExpansion
{:=}%
%EndExpansion
H_{n}\check{C}_{\bullet }\left( \left\{ V_{i}\rightarrow V\right\} \mathbf{,}%
\mathcal{A}\right) , \\
&&H^{n}\left( \left\{ V_{i}\rightarrow V\right\} ,\mathcal{B}\right) 
%TCIMACRO{\TeXButton{assigned}{{:=}}}%
%BeginExpansion
{:=}%
%EndExpansion
H^{n}\check{C}^{\bullet }\left( \left\{ V_{i}\rightarrow V\right\} \mathbf{,}%
\mathcal{B}\right) ;
\end{eqnarray*}

\item \label{Def-Cech-homology-Cech-Hn(R)}%
\begin{eqnarray*}
&&~^{Roos}\check{H}_{n}\left( U,\mathcal{A}\right) 
%TCIMACRO{\TeXButton{assigned}{{:=}}}%
%BeginExpansion
{:=}%
%EndExpansion
\underset{R\in Cov\left( U\right) }{\underleftarrow{\lim }}H_{n}\left( R,%
\mathcal{A}\right) , \\
&&~^{Roos}\check{H}^{n}\left( U,\mathcal{B}\right) 
%TCIMACRO{\TeXButton{assigned}{{:=}}}%
%BeginExpansion
{:=}%
%EndExpansion
\underset{R\in Cov\left( U\right) }{\underrightarrow{\lim }}H^{n}\left( R,%
\mathcal{B}\right) ,
\end{eqnarray*}%
see Notation \ref{Not-Hn(sieve)}.

\item \label{Def-Cech-homology-Cech-Hn(Ui)}Assume that the topology on $X$
is generated by a pretopology (Definition \ref{Def-Pretopology}). Then
define:%
\begin{eqnarray*}
&&\check{H}_{n}\left( U,\mathcal{A}\right) 
%TCIMACRO{\TeXButton{assigned}{{:=}}}%
%BeginExpansion
{:=}%
%EndExpansion
\underset{\left\{ U_{i}\rightarrow U\right\} \in Cov\left( U\right) }{%
\underleftarrow{\lim }}H_{n}\left( \left\{ U_{i}\rightarrow U\right\} ,%
\mathcal{A}\right) , \\
&&\check{H}^{n}\left( U,\mathcal{B}\right) 
%TCIMACRO{\TeXButton{assigned}{{:=}}}%
%BeginExpansion
{:=}%
%EndExpansion
\underset{\left\{ U_{i}\rightarrow U\right\} \in Cov\left( U\right) }{%
\underrightarrow{\lim }}H^{n}\left( \left\{ U_{i}\rightarrow U\right\} ,%
\mathcal{B}\right) .
\end{eqnarray*}

\item \label{Def-Cech-homology-Cech-plus} Let $\mathcal{A}_{+}$ and $%
\mathcal{A}_{\#}$ be the following precosheaves:%
\begin{eqnarray*}
&&\mathcal{A}_{+}%
%TCIMACRO{\TeXButton{assigned}{{:=}}}%
%BeginExpansion
{:=}%
%EndExpansion
\left( U\longmapsto \check{H}_{0}\left( U,\mathcal{A}\right) \right) , \\
&&\mathcal{A}_{\#}%
%TCIMACRO{\TeXButton{assigned}{{:=}}}%
%BeginExpansion
{:=}%
%EndExpansion
\mathcal{A}_{++},
\end{eqnarray*}%
and let $\mathcal{B}^{+}$ and $\mathcal{B}^{\#}$ be the following presheaves:%
\begin{eqnarray*}
&&\mathcal{B}^{+}%
%TCIMACRO{\TeXButton{assigned}{{:=}}}%
%BeginExpansion
{:=}%
%EndExpansion
\left( U\longmapsto \check{H}^{0}\left( U,\mathcal{B}\right) \right) , \\
&&\mathcal{B}^{\#}%
%TCIMACRO{\TeXButton{assigned}{{:=}}}%
%BeginExpansion
{:=}%
%EndExpansion
\mathcal{B}^{++}.
\end{eqnarray*}%
There are natural morphisms of functors: 
\begin{eqnarray*}
\lambda _{+} &:&\left( \bullet \right) _{+}\longrightarrow 1_{\mathbf{pCS}%
\left( X,\mathbf{Pro}\left( k\right) \right) }:\lambda _{+}\left( \mathcal{A}%
\right) :\mathcal{A}_{+}\longrightarrow \mathcal{A}, \\
\lambda ^{+} &:&1_{\mathbf{pS}\left( X,\mathbf{Mod}\left( k\right) \right)
}\longrightarrow \left( \bullet \right) ^{+}:\lambda ^{+}\left( \mathcal{B}%
\right) :\mathcal{B}\longrightarrow \mathcal{B}^{+}, \\
\lambda _{++} &:&\left( \bullet \right) _{\#}=\left( \bullet \right)
_{++}\longrightarrow 1_{\mathbf{pCS}\left( X,\mathbf{Pro}\left( k\right)
\right) }:\lambda _{++}\left( \mathcal{A}\right) =\lambda _{+}\left( 
\mathcal{A}\right) \circ \lambda _{+}\left( \mathcal{A}_{+}\right) :\mathcal{%
A}^{++}\longrightarrow \mathcal{A}, \\
\lambda ^{++} &:&1_{\mathbf{pS}\left( X,\mathbf{Mod}\left( k\right) \right)
}\longrightarrow \left( \bullet \right) ^{++}=\left( \bullet \right)
^{\#}:\lambda ^{++}\left( \mathcal{B}\right) =\lambda ^{+}\left( \mathcal{B}%
^{+}\right) \circ \lambda _{+}\left( \mathcal{B}\right) :\mathcal{B}%
\longrightarrow \mathcal{B}^{++}.
\end{eqnarray*}
\end{enumerate}
\end{definition}

\begin{remark}
Compare to Definition \ref{Def-Plus-construction}.
\end{remark}

\begin{proposition}
\label{Rem-Two-Cech-equivalent}\label{Prop-Two-Cech-equivalent}Assume that
the topology on $X$ is generated by a pretopology.

\begin{enumerate}
\item If a sieve $R$ is generated by a cover $\left\{ U_{i}\rightarrow
U\right\} $, then the groups $H_{n}\left( R,\mathcal{A}\right) $, $%
H^{n}\left( R,\mathcal{B}\right) $ from\textbf{\ }Definition \ref%
{Def-Cech-homology}(\ref{Def-Cech-homology-Hn(R)}) are naturally isomorphic
to the groups $H_{n}\left( \left\{ U_{i}\rightarrow U\right\} ,\mathcal{A}%
\right) $, $H^{n}\left( \left\{ U_{i}\rightarrow U\right\} ,\mathcal{B}%
\right) $ from Definition \ref{Def-Cech-homology}(\ref%
{Def-Cech-homology-Hn(Ui)}).

\item The groups $~^{Roos}\check{H}_{n}\left( U,\mathcal{A}\right) $ and $%
~^{Roos}\check{H}^{n}\left( U,\mathcal{B}\right) $ from\textbf{\ }Definition %
\ref{Def-Cech-homology}(\ref{Def-Cech-homology-Cech-Hn(R)}) are naturally
isomorphic to the groups $\check{H}_{n}\left( U,\mathcal{A}\right) $ and $%
\check{H}^{n}\left( U,\mathcal{B}\right) $ from\textbf{\ }Definition \ref%
{Def-Cech-homology}(\ref{Def-Cech-homology-Cech-Hn(Ui)}).
\end{enumerate}
\end{proposition}

\begin{proof}
The reasoning below is similar to \cite[Proposition V.2.3.4 and Exercise
V.2.3.6]{SGA4-2-MR0354653}. Let us prove the statement for the pre\textbf{co}%
sheaf $\mathcal{A}$. The proof for the presheaf $\mathcal{B}$ is similar.
Assume that the sieve $R$ is generated by a family $\left\{ U_{i}\rightarrow
U\right\} $. We construct first natural isomorphisms%
\begin{equation*}
H_{n}\left( R,\mathcal{A}\right) 
%TCIMACRO{\TeXButton{ISO}{\simeq}}%
%BeginExpansion
\simeq%
%EndExpansion
H_{n}\left( \left\{ U_{i}\rightarrow U\right\} ,\mathcal{A}\right) .
\end{equation*}%
Applying $\underleftarrow{\lim }$ will give us the desired natural
isomorphisms%
\begin{equation*}
^{Roos}\check{H}_{n}\left( U,\mathcal{A}\right) =\underset{R\in Cov\left(
U\right) }{\underleftarrow{\lim }}H_{n}\left( R,\mathcal{A}\right) 
%TCIMACRO{\TeXButton{ISO}{\simeq}}%
%BeginExpansion
\simeq%
%EndExpansion
\underset{\left\{ U_{i}\rightarrow U\right\} \in Cov\left( U\right) }{%
\underleftarrow{\lim }}H_{n}\left( \left\{ U_{i}\rightarrow U\right\} ,%
\mathcal{A}\right) =\check{H}_{n}\left( U,\mathcal{A}\right) .
\end{equation*}%
Let $X_{\bullet ,\bullet }$ be the following bicomplex:%
\begin{equation*}
\left( X_{s,t},d,\delta \right) 
%TCIMACRO{\TeXButton{assigned}{{:=}}}%
%BeginExpansion
{:=}%
%EndExpansion
\left( \dbigoplus\limits_{U_{0}\rightarrow U_{1}\rightarrow ...\rightarrow
U_{s}\rightarrow U_{i_{0}}\underset{U}{\times }U_{i_{1}}\underset{U}{\times }%
...\underset{U}{\times }U_{i_{t}}}\mathcal{A}\left( U_{0}\right) ,d,\delta
\right)
\end{equation*}%
where the \textbf{horizontal} differentials $d_{\bullet ,\bullet }$ are like
in Definition \ref{Def-Roos-complex}, while the \textbf{vertical}
differentials $\delta _{\bullet ,\bullet }$ are like in Definition \ref%
{Def-Cech-complex}. Consider the two spectral sequences converging to the
total homology:%
\begin{eqnarray*}
^{ver}E_{s,t}^{2} &=&~^{hor}H_{s}~^{ver}H_{t}\left( X_{\bullet ,\bullet
}\right) \implies H_{s+t}\left( Tot_{\bullet }\left( X\right) \right) , \\
^{hor}E_{s,t}^{2} &=&~^{ver}H_{t}~^{hor}H_{s}\left( X_{\bullet ,\bullet
}\right) \implies H_{s+t}\left( Tot_{\bullet }\left( X\right) \right) .
\end{eqnarray*}%
Since the comma category%
\begin{equation*}
\mathbf{C}_{X}\downarrow \left( U_{i_{0}}\underset{U}{\times }U_{i_{1}}%
\underset{U}{\times }...\underset{U}{\times }U_{i_{t}}\right)
\end{equation*}%
has a terminal object%
\begin{equation*}
%TCIMACRO{%
%\TeXButton{Terminal-object}{\begin{diagram}
%U_{i_{0}}\underset{U}{\times }U_{i_{1}}\underset{U}{\times }...\underset{U}{\times }U_{i_{t}} & \rTo^{1_{U_{i_{0}}\times ...\times U_{i_{t}}}} & U_{i_{0}}\underset{U}{\times }U_{i_{1}}\underset{U}{\times }...\underset{U}{\times }U_{i_{t}},
%\end{diagram}}}%
%BeginExpansion
\begin{diagram}
U_{i_{0}}\underset{U}{\times }U_{i_{1}}\underset{U}{\times }...\underset{U}{\times }U_{i_{t}} & \rTo^{1_{U_{i_{0}}\times ...\times U_{i_{t}}}} & U_{i_{0}}\underset{U}{\times }U_{i_{1}}\underset{U}{\times }...\underset{U}{\times }U_{i_{t}},
\end{diagram}%
%EndExpansion
\end{equation*}%
it follows that%
\begin{equation*}
^{hor}E_{s,t}^{1}=~^{hor}H_{s}\left( X_{\bullet ,t}\right) =\underset{%
\mathbf{C}_{X}~\downarrow ~\left( U_{i_{0}}\underset{U}{\times }U_{i_{1}}%
\underset{U}{\times }...\underset{U}{\times }U_{i_{t}}\right) }{%
\underleftarrow{\lim }^{s}}\mathcal{A}=\left\{ 
\begin{array}{ccc}
\mathcal{A}\left( U_{i_{0}}\underset{U}{\times }U_{i_{1}}\underset{U}{\times 
}...\underset{U}{\times }U_{i_{t}}\right) & \text{if} & s=0, \\ 
0 & \text{if} & s>0,%
\end{array}%
\right.
\end{equation*}%
Therefore%
\begin{equation*}
^{hor}E_{s,t}^{2}=~^{ver}H_{t}~^{hor}H_{s}\left( X_{\bullet ,\bullet
}\right) =\left\{ 
\begin{array}{ccc}
H_{n}\left( \left\{ U_{i}\rightarrow U\right\} ,\mathcal{A}\right) & \text{if%
} & s=0, \\ 
0 & \text{if} & s>0,%
\end{array}%
\right.
\end{equation*}%
the spectral sequence degenerates from $E^{2}$ on, and%
\begin{equation*}
H_{n}\left( Tot_{\bullet }\left( X\right) \right) 
%TCIMACRO{\TeXButton{ISO}{\simeq}}%
%BeginExpansion
\simeq%
%EndExpansion
H_{n}\left( \left\{ U_{i}\rightarrow U\right\} ,\mathcal{A}\right) .
\end{equation*}%
The vertical spectral sequence is as follows:%
\begin{equation*}
^{ver}E_{s,t}^{1}=~^{ver}H_{t}\left( X_{s,\bullet }\right) ,
\end{equation*}%
where $X_{s,\bullet }$ allows the following description:%
\begin{equation*}
X_{s,t}=\dbigoplus\limits_{\substack{ U_{0}\rightarrow U_{1}\rightarrow
...\rightarrow U_{s}\rightarrow U  \\ \varphi \in T\left( U_{s}\rightarrow
U,U_{i_{0}}\underset{U}{\times }U_{i_{1}}\underset{U}{\times }...\underset{U}%
{\times }U_{i_{t}}\right) }}\mathcal{A}\left( U_{0}\right) ,
\end{equation*}%
where%
\begin{equation*}
T\left( U_{s}\rightarrow U,U_{i_{0}}\underset{U}{\times }U_{i_{1}}\underset{U%
}{\times }...\underset{U}{\times }U_{i_{t}}\right) 
%TCIMACRO{\TeXButton{assigned}{{:=}}}%
%BeginExpansion
{:=}%
%EndExpansion
\dcoprod\limits_{i_{0},i_{1},...,i_{t}}Hom_{U}\left( U_{s}\rightarrow
U,U_{i_{0}}\underset{U}{\times }U_{i_{1}}\underset{U}{\times }...\underset{U}%
{\times }U_{i_{t}}\right) ,
\end{equation*}%
and the coproduct (disjoint union) is taken in the category of sets. Denote
temporarily $T\left( U_{s}\rightarrow U,U_{i_{0}}\underset{U}{\times }%
U_{i_{1}}\underset{U}{\times }...\underset{U}{\times }U_{i_{t}}\right) $ by $%
S$. It follows that%
\begin{equation*}
~^{ver}H_{t}\left( X_{s\bullet }\right) =H_{t}\left[ \dbigoplus\limits_{X}%
\mathcal{D}\longleftarrow \dbigoplus\limits_{X\times X}\mathcal{D}%
\longleftarrow ...\longleftarrow \dbigoplus\limits_{X^{n}}\mathcal{D}%
\longleftarrow ...\right] =\left\{ 
\begin{array}{ccc}
\mathcal{D} & \text{if} & t=0~\&~S\mathbf{\neq }\varnothing \\ 
0 & \text{if} & t\neq 0~\&~S\neq \varnothing \\ 
0 & \text{if} & S=\varnothing%
\end{array}%
\right. ,
\end{equation*}%
where%
\begin{equation*}
\mathcal{D}=\dbigoplus\limits_{U_{0}\rightarrow U_{1}\rightarrow
...\rightarrow U_{s}\rightarrow U}\mathcal{A}\left( U_{0}\right) .
\end{equation*}%
The set $S$ is non-empty iff $\left( U_{s}\rightarrow U\right) \in \mathbf{C}%
_{R}$. Finally,%
\begin{eqnarray*}
^{ver}E_{s,t}^{1} &=&~^{ver}H_{t}\left( X_{s\bullet }\right) =\left\{ 
\begin{array}{ccc}
\dbigoplus\limits_{\left( U_{0}\rightarrow U_{1}\rightarrow ...\rightarrow
U_{s}\rightarrow U\right) \in \mathbf{C}_{R}}\mathcal{A}\left( U_{0}\right)
& \text{if} & t=0 \\ 
0 & \text{if} & t\neq 0%
\end{array}%
\right. , \\
^{ver}E_{s,t}^{2} &=&~^{hor}H_{s}~^{ver}H_{t}\left( X_{s\bullet }\right)
=\left\{ 
\begin{array}{ccc}
H_{s}\left( R,\mathcal{A}\right) & \text{if} & t=0 \\ 
0 & \text{if} & t\neq 0%
\end{array}%
\right. ,
\end{eqnarray*}%
the spectral sequence degenerates from $E^{2}$ on, and%
\begin{equation*}
H_{n}\left( \left\{ U_{i}\rightarrow U\right\} ,\mathcal{A}\right) 
%TCIMACRO{\TeXButton{ISO}{\simeq}}%
%BeginExpansion
\simeq%
%EndExpansion
~^{ver}E_{0,n}^{2}%
%TCIMACRO{\TeXButton{ISO}{\simeq}}%
%BeginExpansion
\simeq%
%EndExpansion
Tot_{n}\left( X\right) 
%TCIMACRO{\TeXButton{ISO}{\simeq}}%
%BeginExpansion
\simeq%
%EndExpansion
~^{hor}E_{n,0}^{2}%
%TCIMACRO{\TeXButton{ISO}{\simeq}}%
%BeginExpansion
\simeq%
%EndExpansion
H_{n}\left( R,\mathcal{A}\right) .
\end{equation*}
\end{proof}

\subsection{Pro-homotopy and pro-homology}

Let $\mathbf{Top}\ $be the category of topological spaces and continuous
mappings. The following categories are closely related to $\mathbf{Top}$:
the category $H\left( \mathbf{Top}\right) $ of homotopy types, the category $%
\mathbf{Pro}\left( H\left( \mathbf{Top}\right) \right) $ of pro-homotopy
types, and the category $H\left( \mathbf{Pro}\left( \mathbf{Top}\right)
\right) $ of homotopy types of pro-spaces. The latter category is used in 
\emph{strong shape theory}. It is finer than the former which is used in 
\emph{shape theory}. The pointed versions $\mathbf{Pro}\left( H\left( 
\mathbf{Top}_{\ast }\right) \right) $ and $H\left( \mathbf{Pro}\left( 
\mathbf{Top}_{\ast }\right) \right) $ are defined similarly.

One of the most important tools in strong shape theory is a \emph{strong
expansion} (see \cite{Mardesic-MR1740831}, conditions (S1) and (S2) on p.
129). In this paper, it is sufficient to use a weaker notion: an $H\left( 
\mathbf{Top}\right) $-\emph{expansion} (\cite[\S I.4.1]%
{Mardesic-Segal-MR676973}, conditions (E1) and (E2)). Those two conditions
are equivalent to the following

\begin{definition}
\label{HTOP-extension}\label{Def-HTOP-extension}Let $X$ be a topological
space. A morphism $X\rightarrow \left( Y_{j}\right) _{j\in \mathbf{I}}$ in $%
\mathbf{Pro}\left( H\left( \mathbf{Top}\right) \right) $ is called an $%
H\left( \mathbf{Top}\right) $-expansion (or simply \textbf{expansion}) if
for any polyhedron $P$ the following mapping%
\begin{equation*}
\underrightarrow{\lim }_{j}\left[ Y_{j},P\right] =\underrightarrow{\lim }%
_{j}Hom_{H\left( \mathbf{Top}\right) }\left( Y_{j},P\right) \longrightarrow
Hom_{H\left( \mathbf{Top}\right) }\left( X,P\right) =\left[ X,P\right]
\end{equation*}%
is bijective where $\left[ Z,P\right] $ is the set of homotopy classes of
continuous mappings from $Z$ to $P$.

An expansion is called \textbf{polyhedral} (or an $H\left( \mathbf{Pol}%
\right) $-expansion) if all $Y_{j}$ are polyhedra.
\end{definition}

\begin{remark}
\label{Rem-HPOL-extension}~

\begin{enumerate}
\item The pointed version of this notion (an $H\left( \mathbf{Pol}_{\ast
}\right) $-expansion) is defined similarly.

\item For any (pointed) topological space $X$ there exists an $H\left( 
\mathbf{Pol}\right) $-expansion (an $H\left( \mathbf{Pol}_{\ast }\right) $%
-expansion), see \cite[Theorem I.4.7 and I.4.10]{Mardesic-Segal-MR676973}.

\item Any two $H\left( \mathbf{Pol}\right) $-expansions ($H\left( \mathbf{Pol%
}_{\ast }\right) $-expansions) of a (pointed) topological space $X$ are
isomorphic in the category $\mathbf{Pro}\left( H\left( \mathbf{Pol}\right)
\right) $ ($\mathbf{Pro}\left( H\left( \mathbf{Pol}_{\ast }\right) \right) $%
), see \cite[Theorem I.2.6]{Mardesic-Segal-MR676973}.
\end{enumerate}
\end{remark}

\begin{definition}
\label{Def-Normal-covering}An open covering is called \textbf{normal} \cite[%
\S I.6.2]{Mardesic-Segal-MR676973}, iff there is a partition of unity
subordinated to it.
\end{definition}

\begin{remark}
Theorem 8 from \cite[App.1, \S 3.2]{Mardesic-Segal-MR676973}, shows that an $%
H\left( \mathbf{Pol}\right) $- or an $H\left( \mathbf{Pol}_{\ast }\right) $%
-expansion for $X$ can be constructed using nerves of normal (see Definition %
\ref{Def-Normal-covering}) open coverings of $X$.
\end{remark}

Pro-homotopy is defined in \cite[p. 121]{Mardesic-Segal-MR676973}:

\begin{definition}
\label{Def-Pro-homotopy-groups}\label{Pro-homotopy-groups}For a (pointed)
topological space $X$, define its pro-homotopy pro-sets%
\begin{equation*}
pro\text{-}\pi _{n}\left( X\right) 
%TCIMACRO{\TeXButton{assigned}{{:=}}}%
%BeginExpansion
{:=}%
%EndExpansion
\left( \pi _{n}\left( Y_{j}\right) \right) _{j\in \mathbf{J}}
\end{equation*}%
where $X\rightarrow \left( Y_{j}\right) _{j\in \mathbf{J}}$ is an $H\left( 
\mathbf{Pol}\right) $-expansion if $n=0$, and an $H\left( \mathbf{Pol}_{\ast
}\right) $-expansion if $n\geq 1$.

Similar to the \textquotedblleft usual\textquotedblright\ algebraic
topology, $pro$-$\pi _{0}$ is a pro-set (an object of $\mathbf{Pro}\left( 
\mathbf{Set}\right) $), $pro$-$\pi _{1}$ is a pro-group (an object of $%
\mathbf{Pro}\left( \mathbf{Gr}\right) $), and $pro$-$\pi _{n}$ are abelian
pro-groups (objects of $\mathbf{Pro}\left( \mathbf{Ab}\right) $) for $n\geq
2 $.
\end{definition}

Pro-homology groups are defined in \cite[\S II.3.2]{Mardesic-Segal-MR676973}%
, as follows:

\begin{definition}
\label{Def-Pro-homology-groups}\label{Pro-homology-groups}For a topological
space $X$, and an abelian group $G$, define its pro-homology groups as%
\begin{equation*}
pro\text{-}H_{n}\left( X,G\right) 
%TCIMACRO{\TeXButton{assigned}{{:=}}}%
%BeginExpansion
{:=}%
%EndExpansion
\left( H_{n}\left( Y_{j},G\right) \right) _{j\in \mathbf{J}}
\end{equation*}%
where $X\rightarrow \left( Y_{j}\right) _{j\in \mathbf{J}}$ is a polyhedral
expansion.
\end{definition}

\bibliographystyle{alpha}
\bibliography{Cosheaves}

\end{document}